\documentclass[a4paper,11pt,reqno,noindent]{amsart}
\usepackage[scale=0.8]{geometry}

\usepackage{amsmath}
\usepackage{amssymb}
\usepackage{amsthm}
\usepackage{dsfont}
\usepackage{amsfonts}
\usepackage{cases}
\usepackage{url}
\theoremstyle{plain}
\newtheorem{theor}{Theorem}[section]
\newtheorem{lem}[theor]{Lemma}
\newtheorem{prop}[theor]{Proposition}
\newtheorem{cor}[theor]{Corollary}

\theoremstyle{definition}

\newtheorem{rem}[theor]{Remark}

\usepackage[applemac]{inputenc}
\usepackage[polutonikogreek,english,frenchb]{babel}
\usepackage[T1]{fontenc}
\usepackage[toc,page]{appendix}
\usepackage{textcomp}
\usepackage{amscd}
\usepackage{graphicx}
\usepackage{enumerate}
\usepackage{cancel}
\usepackage{extpfeil}
\usepackage[all]{xy}
\usepackage{empheq}
\usepackage{esint}
\usepackage{mathtools}
\usepackage{float}
\usepackage{wrapfig}

\newcommand{\N}{\mathbb N}

\newcommand{\R}{\mathbb R}
\newcommand{\Z}{\mathbb Z}

\newcommand{\Pc}{\mathcal P}

\newcommand{\F}{\mathcal F}

\newcommand{\M}{\mathcal M}

\newcommand{\Tr}{\operatorname{Tr}}

\newcommand{\Id}{\operatorname{Id}}
\newcommand{\p}{\mathbb{P}}
\newcommand{\E}{\mathbb{E}}

\newcommand{\loc}{{\operatorname{loc}}}

\newcommand{\Ld}{\operatorname{L}}

\newcommand{\ho}{\textrm{hom}}

\usepackage[toc,page]{appendix}

\newcommand{\cvf}[1]{\mathrel{\mathop{\xrightharpoonup{#1}}}}

\newcommand{\step}[1]{\noindent \textit{Step} #1.}
\numberwithin{equation}{section}
\newcommand{\pushright}[1]{\ifmeasuring@#1\else\omit\hfill$\displaystyle#1$\fi\ignorespaces}

\usepackage{color,wasysym}

\usepackage[colorlinks,citecolor=black,urlcolor=black]{hyperref}

\title[Analyticity of homogenized coefficients under Bernoulli perturbations]{Analyticity of homogenized coefficients under Bernoulli perturbations and the Clausius-Mossotti formulas}
\author[M. Duerinckx]{Mitia Duerinckx}
\author[A. Gloria]{Antoine Gloria}
\date{\today}
\address[Mitia Duerinckx]{Aspirant F.N.R.S. \\ D\'epartement de math\'ematique, Universit\'e Libre de Bruxelles, Belgium \\  and MEPHYSTO team, Inria Lille - Nord Europe, Villeneuve d'Ascq, France}
\email{mduerinc@ulb.ac.be}
\address[Antoine Gloria]{D\'epartement de math\'ematique, Universit\'e Libre de Bruxelles, Belgium \\  and MEPHYSTO team, Inria Lille - Nord Europe, Villeneuve d'Ascq, France}
\email{agloria@ulb.ac.be}

\begin{document}
\selectlanguage{english}
\maketitle

This paper is concerned with the behavior of the homogenized coefficients associated with some random stationary ergodic medium under a Bernoulli perturbation. Introducing a new family of energy estimates that combine probability and physical spaces, we prove the analyticity of the perturbed homogenized coefficients with respect to the Bernoulli parameter. Our approach holds under the minimal assumptions of stationarity and ergodicity, both in the scalar and vector cases, and gives analytical formulas for each derivative that essentially coincide with the so-called cluster expansion used by physicists. In particular, the first term yields the celebrated (electric and elastic) Clausius-Mossotti formulas for isotropic spherical random inclusions in an isotropic reference medium. This work constitutes the first general proof of these formulas in the case of random inclusions.

\tableofcontents

\section{Introduction}

Let $\rho=(q_n)_n$ be a stationary and ergodic random point process in $\R^d$ of intensity unity. To each point $q_n$ we associate a family of independent Bernoulli variables $[0,1]\ni p\mapsto b_n^{(p)}$ that takes value $1$ with probability $p$ and value $0$ with probability $1-p$. Given $\alpha,\beta>0$, we define a family of random matrix fields $A^{(p)}$ on $\R^d$ as follows:
$$
A^{(p)}(x)\,:=\,\alpha \Id + \sum_{n=1}^\infty b_n^{(p)} (\beta-\alpha) \Id \mathds1_{B(q_n)}(x),
$$
where $B(q_n)$ denotes the unit volume ball centered at $q_n$ (above we have assumed that the balls $B(q_n)$ are disjoint for simplicity of the discussion). The random matrix fields $A^{(p)}$ are stationary and ergodic, so that the standard stochastic homogenization theory \cite{PapaVara} yields the existence of associated homogenized matrices $A^{(p)}_{\hom}$. 
For $p$ small, we expect $A^{(p)}_{\hom}$ to be a perturbation of order $p$ of the unperturbed medium $\alpha \Id$.
Indeed, in the late 19th century, Clausius (1879) and Mossotti (1850) independently proposed the following expansion in the scalar case (also known as the Maxwell relation \cite{Maxwell-81}):
\begin{equation}\label{intro:eqCM}
A_{\hom}^{(p)}\,=\,\alpha \Id + \frac{\alpha d(\beta-\alpha)}{\beta+\alpha(d-1)} \Id p+ o(p).
\end{equation}
The rigorous proof of this statement and of its counterpart for linear elasticity has remained a challenge since then.
The first justification of the (electric) Clausius-Mossotti relation is due to Almog in dimension $d=3$, whose results in \cite{Almog-13,Almog-14}, combined with elementary homogenization theory, precisely yield \eqref{intro:eqCM} (the convergence rate obtained in \cite[Theorem~1]{Almog-14} is lost when combined with homogenization). The proof is based on (scalar) potential theory and crucially relies on the facts that $d=3$, that $A^{(p)}$ is everywhere a multiple of the identity, and that the inclusions are supposed to be spherical and disjoint.
Another contribution is due to Mourrat \cite{Mourrat-13}, who considered a discrete scalar elliptic equation instead of a continuum elliptic equation for all $d\ge 2$. In the case treated in \cite{Mourrat-13}, $A^{(p)}$ is a discrete set of i.i.d. conductivities $\beta$ and $\alpha$ with probabilities $p$ and $1-p$. The extension of Mourrat's results to the present continuum setting (which is made possible by the recent contributions \cite{Gloria-Marahrens-14,Gloria-Otto-10b}) would yield the improvement
of \eqref{intro:eqCM} to
\begin{equation}\label{intro:eqCM+}
A_{\hom}^{(p)}\,=\,\alpha \Id + \frac{\alpha d(\beta-\alpha)}{\beta+\alpha(d-1)} \Id p+ O(p^{2-\gamma})
\end{equation}
for any $\gamma>0$. The assumptions of \cite{Mourrat-13} are however very stringent, and typically cover the case of a hardcore random Poisson process (i.e. a Poisson point process where overlapping inclusions are deleted). Indeed, Mourrat's proof crucially relies on \emph{quantitative} estimates on the corrector obtained in a series of works \cite{Gloria-Neukamm-Otto-14,Gloria-Otto-09,Gloria-Otto-09b,Marahrens-Otto-13} by the second author, Marahrens, Neukamm, and Otto. These estimates typically hold under a spectral gap estimate assumption, whence the restriction on the random point set. 
There is a small gap in the proof of (the discrete counterpart of) \eqref{intro:eqCM+} (a quantitative  estimate of $|a_1^\circ(\mu)-a_1^\circ(0)|$, see \cite[(8.2)]{Mourrat-13}, is missing to complete the proof of \cite[Theorem~11.3]{Mourrat-13}), which we think can however be fixed using the quantitative results of \cite{Gloria-Otto-09,Gloria-Otto-09b,Marahrens-Otto-13}.
Mourrat also made the nice observation that the reference medium needs not be the unperturbed background medium (of conductivity $\alpha$), and essentially proved that $p\mapsto A^{(p)}_{\hom}$ is $C^{1,1-\gamma}$ on the whole interval $[0,1]$ for all $\gamma>0$ in the discrete setting.

\medskip

The idea of perturbing a non-uniform reference medium first appeared in the work \cite{Anantharaman-LeBris-11} by Anantharaman and Le Bris, who considered the perturbation of a periodic array of inclusions by i.i.d. Bernoulli variables (the inclusions are independently deleted with probability $p$). The corresponding matrix field $A^{(p)}$ is then a random ergodic field with discrete stationarity. As above, for $p$ small, one expects $A^{(p)}_{\hom}$ to be a perturbation of  $A^{(0)}_{\hom}$ (the homogenized coefficients of the unperturbed \emph{periodic} medium) of order $p$. Anantharaman and Le Bris considered the approximation $A_L^{(p)}$ of $A^{(p)}_{\hom}$ obtained by periodizing the random medium $A^{(p)}$ on a cube of size $L$. The qualitative homogenization theory ensures that almost surely the random approximation $A^{(p)}_L$ of $A^{(p)}_{\hom}$ converges to $A_{\hom}^{(p)}$ (note that $A_L^{(0)}$ is deterministic and coincides with $A_{\hom}^{(0)}$ for all $L\in \N$). Although not formulated this way, they essentially proved that the approximation $p\mapsto \mathbb{E}[A^{(p)}_L]$ is $C^2$ at $p=0$, and obtained bounds on the first two derivatives $\partial_p \mathbb{E}[A^{(p)}_L]|_{p=0}$ and $\partial^2_{p} \mathbb{E}[A^{(p)}_L]|_{p=0}$ that are uniform in $L$. This is however not quite enough to prove that $p\mapsto A^{(p)}_{\hom}$ is itself $C^2$ (or $C^{1}$) at zero. 

\medskip

In the present contribution we shall prove in a very general setting (which includes both the examples studied by Mourrat and by Anantharaman and Le Bris) that the map $p\mapsto A^{(p)}_{\hom}$ is analytic on $[0,1]$ (see Theorem~\ref{th:analytic}).
Our result holds under the mildest statistical assumptions on the reference medium $x\mapsto A^{(0)}(x)$ and on the point process, that is (discrete or continuum) stationarity and ergodicity.
We also make the (crucial) technical assumption that the number of  intersections is uniformly bounded (see however Remark~\ref{rem:asboundint} for the specific example of a Poisson point process). 
We also believe that a suitable adaptation of our arguments may allow to treat the case when the Bernoulli law is replaced by more general laws as considered in \cite[Section~3]{Anantharaman-LeBris-12}.
Although our results are much stronger than those of Mourrat \cite{Mourrat-13}, our proof was mainly inspired by the ingenious computations of Anantharaman and Le Bris (see in particular \cite[Proposition~3.4]{Anantharaman-these}), and only relies on soft arguments. 
In particular the crucial ingredient of our proof is a new family of energy estimates that combine both physical and probability spaces, see Proposition~\ref{prop:apimproved} below. Since the proof only uses ingredients that are available for systems, our results hold not only for scalar equations, but also for uniformly elliptic systems and for linear elasticity. In the case of an isotropic constant background medium perturbed by randomly distributed isotropic spherical inclusions, this proves the celebrated (electric) Clausius-Mossotti formula with an optimal error estimate (see Corollary~\ref{cor:cm}),
\begin{equation}\label{intro:eqCM++}
A_{\hom}^{(p)}\,=\,\alpha \Id + \frac{\alpha d(\beta-\alpha)}{\beta+\alpha(d-1)} \Id p+ O(p^{2}),
\end{equation}
as well as its elastic counterpart (see Corollary~\ref{cor:cm2}), under the weakest assumptions possible. 

Our proof makes use of the standard modification of the corrector equation by a massive term of magnitude $T^{-1}$, and the derivatives of $A_{\hom}^{(p)}$ with respect to $p$ are given by the limits as $T\uparrow \infty$ of the derivatives of a deterministic approximation $A_T^{(p)}$ of $A_{\hom}^{(p)}$ ($\sqrt{T}$ plays a similar role as the period $L$ in \cite{Anantharaman-LeBris-11}). Interestingly, in the case when the inclusions are disjoint, the fact that $p\mapsto A_{\hom}^{(p)}$ is $C^{1,1}$ on $[0,1]$ can be obtained as a corollary of the (classical) energy estimates of \cite{Anantharaman-LeBris-11,Anantharaman-these} (applied to $A_T^{(p)}$ instead of $A_L^{(p)}$), further using that they hold for any $p\in [0,1]$ by Mourrat's observation, see Section~\ref{sec:strategy}. The proof that  $p\mapsto A_{\hom}^{(p)}$ is analytic is however more subtle and is based on the new family of energy estimates.

There are two motivations to go beyond $C^{1,1}$.
First, this gives a definite answer to the maximal regularity of the map $p\mapsto A_{\hom}^{(p)}$.
Analyticity was indeed conjectured in applied mechanics and applied physics, see for instance \cite[Chapters~18 \&~19]{Torquato-02}.
This analyticity result contrasts very much with the corresponding regularity of a similar model originally studied by Maxwell \cite{Maxwell-81} and Rayleigh \cite{Rayleigh-92}, see also \cite[Section~1.7]{JKO94} and \cite{Berdichevski-83}. The latter consists of a homogeneous medium periodically perturbed by inclusions of volume $p$ located at integer points of $\R^d$ (note that the perturbed medium is still $\Z^d$-periodic), in which case the associated map $p\mapsto A_{\hom}^{(p)}$ is $C^{3+\frac{4}{d}}$ in dimension $d$, but not more.
The second motivation stems from numerics. Indeed the main motivation for \cite{Anantharaman-LeBris-11,Anantharaman-LeBris-12} is to exploit the perturbative character of $A^{(p)}_{\hom}$ for $p$ small (seen as a model for defects) and use the first two terms of a Taylor-expansion at zero as a good approximation for $A^{(p)}_{\hom}$, the accuracy of which can be optimally quantified by Theorem~\ref{th:analytic} below. Note that Legoll and Minvielle \cite{Legoll-Minvielle-14} made an original use of this approximation of $A_{\hom}^{(p)}$ in a control variate method to reduce the variance for the approximation of homogenized coefficients. 

As emphasized in \cite{Anantharaman-LeBris-11}, a second natural question is to quantify the convergence speed of these approximations of the derivatives of $A_{\hom}^{(p)}$ (which are indeed computable in practice). As opposed to Theorem~\ref{th:analytic}, which is a \emph{qualitative} result, establishing such a convergence speed requires to quantify the speed of convergence in the ergodic theorem for stochastic homogenization, which is a \emph{quantitative} result. 
We give such a result under the assumption that the speed of convergence of  $A_T$ to $A_\ho$ can be quantified, which typically (but by far not exclusively) holds in the case that the equation is scalar and that $A^{(p)}$ satisfies a spectral gap inequality (as assumed in \cite{Mourrat-13}) using results of \cite{Gloria-Otto-10b} (see Corollary~\ref{cor:rates}).

\bigskip

Our approach to prove Theorem~\ref{th:analytic} is constructive and explicit bounds are obtained on the derivatives.
We do not know whether there is an abstract alternative to prove analyticity.
In \cite{Cohen-Devore-Schwab-11}, Cohen, Devore and Schwab obtained a result of the same flavor using a complexification method: they proved the analyticity of the solution of linear elliptic PDEs with respect to parameters in the coefficients in the framework of a chaos expansion. Their setting is however very much different from the setting of Theorem~\ref{th:analytic}.
Indeed, as emphasized in Remark~\ref{rem:ALB} below, the solution of interest here (the corrector) is not even differentiable with respect to the Bernoulli parameter (the homogenized coefficients are analytic because of subtle cancellations). 

It is also not clear to us how much the Bernoulli law can be relaxed to allow for correlations.
In particular it would be interesting to determine whether the regularity of the homogenized coefficients depends on the decay of correlations of the generalization of the Bernoulli law.

\bigskip

For the clarity of the exposition, although our proof of Theorem~\ref{th:analytic} holds in the case of uniformly elliptic systems and of linear elasticity (provided the elasticity tensor is uniformly very strongly elliptic, as standard in homogenization), for non-symmetric coefficients, and for discrete elliptic equations, we use continuum scalar notation and assume the coefficients are symmetric. For non-symmetric coefficients, it is indeed enough to consider, in addition to the primal corrector equation, the dual corrector equation (associated with the pointwise transpose coefficients), which would only make notation heavier. In addition we assume that the coefficients enjoy continuum stationarity (in the case of $\Z^d$-stationarity, the expectation would simply be replaced everywhere by the expectation of the integral over the unit cube). 
Note that our result also covers the case of laminates, or more generally the case when the heterogeneous coefficients are random in some direction(s) and invariant along the other direction(s) (cf. the example of cylindrical fibers considered in \cite{Anantharaman-LeBris-11} and encountered in practice).

\bigskip

The rest of the article is organized as follows.
In Section~\ref{sec:main} we introduce the main notation and state the main results of the paper: the analyticity of the homogenized coefficients with respect to the Bernoulli parameter and the validity of the Clausius-Mossotti formulas.
In Section~\ref{sec:strategy}, we present the general strategy of the proof under the additional simplifying assumption that the inclusions are disjoint.
Section~\ref{sec:proofs} is dedicated to the introduction and proofs of auxiliary results, and in particular of the improved energy estimates.
The main results are proved in Section~\ref{sec:analytic}.

\section{Main results}\label{sec:main}

\subsection{Assumptions}\label{sec:assumptions}
Let $A$ be a random field.
We choose a point process $\rho=(q_n)_n$, and random bounded inclusions $(J_n)_n$ centered at the points $q_n$. 
To the inclusions $(J_n)_n$ we attach i.i.d. Bernoulli variables $(b^{(p)}_n)_n$ with parameter $p\in[0,1]$, and we perturb $A$ on $J_n$ if  $b_n^{(p)}=1$.
The only assumptions we need here are stationarity and ergodicity, as well as some deterministic bound on the degree of intersections between the inclusions. More precise definitions are given below.

In the rest of this paper, we denote by $B_r(x)$ the ball of radius $r$ centered at $x$ in $\R^d$, and we simply write $B_r:=B_r(0)$ for balls centered at the origin, and set $B:=B_1(0)$. We also denote by $Q$ the unit cube $[-\frac{1}{2},\frac{1}{2})^d$.

\medskip
{\bf Point process.} Let $\rho$ be a (locally finite) ergodic stationary point process on $\R^d$, and choose for convenience a measurable enumeration $\rho=(q_n)_{n=1}^\infty$. For any open set $D$ of $\R^d$, we denote by $\rho(D):=\#\{q_n \in D,n\in \N\}$ the number of points of $\rho$ in $D$.

\medskip
{\bf Inclusions centered at the point process.} Let $R>0$ be fixed. For all $n$, let $J_n^\circ$ be random Borel subsets $J_n^\circ\subset B_R(\subset\R^d)$ (maybe depending on $\rho=(q_n)_n$). This defines random bounded Borel inclusions $J_n:=q_n+J_n^\circ$. We assume that this inclusion process is stationary, in the sense that the random set $\bigcup_nJ_n$ is stationary. Moreover, we further assume that the intersections between the inclusions $J_n$'s are of degree bounded by some deterministic constant $\Gamma\in \N$; by stationarity, this just means
\begin{align}\label{eq:boundrho}
\#\{n\in\N:0\in J_n\}\le\Gamma,\qquad\text{almost surely.}
\end{align}
In physics (see e.g. \cite[Section~3.1]{Torquato-02}), the constant $\Gamma$ in~\eqref{eq:boundrho} is called the impenetrability parameter: different inclusions may penetrate each other but only with the fixed finite maximum degree $\Gamma$. 

As $J_n\subset B_R(q_n)$, assumption~\eqref{eq:boundrho} is trivially satisfied if we assume $\rho(Q)\le\theta_0$ a.s. (thus forbidding arbitrary large clusters in the point process), but it is important to note that the only problem is the possibility of intersections of arbitrary large degree, which has a priori nothing to do with the point process $\rho$ itself. In the case of inclusions with inner radius bounded from below, however, assumption~\eqref{eq:boundrho} is equivalent to an assumption on $\rho$ of the form $\rho(Q)\le\theta_0$ almost surely.

\medskip
{\bf Reference random fields.} Given $0<\lambda\le 1$, denote by $\M_{\lambda}$ the space of uniformly elliptic symmetric $d\times d$-matrices $M$ satisfying $\lambda|\xi|^2\le\xi\cdot M\xi\le |\xi|^2$ for all $\xi\in \R^d$. Let $A,A'$ be two $\M_{\lambda}$-valued ergodic stationary random fields on $\R^d$. Note that $A$ and $A'$ do not need to be independent of the point process $\rho$, and we simply assume that this dependence is local, in the sense that $A(0)$ and $A'(0)$ only depend on $\rho$ via the restriction $\rho|_{B_r}$ for some given deterministic $r>0$.

\medskip
{\bf Bernoulli perturbation of $A$.} For any fixed $p\in[0,1]$, we choose a sequence $(b_n^{(p)})_n$ of i.i.d. Bernoulli random variables with $\p[b_n^{(p)}=1]=p$, independent of all previous random elements. We can now consider the following $p$-perturbed random field, which is a perturbation of the random field $A$ on the inclusions for which $b_n^{(p)}=1$:
\[A^{(p)}=A\mathds1_{\R^d\setminus\bigcup_{n\in E^{(p)}}J_n}+A'\mathds1_{\bigcup_{n\in E^{(p)}}J_n},\]
where we have set $E^{(p)}:=\{n\in\N:b_n^{(p)}=1\}$. 

In the case when the inclusions are disjoint, the $p$-perturbed random field $A^{(p)}$ can be rewritten as follows:
\[A^{(p)}=\sum_n\left(b_n^{(p)}A'+(1-b_n^{(p)})A\right)\mathds1_{J_n}+A\mathds1_{\R^d\setminus\bigcup_nJ_n}.\]
Moreover, in that case, as $A'$ is allowed to depend (locally) on the inclusion process, the following interesting particular example can be considered: choose a sequence $(A_n')_n$ of i.i.d. $\M_{\lambda}$-valued random fields, and define
\begin{align}\label{eq:reformA'A_n'}
A':=\mathrm{Id}\mathds1_{\R^d\setminus\bigcup_nJ_n}+\sum_n A_n'\mathds1_{J_n},
\end{align}
so that $A^{(p)}$ takes the form
\[A^{(p)}=\sum_n\left(b_n^{(p)}A_n'+(1-b_n^{(p)})A\right)\mathds1_{J_n}+A\mathds1_{\R^d\setminus\bigcup_nJ_n}.\]

\medskip
{\bf Probability space and product structure.}
Let us now briefly comment on the underlying probability space. Let $(\Omega_1,\F_1,\p_1)$ be a probability space on which the (stationary) random elements $\rho$, $(J_n)_n$, $A$, and $A'$ are defined. For all $p\in[0,1]$ and $n\in\N$, let $\Omega_{2,n}^{(p)}:=\{b_n^{(p)}\in \{0,1\}\}$, endowed with the trivial $\sigma$-algebra $\F_{2,n}^{(p)}$, and let $\p_{2,n}^{(p)}$ be the Bernoulli measure of parameter $p$ on $\Omega_{2,n}^{(p)}$. The probability space we consider in this article is the product space $(\Omega,\F,\p)$ of $(\Omega_1,\F_1,\p_1)$ and $(\Omega_{2,n}^{(p)},\F_{2,n}^{(p)},\p_{2,n}^{(p)})$ for all $p\in[0,1]$ and $n\in\N$ (with the cylindrical $\sigma$-algebra). With $\E$, $\E_1$, $\E_{2,n}^{(p)}$ the expectations with respect to the measures $\p$, $\p_1$, $\p_{2,n}^{(p)}$, respectively, we have by definition $\E=\E_1\prod_{p\in[0,1]}\prod_{n\in\N}\E_{2,n}^{(p)}$.

The independence of the Bernoulli variables, at the origin of this product structure, then takes the form: for any integrable random variables $\chi_n^{(p)}$ and $\eta_n^{(p)}$ defined on $\Omega_1\times\prod_{m,m\ne n}\Omega_{2,m}^{(p)}$ and $\Omega_{2,n}^{(p)}$, respectively,
\begin{align}\label{prod-struc}
\E[\chi_n^{(p)}\eta_n^{(p)}]=\E[\chi_n^{(p)}]\E[\eta_n^{(p)}].
\end{align}

Note that for all $p\in [0,1]$ the random field $A^{(p)}$ is defined on $\Omega_1\times\prod_{n\in\N}\Omega_{2,n}^{(p)}$ and is stationary and ergodic for the measure $\p_1\otimes\bigotimes_{n\in\N}\p_{2,n}^{(p)}$. As such, it can be viewed as a stationary and ergodic random field on $(\Omega,\F,\p)$.

\medskip
{\bf Typical examples.}
One typical example is that of spherical inclusions $J_n=B_R(q_n)$ centered at the points of any ergodic stationary random point process $\rho$ with minimal distance bounded away from $0$. In that case, $\rho$ can be chosen as the hardcore Poisson point process (that is, a modified Poisson process for which points that are at a distance less than $2R_0$ are deleted; see also the hardcore construction in Step~1 of the proof of Theorem~\ref{th:analytic} in Section~\ref{chap:proofth1}) or the random parking measure (see \cite{Penrose-01}, and see also~\cite{Gloria-Penrose-13} for the ergodicity). Instead of spherical inclusions, we can  consider more general (and random) shapes $J_n=q_n+J_n^\circ$, where the $J_n^\circ$'s are i.i.d. copies of some random Borel set, with $J_n^\circ\subset B_R$ a.s.

Another interesting example is when $\rho$ is a Poisson process (or {\it any} other ergodic stationary point process) and when $J_n$ is the Voronoi cell at $q_n$, intersected with the ball at $q_n$ of radius $R$, say. We could alternatively choose for $J_n$ the largest ball of radius less than $R$ centered at $q_n$ and completely included in the Voronoi cell at $q_n$. In this case, the number of inclusions per unit volume is not necessarily uniformly bounded.

\subsection{Notation}

We start with the definition of homogenized coefficients, correctors, and approximate correctors. We then introduce the crucial notion of difference operators, and conclude with the introduction of a notational system for perturbed coefficients which will turn out to be very convenient when it comes to the (numerous) combinatorial arguments involved in the proofs.
  
\medskip
{\bf Correctors, approximate correctors, and homogenized coefficients.}
For any (possibly infinite) subset $E\subset\N$, we define $A^E:=A+C^E$, where $C^E:=(A'-A)\mathds1_{J^E}$ and $J^E:=\bigcup_{n\in E}J_n$.
In these terms, with $E^{(p)}:=\{n\in\N\,:\,b_n^{(p)}=1\}$, we have $A^{E^{(p)}}\,=\,A+C^{E^{(p)}}$, and we use the short-hand notation $C^{(p)}:=C^{E^{(p)}}$ and $A^{(p)}:=A^{E^{(p)}}$. 

For all $T>0$, we define the {\it approximate correctors} $\phi_{T,\xi}$ and $\phi_{T,\xi}^E$ in direction $\xi$, $|\xi|=1$, associated with any field $A$ and any $A^E$, respectively, as the unique solutions in the space $\{v\in H^1_\loc(\R^d):\sup_z\int_{B(z)}(|v|^2+|\nabla v|^2)<\infty\}$ of the equations
\begin{equation}\label{eq:modif-corr}
\frac1T\phi_{T,\xi}-\nabla\cdot A(\nabla\phi_{T,\xi}+\xi)=0,\qquad\text{and}\qquad\frac1T\phi_{T,\xi}^E-\nabla\cdot A^E(\nabla\phi_{T,\xi}^E+\xi)=0.
\end{equation}
These solutions satisfy the following energy estimates (see \cite[Lemma~2.7]{Gloria-Otto-10b}):
\begin{equation}\label{eq:modif-corr-estim}
\sup_z\fint_{B_{\sqrt{T}}(z)}(T^{-1}|\phi_{T,\xi}|^2+|\nabla \phi_{T,\xi}|^2)\lesssim1,\qquad \sup_z\fint_{B_{\sqrt{T}}(z)}(T^{-1}|\phi_{T,\xi}^E|^2+|\nabla \phi_{T,\xi}^E|^2)\,\lesssim \,1.
\end{equation}
To shorten notation, we write $\phi_{T,\xi}^{(p)}$ for $\phi_{T,\xi}^{E^{(p)}}$. For all $p\in [0,1]$, as the random field $A^{(p)}$ is ergodic and stationary, we have (combine for instance the ergodic theorem with~\cite{PapaVara} in the symmetric case, and~\cite[Theorem~1]{Gloria-12} in the non-symmetric case)
\begin{equation}\label{eq:reg-corr-grad}
\lim_{T\uparrow \infty}\E[|\nabla \phi_{T,\xi}^{(p)}-\nabla \phi_\xi^{(p)}|^2]\,=\,0,
\end{equation}
where $\nabla\phi_\xi^{(p)}$ is the gradient of the corrector, i.e. the gradient of the unique measurable random map $\phi_\xi^{(p)}\in \Ld^2_\loc(\R^d)$ solution of the equation
\[-\nabla\cdot A^{(p)}(\nabla\phi_\xi^{(p)}+\xi)=0\]
on $\R^d$ that satisfies $\phi_\xi^{(p)}(0)=0$ almost surely and such that $\nabla \phi_\xi^{(p)}$ is stationary and has bounded second moment.
(Note that $\phi_{T,\xi}$ exists for any matrix field $A$ if $T>0$, whereas $\phi_\xi$ only exists almost surely for a stationary random field $A$.)
As usual, the {\it homogenized coefficients} are then given by
\[\xi\cdot A_{\hom}^{(p)}\xi=\E\left[(\nabla\phi_\xi^{(p)}+\xi)\cdot A(\nabla\phi_\xi^{(p)}+\xi)\right]=\E\left[\xi\cdot A^{(p)}(\nabla\phi_\xi^{(p)}+\xi)\right].\]
They can be approximated by symmetric or non-symmetric approximate homogenized coefficients:
\begin{align}\label{eq:approxhomcoeff}
\xi\cdot A_{\hom}^{(p)}\xi=\lim_{T\uparrow\infty}\E\left[(\nabla\phi_{T,\xi}^{(p)}+\xi)\cdot A^{(p)}(\nabla\phi_{T,\xi}^{(p)}+\xi)\right]=\lim_{T\uparrow\infty}\E\left[\xi\cdot A^{(p)}(\nabla\phi_{T,\xi}^{(p)}+\xi)\right].
\end{align}
We denote by $\xi\cdot A_T^{(p)}\xi:=\E\left[\xi\cdot A^{(p)}(\nabla\phi_{T,\xi}^{(p)}+\xi)\right]$ the non-symmetric approximate homogenized coefficients. When $\xi$ is fixed, we simply write $\phi_T$ for $\phi_{T,\xi}$, $\phi$ for $\phi_\xi$, etc.

\medskip

{\bf Difference operators.}
The aim of this article is to understand how $A^{(p)}_{\hom}$ depends on $p$
for $p$ close to $0$. We shall first study the easier map $p\mapsto A_T^{(p)}$, seen as a function of the approximate corrector $\phi_T^{(p)}$.
Following physicists we introduce for all $n\in \N$ a difference operator $\delta^{\{n\}}$ acting generically on measurable functions of $(\Omega,\F)$, and in particular on approximate correctors as follows: for all $H\subset \N$,
\[\delta^{\{n\}}\phi_T^H:=\phi_T^{H\cup\{n\}}-\phi_T^H.\]
This operator yields a natural measure of the sensitivity of the corrector $\phi_T^H$ with respect to the perturbation of the medium at inclusion $J_n$. 
This is to be compared to the {\it vertical derivative} used in~\cite{Gloria-Otto-09} in the context of quantitative stochastic homogenization and to the {\it randomized derivatives} introduced by Chatterjee~\cite{Chatterjee-08} in the context of Stein's method (see also the Hoeffding decompositions in~\cite{LRP-15}, where these randomized derivatives are used up to any order). For all finite $F\subset\N$, we further introduce the higher-order difference operator $\delta^F=\prod_{n\in F}\delta^{\{n\}}$; more explicitly, this difference operator $\delta^F$ acting on approximate correctors $\phi_T^H$ (for any $H\subset \N$) is defined as follows:
\begin{equation}\label{eq:def-diff}
\delta^{F}\phi_T^H:=\sum_{l=0}^{|F|}(-1)^{|F|-l}\sum_{G\subset F\atop |G|=l}\phi_T^{G\cup H}=\sum_{G\subset F}(-1)^{|F\setminus G|}\phi_T^{G\cup H},
\end{equation}
with the convention $\delta^\varnothing\phi_T^H=(\phi_T^H)^\varnothing:=\phi_T^H$.
Physicists have introduced such operators to derive {\it cluster expansions} (see~\cite{Torquato-02}), which are used as formal proxies for Taylor expansions with respect to the Bernoulli perturbation: up to order $k$ in the parameter $p$, the cluster expansion for the perturbed corrector reads, for small $p\ge0$,
$$
\phi_T^{(p)}\leadsto  \phi_T+\sum_{n\in E^{(p)}}\delta^{\{n\}}\phi_T+\frac1{2!}\sum_{n_1,n_2\in E^{(p)}\atop\text{distinct}}\delta^{\{n_1,n_2\}}\phi_T+\ldots+\frac1{k!}\sum_{n_1,\ldots,n_k\in E^{(p)}\atop\text{distinct}}\delta^{\{n_1,\ldots,n_k\}}\phi_T,
$$
which we rewrite in the more compact form
\begin{align}\label{eq:clusterphiT}
\phi_T^{(p)}\leadsto \sum_{j=0}^k\sum_{F\subset E^{(p)}\atop|F|=j}\delta^F\phi_T=\sum_{j=0}^k\sum_{F\subset E^{(p)}\atop|F|=j}\sum_{G\subset F}(-1)^{|F\setminus G|}\phi_T^{G},
\end{align}
where $\sum_{|G|=j}$ denotes the sum over $j$-uplets of integers (when $j=0$, this sum reduces to the single term $G=\varnothing$).
Intuitively, $\phi_T^{(p)}$ is expected to be close to a series where terms of order $j$ involve a correction due to the interaction of $j$ inclusions (and therefore derivatives of order $j$). Whereas the cluster formula for $A_{\hom}^{(p)}$ in Corollary~\ref{cor:analytic} below holds under the mildest statistical assumptions on the coefficients, the validity of the expansion~\eqref{eq:clusterphiT} is expected to require strong mixing assumptions, cf.~\cite{Mourrat-13} for the first order. This illustrates again the fact that averaged quantities (e.g. homogenized coefficients) are better behaved than pathwise quantities (e.g. correctors).

For convenience, we also set 
\begin{equation}\label{eq:defdeltaxi0}
\delta^F_\xi\phi_T^H:=\delta^F(\phi_T^H+\xi\cdot x), 
\end{equation}
that is, in terms of gradients,
\begin{align}\label{eq:defdeltaxi}
\nabla\delta_\xi^{F}\phi_T^H=\sum_{l=0}^{|F|}(-1)^{|F|-l}\sum_{G\subset F\atop |G|=l}(\nabla\phi_T^{G\cup H}+\xi)=\sum_{G\subset F}(-1)^{|F\setminus G|}(\nabla\phi_T^{G\cup H}+\xi).
\end{align}
By the binomial formula $\sum_{l=0}^{|F|}\binom{|F|}{l}(-1)^{|F|-l}=0$, we have: $\nabla\delta^F_\xi\phi_T^H=\nabla\delta^F\phi_T^H$ for all $F\ne\varnothing$ and $H\subset \N$,  $\nabla\delta^\varnothing_\xi\phi_T^H=\nabla\phi_T^H+\xi$ for all $H\subset \N$, and, for all finite sets $F,G,H\subset \N$,
\begin{equation}\label{eq:defdeltaxi1}
\nabla \delta_\xi^G\phi_T^{F\cup H} \,=\,\sum_{S\subset F}\nabla \delta_\xi^{S\cup G} \phi_T^H.
\end{equation}

\medskip

{\bf Inclusion-exclusion formula.}
When the inclusions are disjoint, we have
\begin{align}\label{eq:exclDISJ}
C^{(p)}=\sum_{n\in E^{(p)}} C^{\{n\}}.
\end{align}
However, when inclusions may overlap, this formula no longer holds since intersections may be accounted for several times.
In the rest of this subsection we define a suitable system of notation to deal with these intersections.

For any (possibly infinite) subset $E\subset\N$, we set $A_E:=A+C_E$, where
$C_E:=(A'-A)\mathds1_{J_E}$ and $J_E:=\bigcap_{n\in E}J_n$. Note that $J_{\{n\}}=J^{\{n\}}=J_n$, and $C^{\{n\}}=C_{\{n\}}$. 
For non-necessarily disjoint inclusions, $C^{(p)}$ is then given by the following general inclusion-exclusion formula:
\begin{align}\label{eq:excl}
C^{(p)}&=\sum_{n\in E^{(p)}} C_{\{n\}}-\sum_{n_1<n_2\in E^{(p)}}C_{\{n_1,n_2\}}+\sum_{n_1<n_2<n_3\in E^{(p)}}C_{\{n_1,n_2,n_3\}}-\ldots\nonumber\\
&=\sum_{k=1}^\infty(-1)^{k+1}\sum_{F\subset E^{(p)}\atop|F|=k}C_{F}.
\end{align}
Since the inclusions $J_n$'s have a diameter bounded by $2R$ and $\rho(B_{2R})$ is almost surely finite, the sum~\eqref{eq:excl} is locally finite almost surely. Recalling that by assumption~\eqref{eq:boundrho} the degree of the intersections of the inclusions is bounded by $\Gamma$, we deduce that we must have $C_F\equiv0$ for all $|F|>\Gamma$. Therefore, the inclusion-exclusion formula~\eqref{eq:excl} actually reads
\begin{align}\label{eq:excl2}
C^{(p)}=\sum_{k=1}^\Gamma(-1)^{k+1}\sum_{F\subset E^{(p)}\atop |F|=k}C_F.
\end{align}

We shall need further notation in the proofs. For all $E,F\subset\N$, $E\ne\varnothing$, we set $J_{E\| F}:=(\bigcap_{n\in E}J_n)\setminus (\bigcup_{n\in F}J_n)$ and $J^E_{\| F}:=(\bigcup_{n\in E}J_n)\setminus (\bigcup_{n\in F}J_n)$, and then
\[C_{E\|F}:=(A'-A)\mathds1_{J_{E\| F}}, \qquad\text{and}\qquad C^E_{\|F}:=(A'-A)\mathds1_{J^E_{\| F}}.\]
In particular, we have $C_{E\|\varnothing}=C_E$, $C^E_{\|\varnothing}=C^E$, and $C^\varnothing_{\|F}=0$. For simplicity of notation (except in the proof of Lemma~\ref{lem:eqsatif}), we also set $C_{\varnothing\|F}=0=C_\varnothing$.
The inclusion-exclusion formula then yields for all $G,H\subset \N$, $G\ne\varnothing$,
\begin{eqnarray}
C^{H}&=&\sum_{S\subset H}(-1)^{|S|+1}C_S,\label{eq:excl3.1}
\\
C^{H}_{\|G}&=&\sum_{S\subset H}(-1)^{|S|+1}C_{S\|G},\label{eq:excl3.2}
\\
C_{G\|H}&=&\sum_{S\subset H}(-1)^{|S|}C_{S\cup G}\label{eq:excl3.3}.
\end{eqnarray}

\medskip

We shall also use the symbols $\sim$, $\gtrsim$ and $\lesssim$ for $=$, $\ge$, $\le$ up to constants that only depend on $R$, $\Gamma$, $d$, and $\lambda$. Subscripts are used to indicate additional dependence of the constants, e.~g. $\lesssim_{\eta}$ means that the multiplicative constant depends on $\eta$, next to $R$, $\Gamma$, $d$, and $\lambda$. Throughout, we will denote by $C$ any positive constant with $C\sim1$, whose value may vary from line to line.

\subsection{Statements}
Our main result asserts the analyticity of the map $p\mapsto A^{(p)}_{\hom}$ corresponding to the perturbed coefficients.

\begin{theor}[Analyticity of the homogenized coefficients]\label{th:analytic}
Under the assumptions of Subsection~\ref{sec:assumptions}, the map $p\mapsto A^{(p)}_{\hom}$ is analytic on $[0,1]$ and there exists a constant $0<c\le 1$ such that, for all $p_0\in[0,1]$ and all $-p_0\wedge c\le p\le (1-p_0)\wedge c$,
\begin{align}\label{eq:analytic}
A^{(p_0+p)}_{\hom}= A_{\hom}^{(p_0)}+\sum_{j=1}^{\infty}\frac{p^j}{j!} A^{(p_0),j}_{\hom},
\end{align}
where the series converges, and where, for any $j\ge1$, $A^{(p_0),j}_{\hom}$ denotes the (well-defined) $j$-th derivative of the map $p\mapsto A^{(p)}_{\hom}$at $p_0$.\qed
\end{theor}

Since our proof is constructive, we obtain formulas for the derivatives. These formulas involve two approximation arguments: the addition of a massive term $T^{-1}$ in the corrector equation to deal with integrability issues at large distances, and a hardcore approximation of the point process to deal with integrability issues at short distances.

\begin{cor}[Formulas for derivatives]\label{cor:analytic}
Let the assumptions of Subsection~\ref{sec:assumptions} prevail. We can construct a sequence $(\rho_\theta)_\theta$ of hardcore approximations of the stationary point process $\rho$ in the following sense: for any $\theta>0$, $\rho_\theta$ is an ergodic stationary point process on $\R^d$ such that $\rho_\theta\subset\rho$, $\rho_\theta(Q)\le\theta$ a.s., and  $\rho_\theta\uparrow\rho$ locally almost surely as $\theta\uparrow\infty$. For any $F,G\subset\N$, denote by $A_\theta^F,(C_\theta)_{F\|G},(C_\theta)_F$ the coefficients $A^F,C_{F\|G},C_F$ corresponding to $\rho_\theta$ in place of $\rho$, and further denote by $\phi_{T,\theta,\xi}^F$ the approximate corrector $\phi_{T,\xi}^F$ associated with the coefficients corresponding to $\rho_\theta$ in place of $\rho$.

Then, for all $k\ge1$ and all $p_0\in[0,1]$, the $k$-th derivative $A^{(p_0),k}_{\hom}$ at $p_0$ satisfies the following three equivalent formulas, for all $\xi$,
\begin{align}
\xi\cdot A^{(p_0),k}_{\hom}\xi=&~k!\lim_{T\uparrow\infty}\lim_{\theta\uparrow\infty}\sum_{|F|=k}\sum_{G\subsetneq F}(-1)^{|F\setminus G|+1}\E\left[\nabla\delta_\xi^G\phi_{T,\theta,\xi}^{E^{(p_0)}\setminus F}\cdot (C_\theta)_{F\setminus G\| G}(\nabla\phi_{T,\theta,\xi}^{E^{(p_0)}\cup F}+\xi)\right]\label{eq:formderdisj0}
\\
=&~k!\lim_{T\uparrow\infty}\lim_{\theta\uparrow\infty}\sum_{|F|=k}\sum_{G\subset F\atop G\ne\varnothing}(-1)^{|G|+1}\E\left[(\nabla\phi_{T,\theta,\xi}^{E^{(p_0)}\setminus F}+\xi)\cdot (C_\theta)_G\nabla\delta_\xi^{F\setminus G}\phi_{T,\theta,\xi}^{G\cup (E^{(p_0)}\setminus F)}\right]\label{eq:formderdisj1}\\
=&~k!\lim_{T\uparrow\infty}\lim_{\theta\uparrow\infty}\sum_{|F|=k}\E\bigg[\sum_{G\subset F}(-1)^{|F\setminus G|}\xi\cdot A_\theta^{G\cup (E^{(p_0)}\setminus F)}(\nabla\phi_{T,\theta,\xi}^{G\cup (E^{(p_0)}\setminus F)}+\xi)\bigg],\label{eq:formderdisj2}
\end{align}
where the limits exist and where the sums are absolutely convergent for any fixed $T,\theta<\infty$ (recall that $\sum_{|F|=k}$ stands for the sum running over all the $k$-uplets of distinct positive integers).

Moreover, in the case when the point process $\rho$ satisfies $\E[\rho(Q)^s]<\infty$ for all $s\ge1$,
then the limits in $\theta$ as well as all subscripts $\theta$ can be omitted in the above formulas~\eqref{eq:formderdisj0}--\eqref{eq:formderdisj2}. Finally, in the case $k=1$, and under the additional assumption that $\rho(Q)\le \theta_0$ a.s. for some fixed $\theta_0>0$, we can pass to the limit in $T$ inside the sum in~\eqref{eq:formderdisj1}: for all $p_0\in[0,1]$,
\begin{align}\label{eq:der1form}
\xi\cdot A^{(p_0),1}_{\hom}\xi=\sum_{n}\E\left[(\nabla\phi_\xi^{E^{(p_0)}\setminus \{n\}}+\xi)\cdot C^{\{n\}}(\nabla\phi_\xi^{E^{(p_0)}\cup\{n\}}+\xi)\right],
\end{align}
where the sum is still absolutely convergent.\qed
\end{cor}

Formula \eqref{eq:formderdisj2} is the rigorous version of the so-called \emph{cluster expansion formula} formally used by physicists (see \cite{Torquato-02}) as well as in \cite{Anantharaman-LeBris-11}: it compares the homogenized coefficients corresponding to the coefficients obtained with a finite number of perturbed inclusions. In particular, the $k$-th derivative of $A^{(p)}_{\hom}$ with respect to $p$ is obtained by considering $k$ perturbed inclusions. Note that these cluster expansion formulas can be rewritten as in~\eqref{eq:clusterphiT} using the difference operators defined in~\eqref{eq:def-diff}: for all $k\ge1$,
\begin{align*}
\xi\cdot A^{(p_0),k}_{\hom}\xi&=k!\lim_{T\uparrow\infty}\lim_{\theta\uparrow\infty}\sum_{|F|=k}\E\bigg[\delta^{F}\Big(\xi\cdot A_\theta^{E^{(p_0)}\setminus F}(\nabla\phi_{T,\theta,\xi}^{E^{(p_0)}\setminus F}+\xi)\Big)\bigg],
\end{align*}
where $\delta^F$ now acts on the random variable $\xi\cdot A_\theta^{E^{(p_0)}\setminus F}(0)(\nabla\phi_{T,\theta,\xi}^{E^{(p_0)}\setminus F}(0)+\xi)$.
For $k=1$, it essentially coincides with the formula obtained by Mourrat in~\cite{Mourrat-13}:
\[\xi\cdot A^{(p_0),1}_{\hom}\xi=\lim_{T\uparrow\infty}\lim_{\theta\uparrow\infty}\sum_n\E\left[\xi\cdot A_\theta^{E^{(p_0)}\cup\{n\}}(\nabla\phi_{T,\theta,\xi}^{E^{(p_0)}\cup\{n\}}+\xi)-\xi\cdot A_\theta^{E^{(p_0)}\setminus\{n\}}(\nabla\phi_{T,\theta,\xi}^{E^{(p_0)}\setminus\{n\}}+\xi)\right].\]

Also note that, in the particular case when the inclusions $J_n$'s are disjoint, formula~\eqref{eq:formderdisj1} takes the following simpler form: for all $p_0\in[0,1]$ and $k\ge1$,
\begin{align*}
\xi\cdot A_{\hom}^{(p_0),k}\xi\,=\,&~k!\lim_{T\uparrow\infty}\lim_{\theta\uparrow\infty}\sum_{|F|=k}\sum_{n\in F}\E\left[(\nabla\phi_{T,\theta,\xi}^{E^{(p_0)}\setminus F}+\xi)\cdot (C_\theta)^{\{n\}}\nabla\delta^{F\setminus\{n\}}_\xi\phi_{T,\theta,\xi}^{\{n\}\cup (E^{(p_0)}\setminus F)}\right],
\end{align*}
which further reduces, under the additional assumption that $\E[\rho(Q)^s]<\infty$ for all $s\ge1$, to
\begin{align*}
\xi\cdot A_{\hom}^{(p_0),k}\xi\,=\,&~k!\lim_{T\uparrow\infty}\sum_{|F|=k}\sum_{n\in F}\E\left[(\nabla\phi_{T,\xi}^{E^{(p_0)}\setminus F}+\xi)\cdot C^{\{n\}}\nabla\delta^{F\setminus\{n\}}_\xi\phi_{T,\xi}^{\{n\}\cup (E^{(p_0)}\setminus F)}\right].
\end{align*}

\medskip

As a direct consequence of Theorem~\ref{th:analytic} we obtain the following universality principle, well-known by physicists: at first order in the volume fraction of the perturbation, the perturbed homogenized coefficient does not depend on the underlying point process $\rho$. More precisely,
\begin{cor}[First-order universality principle]\label{cor:invpr}
On top of the assumptions of Subsection~\ref{sec:assumptions},  assume that $\E[\rho(Q)^2]<\infty$. Then, we may define the volume fraction of the perturbation by the limit
\begin{align}\label{eq:volfrac}
v_p:=\lim_{L\uparrow\infty}\frac{\E\left[|LQ\cap\bigcup_{n\in E^{(p)}}J_n|\right]}{L^d},
\end{align}
and there exists some matrix $K$ such that for all $p\ge0$,
\[A^{(p)}_{\hom}=A_{\hom}^{(0)}+Kv_p+O(v_p^2).\]
If the point process $\rho$ is independent of $A$, of $A'$ (or else of $(A_n')_n$ in the particular example~\eqref{eq:reformA'A_n'}) and of the random volumes $|J_n^\circ|$'s, then the constant $K$ does not depend on the choice of the underlying point process $\rho$.
\qed
\end{cor}
Since the formulas given by Corollary~\ref{cor:analytic} for the $k$-th derivative $A^{(0),k}_{\hom}$ of $A^{(p)}_{\hom}$ at $0$ involve terms of the form $\E[\sum_{n_1,\ldots,n_k}f(q_{n_1},\ldots,q_{n_k})]$, they depend on moments of $\rho$ up to order $k$, so that stronger dependence on the point process $\rho$ is expected for higher-order terms (see indeed \cite[p.~493--494]{Torquato-02}).

Formula~\eqref{eq:der1form} for the first derivative has the  advantage of being exact (there is no limit left in $T$), and, at $p_0=0$, it is given by the solution of the corrector equation corresponding to a single inclusion. In particular, this makes explicit calculations possible for spherical inclusions, and allows us to    prove the celebrated Clausius-Mossotti formula in a very general context. 

\begin{cor}[Electric Clausius-Mossotti formula]\label{cor:cm}
On top of the assumptions of Subsection~\ref{sec:assumptions}, assume that the inclusions are spherical, i.e. $J_n=B_R(q_n)$, and that both the unperturbed and perturbed coefficients are constant and isotropic: $A=\alpha\Id$ and $A'=\beta\Id$. 
Denoting by $v_p$ the volume fraction~\eqref{eq:volfrac} of the perturbation, we then have, for all $p\ge0$,
\[A^{(p)}_{\hom}=\alpha\Id+\frac{\alpha d(\beta-\alpha)}{\beta+\alpha(d-1)}\Id v_p+O(v_p^2).\]\qed
\end{cor}

As pointed out in the introduction, all our results also hold for linear elasticity.
This allows us to give the first rigorous proof of the elastic Clausius-Mossotti formula for random inclusions. Recall that an isotropic stiffness tensor $A$ has the form $\frac12\xi: A:\xi=G|\xi|^2+\frac\lambda2(\Tr \xi)^2$, where $G$ and $\lambda$ are the Lamé coefficients, to which we associate the bulk modulus $K=\lambda+2G/d$ and shear modulus $G$.

\begin{cor}[Elastic Clausius-Mossotti formula]\label{cor:cm2}
On top of the assumptions of Subsection~\ref{sec:assumptions}, assume that the inclusions are spherical, i.e. $J_n=B_R(q_n)$, and that both the unperturbed and perturbed stiffness tensors $A$ and $A'$ are constant and isotropic, and denote by $K,G>0$ and $K',G'>0$ their respective bulk and shear moduli. 
Let $A^{(1)}$ be the stiffness matrix of an isotropic medium of bulk modulus $K_1=K+(K'-K)\frac{K+\beta}{K'+\beta}$ and shear modulus $G_1=G+(G'-G)\frac{G+\alpha}{G'+\alpha}$, where we have set 
\begin{align}\label{eq:defab}
\alpha=G\frac{d^2K+2(d+1)(d-2)G}{2d(K+2G)},\qquad\beta=2G\frac{d-1}d.
\end{align}
Denoting by $v_p$ the volume fraction~\eqref{eq:volfrac} of the perturbation, we then have, for all $p\ge0$,
$$
A^{(p)}_{\hom}\,=\,A+A^{(1)}v_p+O(v_p^2).
$$
\qed
\end{cor}

Corollaries~\ref{cor:cm} and~\ref{cor:cm2} treat spherical inclusions, in which case the solution of the corrector equation with a single inclusion can be calculated explicitly, so that~\eqref{eq:der1form} can be turned into an explicit formula. In the case of ellipsoidal inclusions, explicit calculations can also be made in terms of the so-called depolarization coefficients in the electric case (see~\cite{Stratton-41}), or in terms of the Eshelby tensor in the elastic case (see~\cite{Eshelby-57}, and also~\cite{Mura-91} for more precise analytic computations), so that an explicit formula for the first derivative $A^{(0),1}_{\hom}$ can also be derived. The comparison of these results for spherical and ellipsoidal inclusions illustrates the fact that the first derivative already heavily depends on the geometry of the microstructure (see e.g. \cite[Section~19.1.2]{Torquato-02}).

An explicit formula could in principle also be obtained for the second derivative at $p_0=0$ for spherical inclusions, since the corrector equation for two disjoint spheres can be solved analytically as well (see~\cite{Ross-68} and~\cite[Section~5]{Jeffrey-73}). 

\medskip

Formulas~\eqref{eq:formderdisj0}, \eqref{eq:formderdisj1} and~\eqref{eq:formderdisj2} for the derivatives as given by Theorem~\ref{th:analytic} are expressed as limits in terms of the approximate corrector gradient. For practical purposes, it may be important to prove rates of convergence for these limits. This is a quantitative ergodic result and therefore requires quantitative assumptions. 
In what follows we assume that a quantitative convergence result is available for the convergence of $A_T^{(p)}:=\E[\xi\cdot A^{(p)}(\nabla \phi_T^{(p)}+\xi)]$ to $A_{\hom}^{(p)}$ (through the convergence of $\nabla \phi_T^{(p)}$ to $\nabla \phi^{(p)}$) and show how this rate is inherited by their derivatives with respect to $p$. 
\begin{cor}\label{cor:rates}
On top of the assumptions of Subsection~\ref{sec:assumptions}, assume that $\E[\rho(Q)^s]<\infty$ for all $s\ge1$, and further assume that there exists a function $\gamma$ such that, for all $T>0$ and $p\in [0,1]$,
\begin{align}\label{eq:quantas}
\E[|\nabla(\phi_T^{(p)}-\phi_{2T}^{(p)})|^2]\lesssim\gamma(T)^2.
\end{align}
Let $p\in[0,1]$ be fixed. Recall the formulas for the approximate derivatives of $A^{(p)}_{\hom}$: for all $k\ge0$,
\begin{align}
\xi\cdot A^{(p),k}_T\xi:=
&~k!\sum_{|F|=k}\sum_{G\subset F}(-1)^{|F\setminus G|}\E\left[\xi\cdot A^{G\cup E^{(p_0)}\setminus F}(\nabla\phi_{T,\xi}^{G\cup E^{(p_0)}\setminus F}+\xi)\right].\label{eq:nonsymapprox}
\end{align}
Then, there is a constant $C\sim1$ such that, for all $k\ge0$, we have
\[\left| A^{(p),k}_{T}- A^{(p),k}_{2T}\right|\le k!C^k\gamma(T)^{2^{-k}}.\]
In particular, if $\gamma(T)\lesssim T^{-\alpha}$ for some $\alpha>0$, this yields for some constant $C\sim_\alpha1$,
\[\left| A^{(p),k}_{T}- A^{(p),k}_{\hom}\right|\le k!C^kT^{-2^{-k}\alpha}.\]
\qed
\end{cor}
It is not clear to us whether Corollary~\ref{cor:rates} is optimal, and symmetric approximations could yield better rates.
Such improvements, which would require nontrivial arguments based on quantitative homogenization theory, are not the goal of this article.

The optimal expected rate $\gamma$ for the approximate corrector gradient in~\eqref{eq:quantas} is as follows:
\begin{align*}
\gamma(T)^2=\begin{cases}
{T^{-1}},&\text{if $d=2$;}\\
{T^{-3/2}},&\text{if $d=3$;}\\
{T^{-2}\log T},&\text{if $d=4$;}\\
{T^{-2}},&\text{if $d>4$.}
\end{cases}
\end{align*}
It holds for instance under the assumption that the coefficients satisfy a spectral gap estimate (see \cite{Gloria-Otto-10b}), and therefore covers the example of the Poisson point process $(q_n)_n$ with inclusions $J_n$ defined as the intersection of the Voronoi cells with the ball of radius $1$ centered at $q_n$,
with fixed conductivity $A'\equiv A_1$ in the inclusions, and $A\equiv A_2$ in the reference medium.
In higher dimensions, extrapolations techniques wrt $T$ are needed (cf. \cite{Gloria-Neukamm-Otto-14}).

\medskip

In the following remark, we shortly discuss the need for assumption~\eqref{eq:boundrho}, that is finiteness of the degree of intersections between the inclusions.

\begin{rem}\label{rem:asboundint}
Assumption~\eqref{eq:boundrho} is crucially used in the proof of the new family of improved energy estimates (Proposition~\ref{prop:apimproved}), which are in turn at the core of the analyticity result.
It is unclear to us how this assumption can be relaxed in general.
In particular, this approach cannot treat the natural example of a Poisson point process with spherical inclusions of fixed radius (that is, $\rho=(q_n)_n$ is a Poisson point process and $J_n:=B(q_n)$, with (say) fixed conductivity $A'\equiv A_1$ in the inclusions and $A\equiv A_2$ in the reference medium).
In this specific example, the random fields $A^{(p)}$'s all satisfy a spectral gap estimate (in the form used in \cite{Gloria-Marahrens-14,Gloria-Otto-10b}) and the point process satisfies $\E[\rho(Q)^s]<\infty$ for all $s\ge1$ (we even have $\E[e^{c\rho(Q)}]<\infty$ for all $c>0$).
The quantitative results of \cite{Gloria-Marahrens-14,Gloria-Otto-10b} then provide
additional analytical tools, as used in the discrete setting by Mourrat~\cite{Mourrat-13}, 
which allow to prove a weaker version of Proposition~\ref{prop:apimproved} and conclude that the map $p\mapsto A^{(p)}_{\hom}$ is (at least) $C^\infty$ on $[0,1]$ with derivatives given by the same analytical formulas as before.
Yet, this is not enough to prove analyticity. In addition this approach makes quantitative assumptions on the random fields themselves (and not only on the point process), which contrasts dramatically with the results of this paper, and yields limitations in terms of applications (it is however enough to have the Clausius-Mossotti formulas in the form of Corollaries~\ref{cor:cm} and~\ref{cor:cm2}). Likewise, a quantitative approach (this time based on recent contributions by Lamacz, Neukamm and Otto \cite{LNO-13}) allows to prove that, in the case of Bernoulli bond percolation on the integer lattice, the homogenized conductivity is (at least) $C^\infty$ below the percolation threshold.
We refer the readers to \cite{MD-thesis} for details on the results for the Poisson point process with sphrerical inclusions and for diffusion on the percolation cluster.
\qed
\end{rem}

\medskip

Before we turn to the proofs, let us emphasize an observation by Anantharaman and Le Bris in \cite{Anantharaman-LeBris-11} on the regularity of the corrector with respect to $p$ in the case of disjoint inclusions --- which dramatically contrasts with the analyticity of $A^{(p)}_{\hom}$.

\begin{rem}\label{rem:ALB}
By testing the equation $-\nabla\cdot A^{(p)}\nabla(\phi^{(p)}-\phi)=\nabla\cdot C^{(p)}(\nabla\phi+\xi)$ in probability, we have
\begin{align*}
\E[|\nabla(\phi^{(p)}-\phi)|^2]&\lesssim \E[|C^{(p)}|^2(1+|\nabla\phi|^2)]\sim p.
\end{align*}
We believe this scaling is optimal, so that the map $[0,1]\to\Ld^2(\Omega):p\mapsto \nabla\phi^{(p)}(0)$ is expected to be nowhere differentiable.
Since we prove that the homogenized coefficients are analytic, this illustrates that averaged quantities behave much better than pointwise quantities (like the corrector gradient). 
\qed
\end{rem}

\section{Strategy of the proof}\label{sec:strategy}

In this section we present the strategy of the proof of Theorem~\ref{th:analytic}. The key ingredient is a new family of energy estimates, the proof of which essentially combines combinatorial and induction arguments.
When inclusions are disjoint, the combinatorics is significantly less involved than in the general case of non-necessarily disjoint inclusions.
In order to focus only on the core of the proof of Theorem~\ref{th:analytic} and to avoid additional combinatorial technicalities in this presentation, we shall momentarily assume that the inclusions are disjoint.

Fix some direction $\xi\in\R^d$, $|\xi|=1$.
The aim of this paper is to investigate the difference
\[\Delta^{(p)}:=\xi\cdot (A^{(p)}_{\hom}-A_{\hom})\xi,\]
and express it as a convergent power series in the variable $p$ around $0$.
Since the approximate correctors behave much better than the correctors themselves, we start with the analysis of 
the approximate difference
\[\Delta_{T}^{(p)}:=\xi\cdot (A^{(p)}_{T}-A_{T})\xi,\]
for fixed $T>0$. 
Indeed, the approximate difference is a good proxy for the difference since  $\lim_T\Delta_T^{(p)}=\Delta^{(p)}$ by~\eqref{eq:approxhomcoeff}. Next we rewrite the approximate difference in a form which is more suitable for the analysis. By definition,
\begin{align}
\Delta_T^{(p)}&=\E[\xi\cdot A^{(p)}(\nabla\phi_T^{(p)}+\xi)]-\E[\xi\cdot A(\nabla\phi_T+\xi)]\nonumber\\
&=\E[\xi\cdot C^{(p)}(\nabla\phi_T^{(p)}+\xi)]+\E[\xi\cdot A\nabla(\phi_T^{(p)}-\phi_T)].\label{eq:baddecomp0}
\end{align}
The first term is already in a nice form (since it is of order $p$), while the second term is not (recall Remark~\ref{rem:ALB}: an energy estimate would only imply that it is of order $\sqrt{p}$). In the following lemma, we make use of the corrector equation to unravel some cancellations.

\begin{lem}\label{lem:firstdecomp}
The approximate difference $\Delta_T^{(p)}$ satisfies
\begin{align}
\Delta_T^{(p)}&=\E[(\nabla\phi_T+\xi)\cdot C^{(p)}(\nabla\phi_T^{(p)}+\xi)]\label{eq:err1wd}.
\end{align}
\end{lem}
\begin{proof}
Using that $A^{(p)}=A+C^{(p)}$, the second term of~\eqref{eq:baddecomp0} turns into
\begin{align*}
&{\E[\xi\cdot A\nabla (\phi_T^{(p)}-\phi_T)]}\\
=~&\E[(\nabla \phi_T^{(p)}+\xi)\cdot A\nabla (\phi_T^{(p)}-\phi_T)]-\E[\nabla \phi_T^{(p)}\cdot A\nabla (\phi_T^{(p)}-\phi_T)]\\
=~&\E[(\nabla \phi_T^{(p)}+\xi)\cdot A^{(p)}\nabla (\phi_T^{(p)}-\phi_T)]-\E[(\nabla \phi_T^{(p)}+\xi)\cdot C^{(p)}\nabla (\phi_T^{(p)}-\phi_T)]\\
&-\E[\nabla \phi_T^{(p)}\cdot A^{(p)}(\nabla \phi_T^{(p)}+\xi)]+\E[\nabla \phi_T^{(p)}\cdot A(\nabla\phi_T+\xi)]+\E[\nabla \phi_T^{(p)}\cdot C^{(p)}(\nabla \phi_T^{(p)}+\xi)].
\end{align*}
By symmetry of the coefficients $A$ and $C^{(p)}$, reorganizing the terms yields
\begin{align}
{\E[\xi\cdot A\nabla (\phi_T^{(p)}-\phi_T)]}&=-\E[\nabla \phi_T\cdot A^{(p)}(\nabla \phi_T^{(p)}+\xi)]+\E[\nabla \phi_T^{(p)}\cdot A(\nabla\phi_T+\xi)]\nonumber\\
&\hspace{1cm}+\E[\nabla \phi_T\cdot C^{(p)}(\nabla \phi_T^{(p)}+\xi)].\label{eq:Ant-2.1}
\end{align}
The sum of the first two terms of the right-hand side of \eqref{eq:Ant-2.1} coincides with the sum of the weak formulations in probability of the equations
\begin{align*}
\frac1T\phi_T^{(p)}-\nabla\cdot A^{(p)}(\nabla\phi_T^{(p)}+\xi)=0\qquad\text{and}\qquad\frac1T\phi_T-\nabla\cdot A(\nabla\phi_T+\xi)=0,
\end{align*}
tested with $\phi_T$ and $\phi^{(p)}_T$ respectively, so that \eqref{eq:Ant-2.1} reduces to
\begin{align*}
\E[\xi\cdot A\nabla (\phi_T^{(p)}-\phi_T)]&=\E[\nabla \phi_T\cdot C^{(p)}(\nabla \phi_T^{(p)}+\xi)],
\end{align*}
and~\eqref{eq:err1wd} then follows from~\eqref{eq:baddecomp0}.
\end{proof}

Assuming that the inclusions are disjoint, we may use the inclusion-exclusion formula~\eqref{eq:excl2} in the elementary form of~\eqref{eq:exclDISJ},
so that~\eqref{eq:err1wd} turns into
\[\Delta_T^{(p)}=\sum_{n}\E\left[(\nabla\phi_T+\xi)\cdot C^{\{n\}}(\nabla\phi_T^{(p)}+\xi)\mathds1_{n\in E^{(p)}}\right],\]
or alternatively, using the constraint $n\in E^{(p)}$ to replace $\phi_T^{(p)}$ by $\phi_T^{E^{(p)}\cup\{n\}}$,
\[\Delta_T^{(p)}=\sum_{n}\E\left[(\nabla\phi_T+\xi)\cdot C^{\{n\}}(\nabla\phi_T^{E^{(p)}\cup\{n\}}+\xi)\mathds1_{n\in E^{(p)}}\right].\]
Note that this sum is absolutely convergent since the $C^{\{n\}}$'s are assumed to have disjoint supports.
As $\mathds1_{n\in E^{(p)}}$ only depends on $b_n^{(p)}$ and $(\nabla\phi_T+\xi)\cdot C^{\{n\}}(\nabla\phi_T^{E^{(p)}\cup \{n\}}+\xi)$ does not depend on $b_n^{(p)}$, we have by independence~\eqref{prod-struc}, using that $b_n^{(p)}$ is a Bernoulli random variable of parameter $p$,
\begin{align*}
\Delta_T^{(p)}=p\sum_{n}\E\left[(\nabla\phi_T+\xi)\cdot C^{\{n\}}(\nabla\phi_T^{E^{(p)}\cup \{n\}}+\xi)\right].
\end{align*}
We further decompose the right-hand side
\begin{align*}
\Delta_T^{(p)}&=p\sum_{n}\E\left[(\nabla\phi_T+\xi)\cdot C^{\{n\}}(\nabla\phi_T^{\{n\}}+\xi)\right]\\
&\qquad+p\sum_{n}\E\left[(\nabla\phi_T+\xi)\cdot C^{\{n\}}\nabla(\phi_T^{E^{(p)}\cup \{n\}}-\phi_T^{\{n\}})\right],
\end{align*}
and observe that the second sum is a difference of the same nature as $\Delta_T^{(p)}$, which begs for an induction argument, and the following lemma is indeed proved by induction (see Lemma~\ref{lem:decompdiffdisj} for a more general statement).
\begin{lem}\label{lem:decompdiffdisj-sim}
Assume that the inclusions $J_n$'s are disjoint and that $\E[\rho(Q)^s]<\infty$ for all $s\ge1$. For all $k\ge0$, all $T>0$, and all $p\in[0,1]$, we have
\begin{align}\label{eq:toprdecnondisj-0}
\Delta_T^{(p)}&=\sum_{j=1}^kp^j\Delta_T^j+p^{k+1}E_T^{(p),k+1}
\end{align}
where, for all $0\le j\le k$, the approximate derivatives $\Delta_T^j$ and the 
error $E_T^{(p),k+1}$ are given by
\begin{align}\label{eq:rewriteSkdisj1-0}
\Delta_T^j&:=~\sum_{|F|=j}\sum_{n\in F}\E\left[\nabla\delta_\xi^{F\setminus\{n\}}\phi_T\cdot C^{\{n\}}(\nabla\phi_T^{F}+\xi)\right],
\\
\label{eq:rewriteSkdisj1err-0}
E_T^{(p),k+1}&:=~\sum_{|F|=k+1}\sum_{n\in F}\E\left[\nabla\delta_\xi^{F\setminus\{n\}}\phi_T\cdot C^{\{n\}}(\nabla\phi_T^{E^{(p)}\cup F}+\xi)\right],
\end{align}
and the sums in~\eqref{eq:rewriteSkdisj1-0} and~\eqref{eq:rewriteSkdisj1err-0} are absolutely convergent.\qed
\end{lem}
Since the combinatorics in the proof of Lemma~\ref{lem:decompdiffdisj}
is not more involved than for the proof of Lemma~\ref{lem:decompdiffdisj-sim}, we
refer the reader to the proof of the former. 

If we can prove that $|E_T^{(p),k}|\le C^k$ for all $k\ge1$ and for some constant $C\sim1$ (independent of $T>0$ and of $p\in[0,1]$), then we can easily pass to the limit $T\uparrow\infty$ in the expansion~\eqref{eq:toprdecnondisj-0} and obtain a convergent power-series expansion for the exact difference $\Delta^{(p)}$ itself around $p=0$.

The following lemma shows that a new family of energy estimates is needed to control the error terms.
We display the proof of this lemma, which is significantly simpler than the corresponding proof in the general case of non-necessarily disjoint inclusions (see Proposition~\ref{prop:apriori}).
\begin{lem}\label{lem:errbound-0}
Assume that the inclusions $J_n$'s are disjoint and that $\E[\rho(Q)^s]<\infty$ for all $s\ge1$. Then, there is a constant $C\sim1$ (independent of $T$, of $p$ and of the moments of $\rho$) such that, for all $k\ge0$, $T>0$, and $p\in [0,1]$, the error $E_T^{(p),k+1}$ defined in Lemma~\ref{lem:decompdiffdisj} satisfies
\begin{align}\label{eq:errbound-0}
|E^{(p),k+1}_T|&\lesssim \sum_{j=0}^{k}\E\bigg[\sum_{|G|=j}\Big|\sum_{|F|=k-j\atop F\cap G=\varnothing}\nabla\delta_\xi^{F\cup G}\phi_T\Big|^2\bigg]+\sum_{j=0}^{k+1}\E\Big[\sum_{|G|=j}|\nabla\delta_\xi^G\phi_T^{(p)}|^2\Big].
\end{align}
\qed
\end{lem}
\begin{proof}
Let $k\ge0$. First rewrite the error as follows:
\begin{equation}\label{eq:ant-2.0}
E^{(p),k+1}_T=\sum_{|F|=k}\sum_{n\notin F}\E\left[\nabla\delta_\xi^F\phi_T\cdot C^{\{n\}}(\nabla\phi_T^{E^{(p)}\cup F\cup\{n\}}+\xi)\right].
\end{equation}
Recalling identity $\sum_{G\subset H}\nabla \delta_\xi^G\phi_T=\nabla\phi_T^H+\xi$ for all $H\subset \N$, we deduce
\[\nabla\phi_T^{E^{(p)}\cup F\cup\{n\}}+\xi=\sum_{G\subset F}\nabla \delta_\xi^G\phi_T^{(p)}+\sum_{G\subset F}\nabla \delta_\xi^{G\cup\{n\}}\phi_T^{(p)},\]
so that~\eqref{eq:ant-2.0} turns into
\[E^{(p),k+1}_T=\sum_{|F|=k}\sum_{G\subset F}\sum_{n\notin F}\E\left[\nabla\delta_\xi^F\phi_T\cdot C^{\{n\}}(\nabla\delta_\xi^G\phi_T^{(p)}+\nabla\delta_\xi^{G\cup\{n\}}\phi_T^{(p)})\right],\]
or equivalently
\begin{align}\label{eq:reerrorinter}
E^{(p),k+1}_T=\sum_{j=0}^k\sum_{|G|=j}\sum_{n\notin G}\sum_{|F|=k-j\atop F\cap(G\cup\{n\})=\varnothing}\E\left[\nabla\delta_\xi^{F\cup G}\phi_T\cdot C^{\{n\}}(\nabla\delta_\xi^G\phi_T^{(p)}+\nabla\delta_\xi^{G\cup\{n\}}\phi_T^{(p)})\right].
\end{align}
For all $n\notin G$ and all maps $f$, we  obviously have (compare with the more general statement~\eqref{eq:claimcombi+})
\[\sum_{|F|=k-j\atop F\cap(G\cup\{n\})=\varnothing}f(F,G,n)=\sum_{|F|=k-j\atop F\cap G=\varnothing}f(F,G,n)-\sum_{|F|=k-j-1 \atop F\cap (G\cup\{n\})=\varnothing}f(F\cup\{n\},G,n),\]
so that we may rearrange the terms in~\eqref{eq:reerrorinter} as follows:
\begin{align*}
|E^{(p),k+1}_T|&\lesssim \sum_{j=0}^k\sum_{|G|=j}\sum_{n\notin G}\E\bigg[\mathds1_{J_n}\Big|\sum_{|F|=k-j\atop F\cap G=\varnothing}\nabla\delta_\xi^{F\cup G}\phi_T\Big|\Big(|\nabla\delta_\xi^G\phi_T^{(p)}|+|\nabla\delta_\xi^{G\cup\{n\}}\phi_T^{(p)}|\Big)\bigg]\\
&\quad+\sum_{j=0}^{k}\sum_{|G|=j}\sum_{n\notin G}\E\bigg[\mathds1_{J_n}\Big|\sum_{|F|=k-j-1\atop F\cap (G\cup\{n\})=\varnothing}\nabla\delta_\xi^{F\cup G\cup\{n\}}\phi_T\Big|\Big(|\nabla\delta_\xi^G\phi_T^{(p)}|+|\nabla\delta_\xi^{G\cup\{n\}}\phi_T^{(p)}|\Big)\bigg].
\end{align*}
By Young's inequality and the fact that the inclusions $J_n$'s are disjoint, this  yields
\begin{align*}
|E^{(p),k+1}_T|&\lesssim \sum_{j=0}^k\sum_{|G|=j}\bigg(\E\bigg[\Big|\sum_{|F|=k-j\atop F\cap G=\varnothing}\nabla\delta_\xi^{F\cup G}\phi_T\Big|^2\bigg]+\E[|\nabla\delta_\xi^G\phi_T^{(p)}|^2]\bigg)\\
&\quad+\sum_{j=0}^{k}\sum_{|G|=j}\sum_{n\notin G}\bigg(\E\bigg[\mathds1_{J_n}\Big|\sum_{|F|=k-j-1\atop F\cap (G\cup\{n\})=\varnothing}\nabla\delta_\xi^{F\cup G\cup\{n\}}\phi_T\Big|^2\bigg]+\E[\mathds1_{J_n}|\nabla\delta_\xi^{G\cup\{n\}}\phi_T^{(p)}|^2]\bigg),
\end{align*}
and the announced result already follows.
\end{proof}

In view of \eqref{eq:errbound-0} it is enough to prove the following family of energy estimates: there exists $C\sim1$ such that for all $k\ge j\ge 0$ we have
\begin{equation}\label{eq:fam-ener}
\E\bigg[\sum_{|G|=j}\Big|\sum_{|F|=k-j\atop F\cap G=\varnothing} \nabla \delta_\xi^{F\cup G} \phi_T^{(p)}\Big|^2\bigg]\,\leq\, C^k.
\end{equation}
On the one hand, a straightforward energy estimate directly yields (cf. Lemma~\ref{lem:apjust1})
\begin{align}\label{eq:pureL2apbis}
\E\bigg[\Big|\sum_{n}\nabla\delta_\xi^{\{n\}}\phi_T^{(p)}\Big|^2\bigg]\lesssim1.
\end{align}
On the other hand, an induction argument yields for some $C\sim 1$ and all $j\ge 0$ (cf. Lemma~\ref{lem:aprioriproba})
\begin{align}\label{eq:pureL2ap0}
\E\bigg[\sum_{|F|=j}|\nabla\delta_\xi^{F}\phi_T^{(p)}|^2\bigg]\leq C^{j}.
\end{align}
For $j\le2$, this estimate already appears in~\cite{Anantharaman-LeBris-11} (with however the massive term approximation replaced by the approximation by periodization).
As mentioned in the introduction, in view of Lemmas~\ref{lem:decompdiffdisj-sim} and~\ref{lem:errbound-0}, these uniform bounds  (combined with the fact that the estimates are independent of $p$ and combined with some invariance argument due to the structure of Bernoulli random variables, see Step~3 of the proof of Theorem~\ref{th:analytic} in Section~\ref{chap:proofth1})  imply that $p\mapsto A_{\hom}^{(p)}$ is $C^{1,1}$ on $[0,1]$.

Before we describe the complete induction strategy used in Section~\ref{chap:apriori} to prove  \eqref{eq:fam-ener}, let us start by showing it in action, proving the result for $k=2$
based on the corresponding result for $k=1$ (that is, \eqref{eq:pureL2apbis} and~\eqref{eq:pureL2ap0} for $j=1$).
This proof is instructive in three respects: it implements the general induction strategy in the first nontrivial step, it shows that we need to use several forms of the equation satisfied by $\delta^{\{n,m\}} \phi_T$, and it suggests that the proof of these equivalent forms relies on combinatorial arguments.
\begin{lem}\label{lem:improvedap2}
Assume that the inclusions $J_n$'s are disjoint and that $\E[\rho(Q)^s]<\infty$ for all $s\ge1$. Then, for all $T>0$ and $p\in[0,1]$,
\begin{eqnarray}
 \E\bigg[\sum_{m\neq n} | \nabla\delta^{\{n,m\}}\phi_T^{(p)}|^2\bigg]&\lesssim&1,\label{eq:strat-ind-3}
\\
 \E\bigg[\sum_{n} \Big| \sum_{m,m\neq n}\nabla\delta^{\{n,m\}}\phi_T^{(p)}\Big|^2\bigg]&\lesssim&1,\label{eq:strat-ind-2}
\\
\E\bigg[\Big|\sum_{n\ne m}\nabla\delta^{\{n,m\}}\phi_T^{(p)}\Big|^2\bigg]&\lesssim&1. \label{eq:strat-ind-1}
\end{eqnarray}
\qed
\end{lem}
\begin{proof}
For notational convenience, we consider $p=0$ only.
In what follows we take for granted that the series we consider are all absolutely converging, which is indeed ensured for fixed $T$ by the
(suboptimal) estimates of Lemma~\ref{lem:finiteT}.
We split the proof into four steps.
In the first step we give three forms of the equation satisfied by $\delta^{\{n,m\}}\phi_T$. In the second step we prove 
\eqref{eq:strat-ind-3} based on one equation and \eqref{eq:pureL2apbis}.
In the third step we prove \eqref{eq:strat-ind-2} based on another equation, 
\eqref{eq:pureL2apbis}, \eqref{eq:pureL2ap0} for $j=1$, and \eqref{eq:strat-ind-3}.
In the last step we prove \eqref{eq:strat-ind-1} based on a third form of the equation, \eqref{eq:pureL2apbis}, \eqref{eq:pureL2ap0} for $j=1$, \eqref{eq:strat-ind-3}, and  \eqref{eq:strat-ind-2}.

\medskip

\step{1} Equations satisfied by $\delta^{\{n,m\}}\phi_T$.

\noindent Let $m\neq n$. The equation satisfied by the difference  $\delta^{\{n,m\}}\phi_T$ can be written in several forms, with perturbed or unperturbed operators. 
With the unperturbed operator, we have
\begin{align*}
&\frac1T\delta^{\{n,m\}}\phi_T-\nabla\cdot A\nabla \delta^{\{n,m\}}\phi_T\\
=~&\nabla\cdot C^{\{n,m\}}(\nabla\phi_T^{\{n,m\}}+\xi)-\nabla\cdot C^{\{n\}}(\nabla\phi_T^{\{n\}}+\xi)-\nabla\cdot C^{\{m\}}(\nabla\phi_T^{\{m\}}+\xi).
\end{align*}
By the inclusion-exclusion formula in the simple form $C^{\{n,m\}}=C^{\{n\}}+C^{\{m\}}$ due to disjointness of the inclusions, the
equation takes the form
\begin{align}\label{eq:eqnsatis00}
\frac1T\delta^{\{n,m\}}\phi_T-\nabla\cdot A\nabla \delta^{\{n,m\}}\phi_T=\nabla\cdot C^{\{n\}}\nabla\delta^{\{m\}}\phi_T^{\{n\}}+\nabla\cdot C^{\{m\}}\nabla\delta^{\{n\}}\phi_T^{\{m\}}.
\end{align}
This equation will be used to prove \eqref{eq:strat-ind-1}.
A combinatorial argument (which is elementary here because the difference operators are of order 2 only and the inclusions are disjoint) allows one to turn the equation satisfied by $\delta^{\{n,m\}}\phi_T$ into 
\begin{equation}
\frac1T\delta^{\{n,m\}}\phi_T-\nabla\cdot A^{\{n\}}\nabla \delta^{\{n,m\}}\phi_T\,=\,\nabla\cdot C^{\{m\}}\nabla\delta^{\{n\}}\phi_T^{\{m\}}+\nabla\cdot C^{\{n\}}\nabla\delta^{\{m\}}\phi_T,\label{eq:eqnsatis00-bis}
\end{equation}
which involves a perturbed operator (with the partially perturbed coefficients $A^{\{n\}}$), and will be used to prove \eqref{eq:strat-ind-2}.
The third and last version of the equation takes the form:
\begin{equation}
\frac1T\delta^{\{n,m\}}\phi_T-\nabla\cdot A^{\{n,m\}}\nabla \delta^{\{n,m\}}\phi_T\,=\, \nabla\cdot C^{\{m\}}\nabla\delta^{\{n\}}\phi_T+\nabla\cdot C^{\{n\}}\nabla\delta^{\{m\}}\phi_T,\label{eq:eqnsatis00-ter}
\end{equation}
which involves the completely perturbed operator, and will be used to prove \eqref{eq:strat-ind-3}.

\medskip

\step{2} Proof of \eqref{eq:strat-ind-3}.

\noindent The starting point is \eqref{eq:eqnsatis00-ter}, the right-hand side of which only involves first-order differences of the unperturbed corrector $\phi_T$.
Although the argument of the expectation in the left-hand side of \eqref{eq:strat-ind-3} is stationary, the equation \eqref{eq:eqnsatis00-ter} is not stationary.
We shall first obtain energy estimates associated with \eqref{eq:eqnsatis00-ter} which are localized in space. It is only after summing these estimates over $n$ and $m$, taking the expectation, and passing to the limit in the localization parameter that the desired estimate \eqref{eq:strat-ind-3} in expectation will come out in the form
\begin{equation}\label{eq:strat-ind-3-1}
\E\bigg[\sum_{n\neq m} |\nabla \delta^{\{n,m\}}\phi_T|^2\bigg] ~\le~ C\,\E\bigg[\sum_{n}|\nabla \delta^{\{n\}}\phi_T|^2\bigg],
\end{equation}
to be combined with \eqref{eq:pureL2apbis}.

For all $N\ge 0$, we then introduce a cut-off function $\chi_N$ for $B_N$ in $B_{2N}$ such that $|\nabla \chi_N|\lesssim 1/N$, and test equation~\eqref{eq:eqnsatis00-ter}
with test function $\chi_N \delta^{\{n,m\}}\phi_T\in H^1(\R^d)$.
This yields for all $n\neq m$ after integration by parts, using the properties of $\chi_N$, and rearranging the terms:
\begin{multline*}
\int_{B_N} |\nabla \delta^{\{n,m\}}\phi_T|^2 \,\leq \,
C\int_{B_{2N}}(\mathds1_{J_n} |\nabla \delta^{\{m\}}\phi_T|+
\mathds1_{J_m} |\nabla \delta^{\{n\}}\phi_T|)|\nabla \delta^{\{n,m\}}\phi_T|
\\
+\frac{C}{N}\int_{B_{2N}}  |\nabla \delta^{\{n,m\}}\phi_T|| \delta^{\{n,m\}}\phi_T|
+\frac{C}{N} \int_{B_{2N}}(\mathds1_{J_n} |\nabla \delta^{\{m\}}\phi_T|+
\mathds1_{J_m} |\nabla \delta^{\{n\}}\phi_T|)| \delta^{\{n,m\}}\phi_T|.
\end{multline*}
We use Young's inequality on each term (to ultimately absorb part of the right-hand side into the left-hand side), sum this inequality over $n,m\in \N$ with $n\neq m$, and take the expectation to 
obtain (for a possibly larger $C$)
\begin{multline*}
\int_{B_N}\E\bigg[\sum_{n\neq m} |\nabla \delta^{\{n,m\}}\phi_T|^2\bigg] \,\leq \,
C\int_{B_{2N}}\E\bigg[\sum_{n\neq m}(\mathds1_{J_n} |\nabla \delta^{\{m\}}\phi_T|^2
+
\mathds1_{J_m} |\nabla \delta^{\{n\}}\phi_T|^2)\bigg]
\\
+\frac{1}{C}\int_{B_{2N}}\E\bigg[\sum_{n\neq m}|\nabla \delta^{\{n,m\}}\phi_T|^2\bigg]
+\frac{C}{N}\int_{B_{2N}}  \E\bigg[\sum_{n\neq m}| \delta^{\{n,m\}}\phi_T|^2\bigg],
\end{multline*}
where all the terms make sense and are finite by Lemma~\ref{lem:finiteT}.
Since all the arguments of the expectations are now stationary, one may get rid of the integrals, which allows one to absorb the second right-hand side term into the left-hand side (choosing $C>0$ big enough) and obtain
\begin{equation*}
\E\bigg[\sum_{n\neq m} |\nabla \delta^{\{n,m\}}\phi_T|^2\bigg] \,\leq \,
C\,\E\bigg[\sum_{n\neq m}\mathds1_{J_n} |\nabla \delta^{\{m\}}\phi_T|^2\bigg]
+\frac{C}{N} \E\bigg[\sum_{n\neq m}| \delta^{\{n,m\}}\phi_T|^2\bigg].
\end{equation*}
We are now in position to conclude: by taking $N\uparrow \infty$ we get rid of the second term of the right-hand side,
so that \eqref{eq:strat-ind-3-1} follows.

\medskip

\step{3} Proof of \eqref{eq:strat-ind-2}.

\noindent The desired estimate is a consequence of \eqref{eq:pureL2apbis}, of \eqref{eq:pureL2ap0} for $j=1$, of \eqref{eq:strat-ind-3}, and of
\begin{multline}\label{eq:boundexpec2c}
\E\bigg[\sum_n\Big|\sum_{m,m\ne n}\nabla \delta^{\{n,m\}}\phi_T\Big|^2\bigg]\\
\lesssim\E\bigg[\sum_{n\neq m}|\nabla\delta^{\{n,m\}}\phi_T|^2\bigg]
+\E\bigg[\Big|\sum_n \nabla \delta^{\{n\}}\phi_T\Big|^2\bigg]
+\E\bigg[\sum_n |\nabla \delta^{\{n\}}\phi_T|^2\bigg].
\end{multline}
The starting point is equation \eqref{eq:eqnsatis00-ter} that we first sum over $m$ for $m\neq n$:
\begin{multline*}
\frac1T\sum_{m,m\neq n}\delta^{\{n,m\}}\phi_T-\nabla\cdot A^{\{n\}}\nabla \sum_{m,m\neq n}\delta^{\{n,m\}}\phi_T\nonumber\\
=\nabla\cdot \sum_{m,m\neq n}C^{\{m\}}\nabla\delta^{\{n\}}\phi_T^{\{m\}}+\nabla\cdot C^{\{n\}}\nabla\sum_{m,m\neq n}\delta^{\{m\}}\phi_T.
\end{multline*}
Following the approach of Step~2, we test this equation in space with $\chi_N \sum_{m,m\neq n}\delta^{\{n,m\}}\phi_T$ and the same cut-off $\chi_N$.
We obtain after summing the estimate over $n$, taking the expectation, and passing to the limit $N\uparrow \infty$,
$$
\E\bigg[\sum_n \Big|\nabla \sum_{m,m\neq n} \delta^{\{n,m\}}\phi_T\Big|^2\bigg] \,\lesssim\,
\E\bigg[\sum_n \Big|\sum_{m,m\neq n} C^{\{m\}}\nabla \delta^{\{n\}}\phi_T^{\{m\}}\Big|^2\bigg]
+
\E\bigg[\sum_n  \Big| C^{\{n\}} \sum_{m,m\neq n}\nabla\delta^{\{m\}}\phi_T\Big|^2\bigg],
$$
and hence, using that $\mathds1_{J_n}\mathds1_{J_m}=0$ for $n\ne m$ by the disjointness of the inclusions,
$$
\E\bigg[\sum_n \Big|\nabla \sum_{m,m\neq n} \delta^{\{n,m\}}\phi_T\Big|^2\bigg] \,\lesssim\,
\E\bigg[\sum_n \sum_{m,m\neq n} \mathds1_{J_{m}}|\nabla \delta^{\{n\}}\phi_T^{\{m\}}|^2\bigg]
+
\E\bigg[\sum_n  \mathds1_{J_{n}} \Big|\nabla \sum_{m,m\neq n}\delta^{\{m\}}\phi_T\Big|^2\bigg].
$$
By the decomposition $\delta^{\{n\}}\phi_T^{\{m\}}=\delta^{\{n\}}\phi_T+\delta^{\{n,m\}}\phi_T$
and the inequality $\sum_m \mathds1_{J_{m}}\le 1$, the first right-hand side term turns into
\begin{equation*}
\E\bigg[\sum_n \sum_{m,m\neq n} \mathds1_{J_{m}}|\nabla \delta^{\{n\}}\phi_T^{\{m\}}|^2\bigg]\,\lesssim\,\E\bigg[\sum_n |\nabla \delta^{\{n\}}\phi_T|^2\bigg]
+
\E\bigg[\sum_{m\neq n}|\nabla \delta^{\{n,m\}}\phi_T|^2\bigg].
\end{equation*}
The desired inequality \eqref{eq:boundexpec2c} then follows from transforming the second right-hand side term as follows: we complete the sum over $m$, use the triangle inequality and 
the inequality $\sum_n \mathds1_{J_{n}}\le 1$, so that
\begin{equation}\label{eq:boundexpec2d}
\E\bigg[\sum_n\mathds1_{J_n}\Big|\sum_{m,m\ne n}\nabla\delta^{\{m\}}\phi_T\Big|^2\bigg]
\,\lesssim\,
\E\bigg[\Big|\sum_m \nabla \delta^{\{m\}}\phi_T\Big|^2\bigg]
+
\E\bigg[\sum_n |\nabla \delta^{\{n\}}\phi_T|^2\bigg].
\end{equation}

\medskip

\step{4} Proof of  \eqref{eq:strat-ind-1}.

\noindent The desired estimate is a consequence of \eqref{eq:pureL2apbis}, of \eqref{eq:pureL2ap0} for $j=1$, of \eqref{eq:strat-ind-2}, and of
\begin{multline}\label{eq:boundexpec2b}
\E\bigg[\Big| \sum_{n\ne m}\nabla\delta^{\{n,m\}}\phi_T\Big|^2\bigg]\\
\lesssim \E\bigg[\sum_n\Big|\sum_{m,m\ne n}\nabla\delta^{\{n,m\}}\phi_T\Big|^2\bigg]+\E\bigg[\Big|\sum_n \nabla \delta^{\{n\}}\phi_T\Big|^2\bigg]+\E\bigg[\sum_n |\nabla \delta^{\{n\}}\phi_T|^2\bigg].
\end{multline}
The starting point is equation \eqref{eq:eqnsatis00}, that we sum over $n\neq m$:
\begin{align*}
\frac1T\sum_{n\ne m}\delta^{\{n,m\}}\phi_T-\nabla\cdot A\nabla \sum_{n\ne m}\delta^{\{n,m\}}\phi_T=2\nabla\cdot \sum_{n\ne m}C^{\{n\}}\nabla\delta^{\{m\}}\phi_T^{\{n\}}.
\end{align*}
Proceeding again as in Step~2, this yields
\begin{align*}
\E\bigg[\Big| \sum_{n\ne m}\nabla\delta^{\{n,m\}}\phi_T\Big|^2\bigg]\lesssim\E\bigg[\Big|\sum_{n\ne m}C^{\{n\}}\nabla\delta^{\{m\}}\phi_T^{\{n\}}\Big|^2\bigg].
\end{align*}
(Note that since each term of the equation is stationary after summation over $n$ and $m$, this coincides with the energy estimate in probability.)
Since the inclusions are disjoint, we are left with
\begin{align}\label{ex:diag}
\E\bigg[\Big| \sum_{n\ne m}\nabla\delta^{\{n,m\}}\phi_T\Big|^2\bigg]\,\lesssim\,\E\bigg[\sum_n\mathds1_{J_n}\Big|\sum_{m,m\ne n}\nabla\delta^{\{m\}}\phi_T^{\{n\}}\Big|^2\bigg].
\end{align}
By the decomposition $\delta^{\{m\}}\phi_T^{\{n\}}=\delta^{\{m\}}\phi_T+\delta^{\{n,m\}}\phi_T$, this turns into
$$
\E\bigg[\Big| \sum_{n\ne m}\nabla\delta^{\{n,m\}}\phi_T\Big|^2\bigg]
\,\lesssim\,
\E\bigg[\sum_n\mathds1_{J_n}\Big|\sum_{m,m\ne n}\nabla\delta^{\{m\}}\phi_T\Big|^2\bigg]
+
\E\bigg[\sum_{n}\Big|\sum_{m,m\neq n}\nabla \delta^{\{n,m\}}\phi_T\Big|^2\bigg],
$$
and the desired inequality \eqref{eq:boundexpec2b} follows from \eqref{eq:boundexpec2d}.
\end{proof}
This lemma implies that $p\mapsto A_{\hom}^{(p)}$ is $C^{2,1}$ on $[0,1]$.
As already mentioned, this lemma illustrates the induction argument we shall use in the proofs of  Lemma~\ref{lem:aprioriproba} and of  Proposition~\ref{prop:apimproved}.

In the proof of Lemma~\ref{lem:aprioriproba}, we shall always consider the equation for $\delta^{F}\phi_T$ with coefficients $A^{F}$, so that the right-hand side will only involve unperturbed correctors, and then sum over $F$ the resulting energy estimate (first localized in space).

The proof of Proposition~\ref{prop:apimproved} is more involved.
Call $P(j,k)$ property~\eqref{eq:fam-ener}. 
We make a first induction on $k$ and then on $j$.
Note that at step $k$ there are $k$ different forms of the equation satisfied by $\delta^F\phi_T$ (for $|F|=k$).
By Lemma~\ref{lem:aprioriproba}, $P(k,k)$ holds for all $k\in \N$.
Then, given $P(k+1,k+1)$ and $P(i,l)$ for all $i\le l\le k$, we shall prove 
$P(k+1-j,k+1)$ iteratively starting with $j=1$.
Indeed, $P(k+1-j,k+1)$ will follow from $P(k+1-j',k+1)$ for $j'< j$ and $P(j',l)$ for all $j'\le l\le k$,
using the form of the equation where the coefficients are $k+1-j$ times perturbed.
The last step $P(0,k+1)$ is similar to Step~4 in the proof above and relies on the equation with the unperturbed operator.

To be more precise, in the case of disjoint inclusions, the family of equations is as follows (see Lemma~\ref{lem:eqsatif} for non-necessarily disjoint inclusions): for all disjoint subsets $F,G,H\subset\N$, with $F,G$ finite, $F\cup G\ne\varnothing$,
\begin{equation*}
\frac1T\delta_\xi^{F\cup G}\phi_T^H-\nabla\cdot A^{F\cup H}\nabla\delta_\xi^{F\cup G}\phi_T^H\,=\,\sum_{n\in F}\nabla\cdot C^{\{n\}}\nabla\delta_\xi^{(F\setminus\{n\})\cup G}\phi_T^H+\sum_{n\in G}\nabla\cdot C^{\{n\}}\nabla\delta_\xi^{F\cup(G\setminus\{n\})}\phi_T^{H\cup\{n\}}.
\end{equation*}

\section{Auxiliary results and improved energy estimates} \label{sec:proofs}

\subsection{Perturbed corrector equations}

We start by making precise the equations satisfied by the map $\delta_\xi^F\phi_T^G$ for disjoint subsets $F,G\subset\N$, which will be used abundantly in this paper.
The proof of this lemma (like many other auxiliary results of this paper) is purely combinatorial.
\begin{lem}\label{lem:eqsatif}
For all disjoint subsets $F,H\subset \N$, with $F$ finite, $F\ne\varnothing$, and for all $T>0$, the map $\delta_\xi^{F}\phi_T^H$ 
defined in \eqref{eq:def-diff} satisfies the following two equations (weakly) in $\R^d$:
\begin{align}\label{eq:satif1}
\frac1T\delta_\xi^F\phi_T^H-\nabla\cdot A^{F\cup H}\nabla\delta_\xi^F\phi_T^H=\sum_{S\subset F}(-1)^{|S|+1}\nabla\cdot C_{S\|H\cup F\setminus S}\nabla\delta_\xi^{F\setminus S}\phi_T^H,
\end{align}
and 
\begin{align}\label{eq:satif2}
\frac1T\delta_\xi^F\phi_T^H-\nabla\cdot A^{H}\nabla\delta_\xi^F\phi_T^H=\sum_{S\subset F}(-1)^{|S|+1}\nabla\cdot C_{S\|H}\nabla\delta_\xi^{F\setminus S}\phi_T^{S\cup H}.
\end{align}
More generally, for all disjoint subsets $F,G,H\subset \N$, with $F,G$ finite, $F\cup G\ne\varnothing$, and for all $T>0$, the map $\delta_\xi^{F\cup G}\phi_T^H$ 
defined in \eqref{eq:def-diff} satisfies the following equation (weakly) in $\R^d$:
\begin{align}\label{eq:satif3}
&\frac1T\delta_\xi^{F\cup G}\phi_T^H-\nabla\cdot A^{F\cup H}\nabla\delta_\xi^{F\cup G}\phi_T^H\\
&\hspace{2cm}=\sum_{S\subset F}\sum_{U\subset G}(-1)^{|S|+|U|+1}\nabla\cdot C_{S\cup U\|H\cup (F\setminus S)}\nabla\delta_\xi^{(F\setminus S)\cup (G\setminus U)}\phi_T^{U\cup H}.\nonumber
\end{align}
\end{lem}

\begin{proof}
Let $T>0$ be fixed. 
Without loss of generality we may assume  that $H=\varnothing$. We first prove~\eqref{eq:satif2}, from which we shall then deduce~\eqref{eq:satif3}.
Equation~\eqref{eq:satif1} is a particular case of~\eqref{eq:satif3} with $G=\varnothing$.

\medskip
\step1 Proof of~\eqref{eq:satif2}.\\
Let $F\subset \N$ be a finite nonempty subset. By definition \eqref{eq:defdeltaxi} of $\delta_\xi^F\phi_T$ and the inclusion-exclusion identity~\eqref{eq:excl3.1}, we have
\begin{align*}
\frac1T\delta_\xi^F\phi_T-\nabla\cdot A\nabla\delta_\xi^F\phi_T&=\sum_{H\subset F}(-1)^{|F\setminus H|}\nabla\cdot C^H(\nabla\phi_T^{ H}+\xi)\\
&=\sum_{H\subset F}\sum_{S\subset H}(-1)^{|F\setminus H|}(-1)^{|S|+1}\nabla\cdot C_{S}(\nabla\phi_T^{H}+\xi)\\
&=\sum_{S\subset F}(-1)^{|S|+1}\nabla\cdot C_{S}\sum_{H\subset F\setminus S}(-1)^{|(F\setminus S)\setminus H|}(\nabla\phi_T^{ H\cup S}+\xi).
\end{align*}
Recognizing the definition of $\delta_\xi^{F\setminus S}\phi_T^S$, this yields
\begin{align*}
\frac1T\delta_\xi^F\phi_T-\nabla\cdot A\nabla\delta_\xi^F\phi_T&=\sum_{S\subset F}(-1)^{|S|+1}\nabla\cdot C_{S}\nabla\delta_\xi^{F\setminus S}\phi_T^{ S},
\end{align*}
and proves the validity of equation~\eqref{eq:satif2}.

\medskip
\step2 A combinatorial identity.\\
For any finite subsets $K,L,M\subset\N$ (with $K$ and $L$ non-empty), we use the following notation:
\[C^K_{L\|M}:=(A'-A)\mathds1_{J^K_{L\|M}},\qquad J^K_{L\|M}=\bigg(\bigcup_{n\in K}J_n\bigg)\cap\bigg(\bigcap_{n\in L}J_n\bigg)\setminus\bigg(\bigcup_{n\in M}J_n\bigg).\]
In this proof, and in this proof only, when $K$ or $L$ is empty, we further set $J^K_{\varnothing\|M}=J^K_{\|M}$ and $J^\varnothing_{L\|M}=0$. We now check the following general, purely combinatorial identity: for any finite disjoint subsets $U,F\subset\N$ and for any $S\subsetneq F$,
\begin{align}\label{eq:firstid+}
(-1)^{|F\setminus S|}C_{(F\setminus S)\cup U\|S}=\sum_{H\subsetneq F\setminus S}(-1)^{|F\setminus (H\cup S)|}C^{F\setminus (H\cup S)}_{U\|H\cup S}.
\end{align}
It is obviously enough to prove this identity for $U=S=\varnothing$.
Setting $G:=F\setminus S$, we need to prove that, for any finite subset $G\subset \N$,
\begin{align}\label{eq:firstids}
(-1)^{|G|}C_{G}=\sum_{H\subset G}(-1)^{|G\setminus H|}C^{G\setminus H}_{\|H}.
\end{align}
Using the inclusion-exclusion identity \eqref{eq:excl3.2} in form of  $C^{G\setminus H}_{\|H}=\sum_{S\subset G\setminus H}(-1)^{|S|+1}C_{S\|H}$, we have
\begin{align*}
\sum_{H\subset G}(-1)^{|G\setminus H|}C^{G\setminus H}_{\|H}&=\sum_{H\subset G}(-1)^{|G\setminus H|}\sum_{S\subset G\setminus H}(-1)^{|S|+1}C_{S\|H}=\sum_{S\subset G}(-1)^{|S|+1}\sum_{H\subset G\setminus S}(-1)^{|G\setminus H|}C_{S\|H}.
\end{align*}
Using then \eqref{eq:excl3.3} in form of $C_{S\|H}=\mathds1_{S\ne\varnothing}\sum_{U\subset H}(-1)^{|U|}C_{S\cup U}$, this turns into
\begin{align*}
\sum_{H\subset G}(-1)^{|G\setminus H|}C^{G\setminus H}_{\|H}&=\sum_{S\subset G\atop S\ne\varnothing}(-1)^{|S|+1}\sum_{H\subset G\setminus S}(-1)^{|G\setminus H|}\sum_{U\subset H}(-1)^{|U|}C_{S\cup U}\\
&=\sum_{S\subset G\atop S\ne\varnothing}(-1)^{|S|+1}\sum_{U\subset G\setminus S}(-1)^{|U|}C_{S\cup U}\sum_{H\subset G\setminus (S\cup U)}(-1)^{|G\setminus (H\cup U)|}.
\end{align*}
Using twice the binomial identity in the form $\sum_{J\subset K}(-1)^{|K\setminus J|}=\mathds1_{K=\varnothing}$, this reduces to
\begin{align*}
\sum_{H\subset G}(-1)^{|G\setminus H|}C^{G\setminus H}_{\|H}&=\sum_{S\subset G\atop S\ne\varnothing}(-1)^{|S|+1}\sum_{U\subset G\setminus S}(-1)^{|U|+|S|}C_{S\cup U}\mathds1_{U=G\setminus S}\\
&=(-1)^{|G|}C_G\sum_{S\subset G\atop S\ne\varnothing}(-1)^{|S|+1}=(-1)^{|G|}C_G,
\end{align*}
and identity~\eqref{eq:firstids} is proven.

\medskip
\step3 Proof of~\eqref{eq:satif3}.\\
Let $F,G\subset \N$ be two fixed disjoint finite subsets, with $F\cup G\ne\varnothing$. Equation~\eqref{eq:satif3} (with $H=\varnothing$) is obviously a direct corollary of~\eqref{eq:satif2} (with $F$ replaced by $F\cup G$ and with $H=\varnothing$) provided we prove the identity
\begin{align}\label{eq:satiffinpr}
- C^F\nabla\delta_\xi^{F\cup G}\phi_T+A_{F,G}=B_{F,G},
\end{align}
where we have defined
\begin{align*}
A_{F,G}:=&\sum_{S\subset F\cup G}(-1)^{|S|+1} C_{S}\nabla\delta_\xi^{(F\cup G)\setminus S}\phi_T^{S}\\
=&
\sum_{S\subset F}\sum_{U\subset G}(-1)^{|S|+|U|+1} C_{S\cup U}\nabla\delta_\xi^{(F\setminus S)\cup (G\setminus U)}\phi_T^{S\cup U},\\
B_{F,G}:=&\sum_{S\subset F}\sum_{U\subset G}(-1)^{|S|+|U|+1} C_{S\cup U\|F\setminus S}\nabla\delta_\xi^{(F\setminus S)\cup(G\setminus U)}\phi_T^U.
\end{align*}
Let us first rewrite $A_{F,G}$ in a more suitable way.
We appeal to the definition~\eqref{eq:defdeltaxi} of $\nabla \delta_\xi^{(F\setminus S)\cup (G\setminus U)}\phi_T^{S\cup U}$, then make the change of variables $H\cup S \leadsto H$ and $U\cup W \leadsto W$, and conclude by using
\eqref{eq:excl3.1}:
\begin{align}
A_{F,G}&=\sum_{S\subset F}\sum_{U\subset G}(-1)^{|S|+|U|+1}C_{S\cup U}\sum_{H\subset F\setminus S}\sum_{W\subset G\setminus U}(-1)^{|G\setminus(U\cup W)|}(-1)^{|F\setminus (H\cup S)|} (\nabla\phi_T^{S\cup U\cup H\cup W}+\xi)\nonumber\\
&=\sum_{H\subset F}\sum_{W\subset G}(-1)^{|G\setminus W|}(-1)^{|F\setminus H|}\sum_{S\subset H}\sum_{U\subset W}(-1)^{|S|+|U|+1}C_{S\cup U} (\nabla\phi_T^{H\cup W}+\xi)\nonumber\\
&\stackrel{\eqref{eq:excl3.1}}{=}\sum_{H\subset F}\sum_{W\subset G}(-1)^{|G\setminus W|+|F\setminus H|}C^{H\cup W} (\nabla\phi_T^{H\cup W}+\xi).\label{eq:redefAfg}
\end{align}
We now treat $B_{F,G}$. The change of variables $F\setminus S\leadsto S$ yields
\begin{align}
B_{F,G}&=\sum_{U\subset G}(-1)^{|U|+1}\sum_{S\subset F}(-1)^{|F\setminus S|}C_{(F\setminus S)\cup U\| S}\nabla\delta_\xi^{ S\cup(G\setminus U)}\phi_T^U\label{eq:b1b2FGdef}\\
&=\underbrace{\sum_{U\subset G}(-1)^{|U|+1}\sum_{S\subsetneq F}(-1)^{|F\setminus S|}C_{(F\setminus S)\cup U\| S}\nabla\delta_\xi^{ S\cup(G\setminus U)}\phi_T^U}_{=:B^1_{F,G}}+\underbrace{\sum_{U\subset G}(-1)^{|U|+1}C_{U\| F}\nabla\delta_\xi^{ F\cup(G\setminus U)}\phi_T^U}_{=:B^2_{F,G}}.\nonumber
\end{align}
We treat both terms $B^1_{F,G}$ and $B^2_{F,G}$ separately.
The combinatorial identity~\eqref{eq:firstid+} and the change of variables $H\cup S\leadsto H$ yield
\begin{align*}
B^1_{F,G}&=\sum_{U\subset G}(-1)^{|U|+1}\sum_{S\subsetneq F} \sum_{H\subsetneq F\setminus S}(-1)^{|F\setminus(H\cup S)|}C^{F\setminus (H\cup S)}_{U\|H\cup S}\nabla\delta_\xi^{ S\cup(G\setminus U)}\phi_T^U\\
&=\sum_{U\subset G}(-1)^{|U|+1}\sum_{H\subsetneq F}(-1)^{|F\setminus H|}C^{F\setminus H}_{U\|H}\sum_{S\subset H}\nabla\delta_\xi^{ S\cup(G\setminus U)}\phi_T^U.
\end{align*}
By the identity \eqref{eq:defdeltaxi1} in the form $\sum_{S\subset H}\nabla\delta_\xi^{ S\cup(G\setminus U)}\phi_T^U=\nabla\delta_\xi^{G\setminus U}\phi_T^{H\cup U}$, this turns into
\begin{align*}
B_{F,G}^1&=\sum_{U\subset G}(-1)^{|U|+1}\sum_{H\subsetneq F}(-1)^{|F\setminus H|}C^{F\setminus H}_{U\|H}\nabla\delta_\xi^{G\setminus U}\phi_T^{H\cup U},
\end{align*}
and thus, by definition~\eqref{eq:def-diff}--\eqref{eq:defdeltaxi0} of $\delta_\xi^{G\setminus U}$ and the change of variables $U\cup W\leadsto W$,
\begin{align*}
B^1_{F,G}&=\sum_{U\subset G}(-1)^{|U|+1}\sum_{H\subsetneq F}(-1)^{|F\setminus H|}\sum_{W\subset G\setminus U}(-1)^{|G\setminus(U\cup W)}C^{F\setminus H}_{U\|H}(\nabla\phi_T^{H\cup U\cup W}+\xi)\\
&=\sum_{H\subsetneq F}(-1)^{|F\setminus H|}\sum_{W\subset G}(-1)^{|G\setminus W|}\sum_{U\subset W}(-1)^{|U|+1}C^{F\setminus H}_{U\|H}(\nabla\phi_T^{H\cup W}+\xi).
\end{align*}
Noting that, by the usual inclusion-exclusion formula,
\begin{align*}
\sum_{U\subset W}(-1)^{|U|+1}C^{F\setminus H}_{U\|H}&=-C^{F\setminus H}_{\|H}+C^{F\setminus H}_{\|H}\sum_{U\subset W\atop U\ne\varnothing}(-1)^{|U|+1}\mathds1_{J_U}\\
&=-C^{F\setminus H}_{\|H}+C^{F\setminus H}_{\|H}\mathds1_{J^W}=-C^{F\setminus H}_{\|H\cup W},
\end{align*}
we conclude that
\begin{align}\label{eq:redefB1fg}
B^1_{F,G}&=-\sum_{H\subset F}(-1)^{|F\setminus H|}\sum_{W\subset G}(-1)^{|G\setminus W|}C^{F\setminus H}_{\|H\cup W}(\nabla\phi_T^{H\cup W}+\xi).
\end{align}

For the second term $B^2_{F,G}$ in~\eqref{eq:b1b2FGdef}, we argue as in Step~1, and obtain
\begin{align}
B^2_{F,G}&=\sum_{H\subset F}(-1)^{|F\setminus H|}\sum_{U\subset G\atop U\ne\varnothing}(-1)^{|U|+1} C_{ U\| F}\sum_{W\subset G\setminus U}(-1)^{|G\setminus(U\cup W)|}(\nabla\phi_T^{U\cup W\cup H}+\xi)\nonumber\\
&=\sum_{H\subset F}(-1)^{|F\setminus H|}\sum_{W\subset G}(-1)^{|G\setminus W|}\sum_{U\subset W\atop U\ne\varnothing}(-1)^{|U|+1} C_{ U\| F}(\nabla\phi_T^{W\cup H}+\xi)\nonumber\\
&=\sum_{H\subset F}\sum_{W\subset G}(-1)^{|F\setminus H|+|G\setminus W|}C^W_{\| F}(\nabla\phi_T^{W\cup H}+\xi).\label{eq:redefB2fg}
\end{align}
Combining~\eqref{eq:redefAfg}, \eqref{eq:b1b2FGdef}, \eqref{eq:redefB1fg} and~\eqref{eq:redefB2fg} then yields
\begin{align*}
A_{F,G}-B_{F,G}&=\sum_{H\subset F}\sum_{W\subset G}(-1)^{|G\setminus W|+|F\setminus H|}\left(C^{H\cup W}+C^{F\setminus H}_{\|H\cup W}-C^W_{\|F}\right)(\nabla\phi_T^{H\cup W}+\xi),
\end{align*}
which proves~\eqref{eq:satiffinpr} by definition~\eqref{eq:def-diff}--\eqref{eq:defdeltaxi0} of $\delta_\xi^{F\cup G}\phi_T$.
\end{proof}

\subsection{Basic energy estimates}
The advantage of the massive term approximations $\phi_T^F$ is to localize the dependence with respect to the coefficients to a ball of radius $\sqrt{T}$ (up to exponentially small corrections). While this regularization in $T$ allows us to get rid of convergence issues at infinity, convergence problems may also occur at short distances because of high concentrations of the point process $\rho$. In order to avoid such issues, we further assume that $\rho$ has all its moments finite. The assumption that $T<\infty$ and the finite moments assumption are crucial to make rigorous all subsequent formal computations. 
Under these assumptions we shall prove some basic energy estimates that are uniform with respect to the regularization parameter $T$ and to the moments bounds on $\rho$; these estimates will be substantially improved in next section.

\begin{lem}\label{lem:finiteT}
Assume that $\E[\rho(Q)^s]<\infty$ for all $s\ge1$. Then, for all $L\sim1$,  $k\ge0$, $H\subset\N$, $s\ge1$, and $T>0$, the following estimate holds
\[\E\bigg[\bigg(\sum_{|F|=k}T^{-\frac12}(\delta_\xi^F\phi_T^H)_L(0)+(\nabla\delta_\xi^F\phi_T^H)_L(0)\bigg)^s\bigg]\le C_T^{sk}\E[\rho(B_R)^{sk}]<\infty
,\]
for some constant $C_{T}\sim_{T} 1$,
where $\delta_\xi^F\phi_T^H$ is as in \eqref{eq:defdeltaxi0}, and where we write $(f)_L(x):=(\fint_{B_L(x)}|f|^2)^{\frac12}$ for the local quadratic average of any map $f$.
\qed
\end{lem}

\begin{proof}
Since our argument is deterministic (we take the expectation only at the very end), we can assume w.l.o.g. that $H=\varnothing$.
By~\eqref{eq:satif1} in Lemma~\ref{lem:eqsatif},  $\delta_\xi^F\phi_T$ satisfies
\[\frac1T\delta_\xi^F\phi_T-\nabla\cdot A^{F}\nabla\delta_\xi^F\phi_T=\sum_{S\subset F}(-1)^{|S|+1}\nabla\cdot C_{S\| F\setminus S}\nabla\delta_\xi^{F\setminus S}\phi_T.\]
Let $z\in \R^d$ and set $\eta_T^z(x):=e^{-c|x-z|/\sqrt T}$ with $c>0$ to be chosen later.
Testing this equation with $\eta_T^z\delta_\xi^F\phi_T$ in the whole space, and noting that $|\nabla\eta^z_T|=c\eta^z_T/\sqrt T$, we obtain the starting point for a Caccioppoli's inequality
\begin{multline*}
\frac1T\int_{\R^d}\eta^z_T|\delta_\xi^F\phi_T|^2+\int_{\R^d}\eta^z_T|\nabla\delta_\xi^F\phi_T|^2\,
\lesssim~\sum_{S\subset F}\int_{J_{S\|F\setminus S}}\eta^z_T|\nabla\delta_\xi^F\phi_T|\,|\nabla\delta_\xi^{F\setminus S}\phi_T|\\
\qquad+\frac{c}{\sqrt T}\sum_{S\subset F}\int_{J_{S\|F\setminus S}}\eta^z_T|\delta_\xi^F\phi_T|\,|\nabla\delta_\xi^{F\setminus S}\phi_T|+\frac c{\sqrt T}\int_{\R^d}\eta^z_T|\delta_\xi^F\phi_T|\,|\nabla\delta_\xi^F\phi_T|.
\end{multline*}
It is crucial to note here that, for fixed $F$, the sets $J_{S\|F\setminus S}$, $S\subset F$, are all disjoint. By Young's inequality, and choosing $c>0$ small enough so that one may absorb all the terms but two in the left-hand side, this turns into
\begin{align*}
\frac1T\int_{\R^d}\eta^z_T|\delta_\xi^F\phi_T|^2+\int_{\R^d}\eta^z_T|\nabla\delta_\xi^F\phi_T|^2\lesssim\sum_{S\subset F}\int_{J_{S\|F\setminus S}}\eta^z_T|\nabla\delta_\xi^{F\setminus S}\phi_T|^2.
\end{align*}
For $L\sim1$, taking the square root of both sides yields
\begin{align}
e^{-cL/\sqrt T}\left(T^{-\frac12}(\delta_\xi^F\phi_T)_L(z)+(\nabla\delta_\xi^F\phi_T)_L(z)\right)&\lesssim\left(\sum_{S\subset F}\int_{J_{S\|F\setminus S}}\eta^z_T|\nabla\delta_\xi^{F\setminus S}\phi_T|^2\right)^{\frac12}\nonumber\\
&\le \sum_{S\subset F}\left(\int_{J_{S\|F\setminus S}}\eta^z_T|\nabla\delta_\xi^{F\setminus S}\phi_T|^2\right)^{\frac12}.\label{eq:preboundT}
\end{align}
Now note that, renumbering the sum in terms of $F\setminus S$, and using that $J_{S\|F\setminus S}=\varnothing$ whenever $S=\varnothing$, we get
\begin{align*}
\sum_{|F|=k}\sum_{S\subset F}\left(\int_{J_{S\|F\setminus S}}\eta_T^z|\nabla\delta_\xi^{F\setminus S}\phi_T|^2\right)^{\frac12}&\le\sum_{|F|\le k-1}\sum_{|S|\le k}\left(\int_{J_{S\|F}}\eta_T^z|\nabla\delta_\xi^{F}\phi_T|^2\right)^{\frac12},
\end{align*}
and hence, as $J_{S\|F}\subset J_n\subset B_R(q_n)$ for any $n\in S$,
\begin{align}
\sum_{|F|=k}\sum_{S\subset F}\left(\int_{J_{S\| F\setminus S}}\eta_T^z|\nabla\delta_\xi^{F\setminus S}\phi_T|^2\right)^{\frac12}&\le \sum_{|F|\le k-1}\sum_{n}\left(\int_{B_R(q_n)}\eta_T^z|\nabla\delta_\xi^{F}\phi_T|^2\right)^{\frac12}\sum_{|S|\le k\atop n\in S}\mathds1_{J_{S\|F}\ne\varnothing}.\label{eq:bornplijd1}
\end{align}
We bound the last sum as follows: for any fixed $n$, recalling that by assumption~\eqref{eq:boundrho} the intersections of the inclusions $J_n$'s are of degree at most $\Gamma\sim1$,
\begin{align}\label{eq:firstboundsumchi}
\sum_{|S|\le k\atop n\in S}\mathds1_{J_{S\|F}\ne\varnothing}&\le \sum_{|S|\le k\atop n\in S}\mathds1_{J_{S}\ne\varnothing}\le \sum_{j=1}^k\binom{\Gamma-1}{j-1}\le 2^{\Gamma-1}\lesssim1.
\end{align}
As we have $\eta_T^z(x)\le e^{-c|q_n-z|/\sqrt T}e^{cR/\sqrt T}$ for all $x\in B_R(q_n)$, we can then deduce from~\eqref{eq:bornplijd1} and \eqref{eq:firstboundsumchi}:
\begin{align*}
\sum_{|F|=k}\sum_{S\subset F}\left(\int_{J_{S\|F\setminus S}}\eta_T^z|\nabla\delta_\xi^{F\setminus S}\phi_T|^2\right)^{\frac12}&\lesssim e^{cR/\sqrt T}\sum_{|F|\le k-1}\sum_{n}e^{-c|q_n-z|/\sqrt T}(\nabla\delta_\xi^{F}\phi_T)_R(q_n).
\end{align*}
We then sum \eqref{eq:preboundT} over $|F|=k$, $k\ge 1$, and use the above estimate to get for any $z\in\R^d$,
\begin{align*}
S_T^k(z):=&~T^{-\frac12}\sum_{|F|=k}(\delta_\xi^F\phi_T)_L(z)+\sum_{|F|=k}(\nabla\delta_\xi^F\phi_T)_L(z)\\
\lesssim&~ e^{c(L+R)/\sqrt T}\sum_ne^{-c|q_n-z|/\sqrt T}\sum_{|F|\le k-1}(\nabla\delta_\xi^F\phi_T)_R(q_n).
\end{align*}
Combining this with \eqref{eq:modif-corr-estim}, we conclude by induction that, for some (deterministic) constant $C_T\sim_T1$,
\[S_T^k(z)\lesssim C_T^k\sum_{j=1}^k\underbrace{\sum_{n_1,\ldots,n_j}e^{-c|q_{n_1}-z|/\sqrt T}\prod_{i=2}^{j}e^{-c|q_{n_{i-1}}-q_{n_i}|/\sqrt T}}_{\displaystyle =:I_T^j(z)}.\]
It only remains to compute the sum $I_T^j(z)$. For that purpose, we compare sums to integrals
\begin{align*}
I_T^j(z)\le e^{cjR/\sqrt T}\sum_{n_1,\ldots,n_j}\int_{B_R(q_{n_1})}\!\!\!\!\ldots\!\int_{B_R(q_{n_j})}e^{-c|x_1-z|/\sqrt T}\prod_{i=2}^{j}e^{-c|x_{{i-1}}-x_{i}|/\sqrt T}dx_j\ldots dx_1,
\end{align*}
and hence,
\begin{align*}
I_T^j(z)\le e^{cjR/\sqrt T}\int_{(\R^d)^j}e^{-c|x_1-z|/\sqrt T}\rho(B_R(x_1))\prod_{i=2}^{j}\bigg(e^{-c|x_{{i-1}}-x_{i}|/\sqrt T}\rho(B_R(x_i))\bigg)dx_1\ldots dx_j.
\end{align*}
Taking expectation of $I_T^j(z)^s$, for some $s\ge1$, and applying the triangle and the H\"older inequalities, we obtain
\begin{align*}
\E[I_T^j(z)^s]^{1/s}&\le e^{cjR/\sqrt T}\E[\rho(B_R)^{sj}]^{1/s}\int_{(\R^d)^j}e^{-c|x_1-z|/\sqrt T}\prod_{i=2}^{j}e^{-c|x_{{i-1}}-x_{i}|/\sqrt T}dx_1\ldots dx_j,
\end{align*}
which finally gives, by an obvious change of variables,
\begin{align*}
\E[I_T^j(z)^s]^{\frac1s}&\le e^{cjR/\sqrt T}\E[\rho(B_R)^{sj}]^{\frac1s}\bigg(\int_{\R^d}e^{-c|x|/\sqrt T}dx\bigg)^j=C^jT^{jd/2}e^{cjR/\sqrt T}\E[\rho(B_R)^{sj}]^{\frac1s},
\end{align*}
and the announced result is then proved.
\end{proof}

Based on this deterministic estimate, we prove a lemma which will be crucial to give sense to formal calculations, and implies the absolute convergence of all the series we will be considering for fixed $T$, at least under the additional assumption that moments of $\rho$ are finite. Note that the result obviously also holds for $\nabla\phi_T^{F}$ replaced e.g. by $\nabla\phi_T^{E^{(p)}\cup F}$.

\begin{lem}\label{lem:convfinT}
Assume that $\E[\rho(Q)^s]<\infty$ for all $s\ge1$.
For all $T>0$, and $k\ge1$, we have
\begin{align}\label{eq:convfinTS}
S_T^k:=\sum_{|F|=k}\sum_{G\subset F}\E\left[|C_{F\setminus G}|\,|\nabla\delta_\xi^G\phi_T|\,(1+|\nabla\phi_T^F|)\right]\,&<\infty,\\
\sum_{|F|=k}\sum_{G\subset F}\E\left[|C_{G}|\,|\nabla\delta_\xi^{F\setminus G}\phi_T^G|\,(1+|\nabla\phi_T^F|)\right]\,&<\infty,\label{eq:convfinTS+}\\
\sum_{|F|=k}\E\left[|C_F|\,(1+|\nabla\phi_T|)(1+|\nabla\phi_T^{ F}|)\right]\,&<\infty.\label{eq:convfinTS2}
\end{align}
\end{lem}

\begin{proof}
We only prove \eqref{eq:convfinTS}; the proofs of the other statements are similar.
Let $k\ge0$ be fixed. 
By stationarity we add a local average over the ball $B_L$, say, with $L\sim1$, we apply the Cauchy-Schwarz inequality, and note that, for all $x\in B_L$,
\[|C_H(x)|\lesssim \mathds1_{x\in J_n,\forall n\in H}\le\mathds1_{x\in B_R(q_n),\forall n\in H}=\mathds1_{q_n\in B_R(x),\forall n\in H}\le\mathds1_{q_n\in B_{R+L},\forall n\in H}=:\chi_L(H),\]
so that we can write by the change of variables $F\leadsto F\cup G$,
\begin{align*}
S_T^k&\lesssim\sum_{|F|=k}\sum_{G\subset F}\E\bigg[\chi_L(F\setminus G)(\nabla\delta_\xi^G\phi_T)_L(1+(\nabla\phi_T^{F})_L)\bigg]\\
&\le\sum_{|F|\le k}\sum_{|G|\le k}\E\bigg[\chi_L(F)(\nabla\delta_\xi^G\phi_T)_L(1+(\nabla\phi_T^{F\cup G})_L)\bigg].
\end{align*}
Using the deterministic estimate~\eqref{eq:modif-corr-estim} in the form of $(\nabla\phi_T^{F\cup G})_L\lesssim T^{\frac{d}{2}}$, this yields
\begin{align*}
S_T^k&\lesssim T^{\frac{d}{2}}\E\bigg[\bigg(\sum_{|F|\le k}\chi_L(F)\bigg)\bigg(\sum_{|G|\le k}(\nabla\delta_\xi^G\phi_T)_L\bigg)\bigg].
\end{align*}
The first sum can be estimated as follows: since $\binom ni\le n^i/i!\le (en/i)^i$ and  $\sum_{i=1}^\infty(e/i)^i\lesssim1$,
\begin{align*}
\sum_{|F|\le k}\chi_L(F)=\sum_{|F|\le k}\mathds1_{q_n\in B_{R+L},\forall n\in F}\le\sum_{i=0}^k\binom{\rho(B_{R+L})}i\lesssim \rho(B_{R+L})^k,
\end{align*}
so that we conclude by Lemma~\ref{lem:finiteT} and the Cauchy-Schwarz inequality that
\begin{align*}
S_T^k&\lesssim T^{d/2}\E[\rho(B_{R+L})^{2k}]^{1/2}\E\bigg[\bigg(\sum_{|G|\le k}(\nabla\delta_\xi^G\phi_T)_L\bigg)^2\bigg]^{1/2}<\infty.\qedhere
\end{align*}
\end{proof}

We now turn to energy estimates that hold uniformly with respect to $T$ and the moments bounds on $\rho$. The following two estimates will be further improved in the following section.

\begin{lem}\label{lem:aprioriproba}
Assume that $\E[\rho(Q)^s]<\infty$ for all $s\ge1$. Then, there exists a constant $C\sim1$ (independent of $T$ and of the moments of $\rho$) such that, for all $k\ge0$ and $T>0$,
\[\E\bigg[\sum_{|F|=k}|\nabla\delta_\xi^{F}\phi_T|^2\bigg]\le C^{k+1}.\]
\qed
\end{lem}
\begin{proof}
For $k=0$, the result $\E[|\nabla\phi_T+\xi|^2]\lesssim1$ reduces to the basic energy estimate on the modified corrector (see~\eqref{eq:modif-corr-estim}). We now argue by induction. Assume that the result holds for some fixed $k\ge0$. From~\eqref{eq:satif1} in Lemma~\ref{lem:eqsatif}, we learn that $\delta_\xi^F\phi_T$ satisfies on $\R^d$
\[\frac1T\delta_\xi^F\phi_T-\nabla\cdot A^{F}\nabla\delta_\xi^F\phi_T=\sum_{S\subset F}(-1)^{|S|+1}\nabla\cdot C_{S\|F\setminus S}\nabla\delta_\xi^{F\setminus S}\phi_T.\]
We test this equation with $\chi_N\delta_\xi^F\phi_T$, where $\chi_N$ is a cut-off function for $B_{N}$ in $B_{2N}$ such that $|\nabla\chi_N|\lesssim1/N$. This yields
\begin{align}\label{eq:bornmlim}
\int_{B_N}|\nabla\delta_\xi^F\phi_T|^2&\lesssim \sum_{S\subset F}\int_{B_{2N}}\mathds1_{J_{S\|F\setminus S}}|\nabla\delta_\xi^F\phi_T|\,|\nabla\delta_\xi^{F\setminus S}\phi_T|\\
&\qquad+\frac1N\int_{B_{2N}}|\delta_\xi^F\phi_T|\,|\nabla\delta_\xi^F\phi_T|+\frac1N\sum_{S\subset F}\int_{B_{2N}}\mathds1_{J_{S\|F\setminus S}}|\delta_\xi^F\phi_T|\,|\nabla\delta_\xi^{F\setminus S}\phi_T|.\nonumber
\end{align}
We then take the expectation, sum over $|F|=k-1$ (all the sums are convergent by Lemmas~\ref{lem:finiteT} and~\ref{lem:convfinT}), and divide by $N^d$:
\begin{align*}
&\E\bigg[\fint_{B_N}\sum_{|F|=k+1}|\nabla\delta_\xi^F\phi_T|^2\bigg]\lesssim \E\bigg[\fint_{B_{2N}}\sum_{|F|=k+1}\sum_{S\subset F}\mathds1_{J_{S\|F\setminus S}}|\nabla\delta_\xi^F\phi_T|\,|\nabla\delta_\xi^{F\setminus S}\phi_T|\bigg]\\
&\hspace{1.3cm}+\frac1N\E\bigg[\fint_{B_{2N}}\sum_{|F|=k+1}|\delta_\xi^F\phi_T|\,|\nabla\delta_\xi^F\phi_T|\bigg]+\frac1N\E\bigg[\fint_{B_{2N}}\sum_{|F|=k+1}\sum_{S\subset F}\mathds1_{J_{S\|F\setminus S}}|\delta_\xi^F\phi_T|\,|\nabla\delta_\xi^{F\setminus S}\phi_T|\bigg].
\end{align*}
Since each sum above is absolutely convergent and defines an integrable stationary random field (the expectation of which obviously does not depend on the point it is taken), this inequality also takes the form
\begin{align*}
&\E\bigg[\sum_{|F|=k+1}|\nabla\delta_\xi^F\phi_T|^2\bigg]\lesssim \E\bigg[\sum_{|F|=k+1}\sum_{S\subset F}\mathds1_{J_{S\|F\setminus S}}|\nabla\delta_\xi^F\phi_T|\,|\nabla\delta_\xi^{F\setminus S}\phi_T|\bigg]\\
&\hspace{2cm}+\frac1N\E\bigg[\sum_{|F|=k+1}|\delta_\xi^F\phi_T|\,|\nabla\delta_\xi^F\phi_T|\bigg]+\frac1N\E\bigg[\sum_{|F|=k+1}\sum_{S\subset F}\mathds1_{J_{S\|F\setminus S}}|\delta_\xi^F\phi_T|\,|\nabla\delta_\xi^{F\setminus S}\phi_T|\bigg].
\end{align*}
Taking the limit $N\uparrow\infty$ then yields
\begin{align}\label{eq:REF-ERG}
\E\bigg[\sum_{|F|=k+1}|\nabla\delta_\xi^F\phi_T|^2\bigg]&\lesssim \E\bigg[\sum_{|F|=k+1}\sum_{S\subset F}\mathds1_{J_{S\|F\setminus S}}|\nabla\delta_\xi^F\phi_T|\,|\nabla\delta_\xi^{F\setminus S}\phi_T|\bigg].
\end{align}
By Young's inequality and the disjointness of the sets $J_{S\|F\setminus S}$, $S\subset F$ (for fixed $F$), in the form of $\sum_{S\subset F}\mathds1_{J_{S\|F\setminus S}}\le1$, 
we may absorb part of the right-hand side into the left-hand side, and obtain
\begin{align}
\E\bigg[\sum_{|F|=k+1}|\nabla\delta_\xi^F\phi_T|^2\bigg]&\lesssim \E\bigg[\sum_{|F|=k+1}\sum_{S\subset F}\mathds1_{J_{S\|F\setminus S}}|\nabla\delta_\xi^{F\setminus S}\phi_T|^2\bigg]\nonumber\\
&\le\E\bigg[\sum_{|F|\le k}|\nabla\delta_\xi^F\phi_T|^2\sum_{|S|\le k+1}\mathds1_{J_{S\|F}}\bigg],\label{eq:lem1apbasis}
\end{align}
where we used that $J_{\varnothing\|F}=\varnothing$.

By assumption~\eqref{eq:boundrho}, proceeding as for~\eqref{eq:firstboundsumchi}, we have
\begin{align}\label{eq:firstboundsumchi+}
\sum_{|S|\le k+1}\mathds1_{J_{S\|F}}(0)\le\sum_{|S|\le k+1}\mathds1_{J_{S}}(0)\le\sum_{j=1}^{k+1}\binom{\Gamma}{j}\le 2^{\Gamma}\lesssim1,
\end{align}
so that~\eqref{eq:lem1apbasis} finally turns into
\begin{align*}
\E\bigg[\sum_{|F|=k+1}|\nabla\delta_\xi^F\phi_T|^2\bigg]&\lesssim \E\bigg[\sum_{|F|\le k}|\nabla\delta_\xi^F\phi_T|^2\bigg],
\end{align*}
from which the desired conclusion follows by the induction assumption.
\end{proof}

For sums $\sum_{|F|=k}$ of size $k=1$, the following result is easily proven as an energy estimate in the probability space; for general $k$ it also holds but the proof relies on a subtle induction and combinatorial argument, which is presented in next section.

\begin{lem}\label{lem:apjust1}
Assume that $\E[\rho(Q)^s]<\infty$ for all $s\ge1$. Then, for all $T>0$, we have (uniformly in $T$ and in the moments of $\rho$)
\[\E\bigg[\Big|\sum_{n}\nabla\delta_\xi^{\{n\}}\phi_T\Big|^2\bigg]\lesssim 1.\]
\qed
\end{lem}

\begin{proof}
By Lemma~\ref{lem:finiteT} for $k=1$, the sum $\sum_{n}\delta_\xi^{\{n\}}\phi_T$ is well-defined 
in $H^1_\loc(\R^d)$ and satisfies the following equation on $\R^d$:
\[\frac1T\sum_n\delta_\xi^{\{n\}}\phi_T-\nabla\cdot A\nabla\sum_n\delta_\xi^{\{n\}}\phi_T=\nabla\cdot\sum_nC^{\{n\}}(\nabla\phi_T^{\{n\}}+\xi).\]
We then test this equation with $\chi_N (\sum_{n}\delta_\xi^{\{n\}}\phi_T)$ for some cut-off $\chi_N$  for $B_{N}$
in $B_{2N}$ such that $|\nabla \chi_N|\lesssim 1/N$.
Since $\sum_{n}\nabla \delta_\xi^{\{n\}}\phi_T$ is stationary, we may proceed as for the proof of \eqref{eq:REF-ERG} in Lemma~\ref{lem:aprioriproba}, and obtain after taking the expectation and the limit $N\uparrow \infty$ (or equivalently testing the equation in probability):
\begin{align*}
\E\bigg[\Big|\sum_n\nabla\delta_\xi^{\{n\}}\phi_T\Big|^2\bigg]&\lesssim\E\bigg[\Big|\sum_nC^{\{n\}}(\nabla\phi_T^{\{n\}}+\xi)\Big|^2\bigg]\lesssim\E\bigg[\bigg(\sum_n\mathds1_{J_n}\bigg)\bigg(\sum_n\mathds1_{J_n}(1+|\nabla\phi_T^{\{n\}}|^2)\bigg)\bigg].
\end{align*}
By assumption~\eqref{eq:boundrho}, proceeding as for~\eqref{eq:firstboundsumchi}, we have $\sum_n\mathds1_{J_n}(0)\lesssim 1$, so that, using in addition the decomposition $1+|\nabla\phi_T^{\{n\}}|^2\lesssim (1+|\nabla\phi_T|^2)+|\nabla\delta_\xi^{\{n\}}\phi_T|^2$,
we obtain
\begin{align*}
\E\bigg[\Big|\sum_n\nabla\delta_\xi^{\{n\}}\phi_T\Big|^2\bigg]&\lesssim \E[1+|\nabla\phi_T|^2]+\E\bigg[\sum_n|\nabla\delta_\xi^{\{n\}}\phi_T|^2\bigg]\lesssim 1,
\end{align*}
where the last inequality follows from Lemma~\ref{lem:aprioriproba} with $k=1$.
\end{proof}

\subsection{Improved energy estimates}\label{chap:apriori}
In this section, we prove the following generalization of Lemma~\ref{lem:apjust1} to any order $k$ in the following form (choosing $j=k$ in~\eqref{eq:boundS} below): there is a constant $C\sim1$ such that, for all $k\ge1$ and $T>0$,
\[\E\bigg[\Big|\sum_{|F|=k}\nabla\delta^F\phi_T\Big|^2\bigg]\le C^{k+1},\]
and we give an interpolation result between this inequality and the energy estimates of Lemma~\ref{lem:aprioriproba}, that is, a $\ell^1(\N^{k-j},L^2(\Omega))$-type estimate.

\begin{prop}\label{prop:apimproved}
Assume that $\E[\rho(Q)^s]<\infty$ for all $s\ge1$. Then, there exists a constant $C\sim1$ (independent of $T$ and of the moments of $\rho$) such that, for all $T>0$, $k\ge0$, and $0\leq j\leq k$, 
\begin{align}\label{eq:boundS}
S_j^k:=\E\bigg[\sum_{|G|=k-j}\Big|\sum_{|F|=j\atop F\cap G=\varnothing}\nabla\delta^{F\cup G}\phi_T\Big|^2\bigg]\le C^{k+1}.
\end{align}
\qed
\end{prop}

\begin{proof}
We proceed by induction and split the proof into two steps.

\medskip

\step{1} Preliminary.\\
For $k=j=0$, the estimate $S_0^0=\E[|\nabla\phi_T+\xi|^2]\lesssim1$ reduces to the energy estimate for the modified corrector. (Note also that the estimate $S_1^1\lesssim1$ already follows from Lemma~\ref{lem:apjust1} --- but this will not be used here.)
We now argue by a double induction argument. Since the result is proven for $k=0$, we may indeed argue by induction on $k$: we assume that $S_j^{k'}\le C^{k'+1}$ for all $0\le j\le k'$ and for all $0\le k'\le k$, and shall prove that $S_j^{k+1}\le C^{k+2}$ for all $0\le j\le k+1$. 
Since Lemma~\ref{lem:aprioriproba} implies the desired result for $j=0$, we may as well argue by induction on $j$: we further assume that $S_{j'}^{k+1}\le C^{k+2}$ for all $0\le j'\le j$, for some $0\le j<k+1$, and shall prove that $S_{j+1}^{k+1}\le C^{k+2}$.

Before we turn to Step~2, we state another combinatorial inequality we shall need in the proof.
Let $G,S\subset \N$ be finite fixed disjoint subsets. We claim that
\begin{align}\label{eq:claimcombi+}
\bigg|\sum_{|F|=k\atop F\cap (G\cup S)=\varnothing}\nabla\delta_\xi^{F\cup G}\phi_T\bigg|\le\sum_{l=0}^{|S|}\sum_{|L|=l\atop L\subset S}\bigg|\sum_{|F|=k-l\atop F\cap (G\cup L)=\varnothing}\nabla\delta_\xi^{F\cup G\cup L}\phi_T\bigg|.
\end{align}
We first rewrite
\begin{align*}
S_{G,S}^k:=\sum_{|F|=k\atop F\cap (G\cup S)=\varnothing}\nabla\delta_\xi^{F\cup G}\phi_T=\sum_{|F|=k\atop F\cap G=\varnothing}\mathds1_{F\cap S=\varnothing}\nabla\delta_\xi^{F\cup G}\phi_T,
\end{align*}
where we can decompose, by the usual inclusion-exclusion argument,
\begin{align*}
\mathds1_{F\cap S=\varnothing}=1-\mathds1_{F\cap S\ne\varnothing}=1-\sum_{l=1}^{|S|}(-1)^{l+1}\sum_{|L|=l\atop L\subset S}\mathds1_{L\subset F}=\sum_{l=0}^{|S|}(-1)^{l}\sum_{|L|=l\atop L\subset S}\mathds1_{L\subset F},
\end{align*}
so that $S^k_{G,S}$ becomes, by a change of variables,
\begin{align*}
S_{G,S}^k&=\sum_{l=0}^{|S|}(-1)^{l}\sum_{|L|=l\atop L\subset S}\sum_{|F|=k\atop F\cap G=\varnothing}\mathds1_{L\subset F}\nabla\delta_\xi^{F\cup G}\phi_T=\sum_{l=0}^{|S|}(-1)^{l}\sum_{|L|=l\atop L\subset S}\sum_{|F|=k-l\atop F\cap (G\cup L)=\varnothing}\nabla\delta_\xi^{F\cup G\cup L}\phi_T,
\end{align*}
and the claim~\eqref{eq:claimcombi+} then follows from the triangle inequality.

\medskip

\step{2} Bound on $S_{j+1}^{k+1}$.\\
Let $G\subset \N$ be a finite subset.
By Lemma~\ref{lem:finiteT} we may sum equation~\eqref{eq:satif3} of Lemma~\ref{lem:eqsatif} (with $H=\varnothing$) for $\delta^{F\cup G} \phi_T$
over $|F|=j+1$, $F\cap G=\varnothing$, which yields the following equation for the sum $\sum_{|F|=j+1\atop F\cap G=\varnothing}\delta^{F\cup G}\phi_T$ on $\R^d$:
\begin{align*}
&\frac1T\sum_{|F|=j+1\atop F\cap G=\varnothing}\delta_\xi^{F\cup G}\phi_T-\nabla\cdot A^{G}\nabla\sum_{|F|=j+1\atop F\cap G=\varnothing}\delta_\xi^{F\cup G}\phi_T\\
&=\nabla\cdot\sum_{|F|=j+1\atop F\cap G=\varnothing}\sum_{S\subset F}\sum_{U\subset G}(-1)^{|S|+|U|+1} C_{S\cup U\|G\setminus U}\nabla\delta_\xi^{(F\setminus S)\cup (G\setminus U)}\phi_T^S\\
&=\nabla\cdot\sum_{U\subset G}\sum_{|S|\le j+1\atop S\cap G=\varnothing}(-1)^{|S|+|U|+1}C_{S\cup U\| G\setminus U}\sum_{|F|=j+1-|S|\atop F\cap (G\cup S)=\varnothing}\nabla\delta_\xi^{F\cup (G\setminus U)}\phi_T^S.
\end{align*}
We test this equation with $\chi_N\sum_{|F|=j+1,F\cap G=\varnothing}\delta_\xi^{F\cup G}\phi_T$, where $\chi_N$ is a cut-off function for $B_N$ in $B_{2N}$ such that $|\nabla \chi_N|\lesssim 1/N$, 
we take the sum over $|G|=(k+1)-(j+1)=k-j$ (which is again absolutely converging by Lemma~\ref{lem:finiteT}), take the expectation, and then use stationarity to pass to the limit $N\uparrow\infty$, as in the proof of~\eqref{eq:REF-ERG} in Lemma~\ref{lem:aprioriproba}. 
With the notation $a\wedge b=\min\{a,b\}$, this yields
\begin{align*}
S_{j+1}^{k+1}&\lesssim \E\bigg[\sum_{|G|=k-j}\bigg|\sum_{U\subset G}\sum_{|S|\le (j+1)\wedge \Gamma\atop S\cap G=\varnothing,S\cup U\ne\varnothing}(-1)^{|S|+|U|+1}C_{S\cup U\|G\setminus U}\sum_{|F|=j+1-|S|\atop F\cap (G\cup S)=\varnothing} \nabla\delta_\xi^{F\cup (G\setminus U)}\phi_T^S\bigg|^2\bigg],
\end{align*}
where the additional restriction $S\cup U\ne\varnothing$ follows from the fact that $C_{S\cup U\|G\setminus U}$ vanishes identically otherwise
and where we have further restricted to $|S|\le \Gamma$ since by assumption \eqref{eq:boundrho} there is no intersection of degree larger than $\Gamma$. Since we have $|C_{S\cup U\|G\setminus U}|\lesssim\mathds1_{J_S}\mathds1_{J_{U\|G\setminus U}}$ (using here notation $\mathds1_{J_\varnothing}=1$), and the $J_{U\|G\setminus U}$'s are disjoint for $U\subset G$ (for fixed $G$), we deduce
\begin{align*}
S_{j+1}^{k+1}&\lesssim \E\bigg[\sum_{|G|=k-j}\sum_{U\subset G}\mathds1_{J_{U\|G\setminus U}}\bigg(\sum_{|S|\le (j+1)\wedge \Gamma\atop S\cap G=\varnothing,S\cup U\ne\varnothing}\mathds1_{J_S}\bigg|\sum_{|F|=j+1-|S|\atop F\cap (G\cup S)=\varnothing} \nabla\delta_\xi^{F\cup (G\setminus U)}\phi_T^S\bigg|\bigg)^2\bigg].
\end{align*}
As for \eqref{eq:firstboundsumchi+}, we have $\sum_{|S|\le j+1}\mathds1_{J_S}(0)\lesssim1$, so that by the Cauchy-Schwarz inequality,
this estimate turns into
\begin{align*}
S_{j+1}^{k+1}&\lesssim \E\bigg[\sum_{|G|=k-j}\sum_{U\subset G}\mathds1_{J_U}\sum_{|S|\le (j+1)\wedge \Gamma \atop S\cap G=\varnothing,S\cup U\ne\varnothing}\mathds1_{J_S}\bigg|\sum_{|F|=j+1-|S|\atop F\cap (G\cup S)=\varnothing} \nabla\delta_\xi^{F\cup (G\setminus U)}\phi_T^S\bigg|^2\bigg].
\end{align*}
Now using the decomposition $\nabla\delta_\xi^{F\cup (G\setminus U)}\phi_T^S=\sum_{R\subset S}\nabla\delta_\xi^{F\cup R\cup (G\setminus U)}\phi_T$ (that is, \eqref{eq:defdeltaxi} with $H=\varnothing$, $G\leadsto F\cup (G\setminus U)$ and $F\leadsto S$), together with the observation that
\[\mathds1_{J_S}\bigg(\sum_{R\subset S}a_R\bigg)^2\le \mathds1_{J_S}\bigg(\sum_{R\subset S}\mathds1_{J_R}a_R\bigg)^2\lesssim \mathds1_{J_S}\sum_{R\subset S}a_R^2,\]
which follows again from combining the Cauchy-Schwarz inequality with inequality $\sum_{|R|\le j+1}\mathds1_{J_R}\lesssim1$, we obtain
\begin{align*}
S_{j+1}^{k+1}&\lesssim \E\bigg[\sum_{|G|=k-j}\sum_{U\subset G}\mathds1_{J_U}\sum_{|S|\le (j+1)\wedge \Gamma \atop S\cap G=\varnothing,S\cup U\ne\varnothing}\mathds1_{J_S}\sum_{R\subset S}\bigg|\sum_{|F|=j+1-|S|\atop F\cap (G\cup S)=\varnothing} \nabla\delta_\xi^{F\cup R\cup (G\setminus U)}\phi_T\bigg|^2\bigg]\\
&\le \sum_{i=0}^{(j+1)\wedge \Gamma}\E\bigg[\sum_{|U|\le k-j}\mathds1_{J_U}\sum_{|G|\le k-j\atop G\cap U=\varnothing}\delta^G_{ijk}\sum_{|S|=i\atop S\cap G=\varnothing}\mathds1_{J_S}\sum_{R\subset S}\bigg|\sum_{|F|=j+1-i\atop F\cap (G\cup U\cup S)=\varnothing} \nabla\delta_\xi^{F\cup R\cup G}\phi_T\bigg|^2\bigg],
\end{align*}
where we have set $\delta^G_{ijk}=0$ when simultaneously $|G|=k-j$ and $i=0$, and $\delta^G_{ijk}=1$ otherwise. 
By \eqref{eq:claimcombi+} and the inequality $\sum_{|L|\le j+1}\mathds1_{J_L}\lesssim1$, for any $R\subset S$ and any $G\cap U=\varnothing=S\cap G$, we have
\begin{align}
\mathds1_{J_S\cap J_U}\bigg|\sum_{|F|=j+1-i\atop F\cap (G\cup U\cup S)=\varnothing} \nabla\delta_\xi^{F\cup R\cup G}\phi_T\bigg|^2&\le\bigg( \sum_{l=0}^{j+1-i}\sum_{|L|=l\atop L\subset U\cup S\setminus R}\mathds1_{J_L}\bigg|\sum_{|F|=j+1-i-l\atop F\cap (L\cup R\cup G)=\varnothing}\nabla\delta_\xi^{F\cup L\cup R\cup G}\phi_T \bigg|\bigg)^2\nonumber\\
&\lesssim \sum_{l=0}^{j+1-i}\sum_{|L|=l\atop L\subset U\cup S\setminus R}\mathds1_{J_L}\bigg|\sum_{|F|=j+1-i-l\atop F\cap (L\cup R\cup G)=\varnothing}\nabla\delta_\xi^{F\cup L\cup R\cup G}\phi_T \bigg|^2,\label{eq:inequtrplop}
\end{align}
and hence we obtain, using $\sum_{|U|\le k}\mathds1_{J_U}\lesssim1$ again, and using $U\cup (S\setminus R)\subset \N\setminus (G\cup R)$,
\begin{align*}
S_{j+1}^{k+1}&\lesssim \sum_{i=0}^{(j+1)\wedge \Gamma}\sum_{l=0}^{j+1-i}\E\bigg[\sum_{|U|\le k-j}\mathds1_{J_U}\sum_{|G|\le k-j\atop G\cap U=\varnothing}\delta^G_{ijk}
\\ &\qquad\qquad\qquad \qquad \sum_{|S|=i\atop S\cap G=\varnothing}\mathds1_{J_S}\sum_{R\subset S}\sum_{|L|=l\atop L\subset U\cup S\setminus R}\mathds1_{J_L}\bigg|\sum_{|F|=j+1-i-l\atop F\cap (L\cup R\cup G)=\varnothing} \nabla\delta_\xi^{F\cup L\cup R\cup G}\phi_T\bigg|^2\bigg]\\
&\lesssim \sum_{i=0}^{(j+1)\wedge \Gamma}\sum_{l=0}^{j+1-i}\E\bigg[\sum_{|G|\le k-j}\delta^G_{ijk}\sum_{|R|\le i\atop R\cap G=\varnothing}\mathds1_{J_R}\sum_{|L|= l\atop L\cap (G\cup R)=\varnothing}\mathds1_{J_L}\bigg|\sum_{|F|=j+1-i-l\atop F\cap (L\cup R\cup G)=\varnothing} \nabla\delta_\xi^{F\cup L\cup R\cup G}\phi_T\bigg|^2\bigg].
\end{align*}
Successively using $\Gamma\lesssim 1$ in the form of 
$\sum_{i=0}^{(j+1)\wedge \Gamma}\sum_{|R|\leq i} \,\lesssim \, \sum_{i=0}^{(j+1)\wedge \Gamma}\sum_{|R|=i}$ and $\sum_{L} \mathds1_{J_L}\lesssim 1$, we obtain by the change of variables $L\cup R \leadsto R$,
\begin{align*}
S_{j+1}^{k+1}&\lesssim \sum_{i=0}^{(j+1)\wedge \Gamma}\sum_{l=0}^{j+1-i}\E\bigg[\sum_{|R|= i}\mathds1_{J_R}\sum_{|L|= l\atop L\cap R=\varnothing}\mathds1_{J_L}\sum_{|G|\le k-j \atop G\cap (L\cup R)=\varnothing}\delta^G_{ijk}\bigg|\sum_{|F|=j+1-i-l\atop F\cap (L\cup R\cup G)=\varnothing} \nabla\delta_\xi^{F\cup L\cup R\cup G}\phi_T\bigg|^2\bigg]\\
&\lesssim \sum_{i=0}^{(j+1)\wedge \Gamma}\sum_{l=0}^{k-j}\E\bigg[\sum_{|R|= i}\mathds1_{J_R}\sum_{|G|=l\atop R\cap G=\varnothing}\delta^G_{ijk}\bigg|\sum_{|F|=j+1-i\atop F\cap (G\cup R)=\varnothing} \nabla\delta_\xi^{F\cup R\cup G}\phi_T\bigg|^2\bigg],
\end{align*}
or equivalently, recalling the definition of the $\delta_{ijk}^G$'s and of the $S_j^k$'s,
\begin{align}\label{eq:boundSjk+}
S_{j+1}^{k+1}&\lesssim \sum_{l=0}^{k-j-1}S_{j+1}^{l+j+1}+\sum_{i=1}^{j+1}\sum_{l=0}^{k-j}\E\bigg[\sum_{|R|= i}\mathds1_{J_R}\sum_{|G|=l\atop R\cap G=\varnothing}\bigg|\sum_{|F|=j+1-i\atop F\cap (G\cup R)=\varnothing} \nabla\delta_\xi^{F\cup R\cup G}\phi_T\bigg|^2\bigg].
\end{align}
Using again the fact that $\sum_{|R|=i}\mathds1_{J_R}\lesssim1$, we can bound
\begin{align*}
\E\bigg[\sum_{|R|= i}\mathds1_{J_R}\sum_{|G|=l\atop R\cap G=\varnothing}\bigg|\sum_{|F|=j+1-i\atop F\cap (G\cup R)=\varnothing} \nabla\delta_\xi^{F\cup R\cup G}\phi_T\bigg|^2\bigg]&=\E\bigg[\sum_{|G|=i+l}\sum_{R\subset G\atop |R|=i}\mathds1_{J_R}\bigg|\sum_{|F|=j+1-i\atop F\cap G=\varnothing} \nabla\delta_\xi^{F\cup G}\phi_T\bigg|^2\bigg]\\
&\lesssim\E\bigg[\sum_{|G|=i+l}\bigg|\sum_{|F|=j+1-i\atop F\cap G=\varnothing} \nabla\delta_\xi^{F\cup G}\phi_T\bigg|^2\bigg]=S_{j+1-i}^{l+j+1},
\end{align*}
so that~\eqref{eq:boundSjk+} turns into
\begin{align*}
S_{j+1}^{k+1}&\lesssim \sum_{l=0}^{k-j-1}S_{j+1}^{l+j+1}+\sum_{i=1}^{j+1}\sum_{l=0}^{k-j}S_{j+1-i}^{l+j+1}=\sum_{l=0}^{k}S_{j+1}^{l}+\sum_{i=0}^{j}\sum_{l=j+1}^{k+1}S_{i}^{l}.
\end{align*}
As the right-hand side only involves the $S^{k'}_{j'}$'s with $k'\le k$ or with $k'=k+1$, $j'\le j$, we conclude that  $S^{k+1}_{j+1}\le C^{k+2}$ by the induction assumption.
\end{proof}

\section{Proofs of the main results}\label{sec:analytic}

In this section, we prove the analyticity of the perturbed coefficients (Theorem~\ref{th:analytic}) and the analytical formulas for the derivatives (Corollary~\ref{cor:analytic}), from which we further deduce the Clausius-Mossotti formulas (Corollaries~\ref{cor:cm} and~\ref{cor:cm2}).

\subsection{Approximate derivatives at $p=0$}\label{chap:approxder}

In this subsection we devise analytical formulas for the derivatives of the map $p\mapsto A_T^{(p)}$ at $p=0$ under the assumptions that $T>0$ and $\E[\rho(Q)^s]<\infty$ for all $s\ge 1$. We shall show in particular that $A_T^{(p)}$ is $C^\infty$ at $p=0$. These results, which rely on the improved energy estimates of Proposition~\ref{prop:apriori}, constitute the core of the proof of Theorem~\ref{th:analytic}.

\medskip

Fix some direction $\xi\in\R^d$, $|\xi|=1$.
As in Section~\ref{sec:strategy}, we consider the exact and approximate differences
\[\Delta^{(p)}:=\xi\cdot (A^{(p)}_{\hom}-A_{\hom})\xi,\qquad\Delta_{T}^{(p)}:=\xi\cdot (A^{(p)}_{T}-A_{T})\xi,\]
and we recall that $\lim_T\Delta_T^{(p)}=\Delta^{(p)}$ follows from~\eqref{eq:approxhomcoeff}. By Lemma~\ref{lem:firstdecomp} the approximate difference satisfies
\begin{align}
\Delta_T^{(p)}&=\E[(\nabla\phi_T+\xi)\cdot C^{(p)}(\nabla\phi_T^{(p)}+\xi)]\label{eq:err1wd-2}.
\end{align}

By assumption~\eqref{eq:boundrho}, we may now appeal to the inclusion-exclusion formula in the form of~\eqref{eq:excl2}, so that~\eqref{eq:err1wd-2} turns into
\[\Delta_T^{(p)}=\sum_{j=1}^\Gamma(-1)^{j+1}\sum_{|F|=j}\E\left[(\nabla\phi_T+\xi)\cdot C_{F}(\nabla\phi_T^{E^{(p)}\cup F}+\xi)\mathds1_{F\subset E^{(p)}}\right],\]
where the sum is absolutely convergent by~\eqref{eq:convfinTS2} in Lemma~\ref{lem:convfinT}. Using that the event $[F\subset  E^{(p)}]$ is by definition independent of the rest of the summand, and that we have i.i.d. Bernoulli variables of parameter $p$, this identity takes the form
\begin{align}\label{eq:deverr-}
\Delta_T^{(p)}=\sum_{j=1}^\Gamma(-1)^{j+1}p^{j}\,\E\left[\sum_{|F|=j}(\nabla\phi_T+\xi)\cdot C_{F}(\nabla\phi_T^{E^{(p)}\cup F}+\xi)\right],
\end{align}
which can be further decomposed as
\begin{align*}
\Delta_T^{(p)}&=\sum_{j=1}^\Gamma(-1)^{j+1}p^{j}\,\E\left[\sum_{|F|=j}(\nabla\phi_T+\xi)\cdot C_{F}(\nabla\phi_T^{F}+\xi)\right]\\
&\qquad+\sum_{j=1}^\Gamma(-1)^{j+1}p^{j}\,\E\left[\sum_{|F|=j}(\nabla\phi_T+\xi)\cdot C_{F}\nabla(\phi_T^{E^{(p)}\cup F}-\phi_T^{F})\right],
\end{align*}
where the sums are still absolutely convergent by~\eqref{eq:convfinTS2} in Lemma~\ref{lem:convfinT}. The first term of the first sum (i.e. corresponding to the choice $j=1$) is of order $p$ and coincides with the argument of the limit in \eqref{eq:formderdisj0} for $k=1$. The second sum can be rewritten as a sum of errors of order at least $p^2$, which can then be combined with the corresponding (higher-order) terms in the first sum, and an induction argument finally allows us to prove the following decomposition:

\begin{lem}\label{lem:decompdiffdisj}
Assume that $\E[\rho(Q)^s]<\infty$ for all $s\ge1$. For any $k\ge0$ and any $p\in[0,1]$, we have
\begin{align}\label{eq:toprdecnondisj}
\Delta_T^{(p)}&=\sum_{j=1}^kp^j\Delta_T^j+\sum_{j=k+1}^{k+\Gamma}p^jE_T^{(p),j,k}
\end{align}
where, for all $j> k\ge0$, the approximate derivatives $\Delta_T^j$ and the 
errors $E_T^{(p),j,k}$ are given by
\begin{eqnarray}\label{eq:rewriteSkdisj1}
\Delta_T^j&:=&\sum_{|F|=j}\sum_{G\subset F}(-1)^{|F\setminus G|+1}\E\left[\nabla\delta_\xi^G\phi_T\cdot C_{F\setminus G\| G}(\nabla\phi_T^{F}+\xi)\right],
\\
\label{eq:rewriteSkdisj1err}
E_T^{(p),j,k}&:=&\sum_{|F|=j}\sum_{G\subset F\atop |G|\le k}(-1)^{|F\setminus G|+1}\E\left[\nabla\delta_\xi^G\phi_T\cdot C_{F\setminus G\| G}(\nabla\phi_T^{E^{(p)}\cup F}+\xi)\right],
\end{eqnarray}
and the sums $\sum_{|F|=j}\sum_{G\subset F}$ in~\eqref{eq:rewriteSkdisj1} and~\eqref{eq:rewriteSkdisj1err} are absolutely convergent for fixed $T$.\qed
\end{lem}

\begin{proof}
We proceed by induction. For $k=0$, \eqref{eq:toprdecnondisj} reduces to~\eqref{eq:deverr-}.
Assume now that~\eqref{eq:toprdecnondisj} holds true for some $k\ge0$.

First of all,  we decompose $E_T^{(p),k+1,k}$ as follows:
\begin{align}
E_T^{(p),k+1,k}&=\Delta_T^{k+1}+G^{(p),k}_T,\label{eq:simpltermb}
\end{align}
where the error reads
\begin{align*}
G^{(p),k}_T:=~&\sum_{|F|=k+1}\sum_{G\subset F\atop |G|\le k}(-1)^{|F\setminus G|+1}\E\left[\nabla\delta_\xi^G\phi_T\cdot C_{F\setminus G\| G}\nabla(\phi_T^{E^{(p)}\cup F}-\phi_T^{ F})\right]\\
=~&\sum_{|F|=k+1}\sum_{j=1}^{k+1}(-1)^{j+1}\sum_{G\subset F\atop |G|=j}\E\left[\nabla\delta_\xi^{F\setminus G}\phi_T\cdot C_{G\|F\setminus G}\nabla(\phi_T^{E^{(p)}\cup F}-\phi_T^{ F})\right],
\end{align*}
since the summand for $G=F$ in \eqref{eq:rewriteSkdisj1} vanishes (cf. $C_{\varnothing \|F}\equiv 0$).

Given $|F|=k+1$, recall that \eqref{eq:satif1} in Lemma~\ref{lem:eqsatif} (for $H=\varnothing$) asserts that \textcolor{red}{$\delta^F_\xi\phi_T$} solves
\[\frac1T\delta_\xi^{F}\phi_T-\nabla\cdot A^{F}\nabla\delta_\xi^{F}\phi_T=\sum_{j=1}^{k+1}(-1)^{j+1}\sum_{G\subset F\atop |G|=j}\nabla\cdot C_{G\|F\setminus G}\nabla\delta_\xi^{F\setminus G}\phi_T,\]
and also recall that since 
$A^{E^{(p)}\cup F}=A^F+C^{(p)}_{\|F}$, $\phi_T^{E^{(p)}\cup F}-\phi_T^F$ solves
\[\frac1T(\phi_T^{E^{(p)}\cup F}-\phi_T^{F})-\nabla\cdot A^{F}\nabla(\phi_T^{E^{(p)}\cup F}-\phi_T^{F})=\nabla\cdot C^{(p)}_{\|F}(\nabla\phi_T^{E^{(p)}\cup F}+\xi).\]
Successively testing these equations with $\phi_T^{E^{(p)}\cup F}-\phi_T^F$ and $\delta^F\phi_T$ respectively (as for the proof of~\eqref{eq:REF-ERG} in Lemma~\ref{lem:aprioriproba}, still noting that all the sums converge absolutely by Lemmas~\ref{lem:finiteT} and~\ref{lem:convfinT}), we get
\begin{align*}
G_T^{(p),k}&=-\frac1T\sum_{|F|=k+1}\E\left[\delta_\xi^F\phi_T(\phi_T^{E^{(p)}\cup F}-\phi_T^F)\right]-\sum_{|F|=k+1}\E\left[\nabla\delta_\xi^F\phi_T\cdot A^F\nabla(\phi_T^{E^{(p)}\cup F}-\phi_T^F)\right]\\
&=\sum_{|F|=k+1}\E\left[\nabla\delta_\xi^F\phi_T\cdot C^{(p)}_{\|F}(\nabla\phi_T^{E^{(p)}\cup F}+\xi)\right].
\end{align*}
Hence, using the inclusion-exclusion formula~\eqref{eq:excl2} as before (cf.~\eqref{eq:deverr-}) and the independence, this yields
\begin{align*}
G_T^{(p),k}&=\sum_{j=1}^\Gamma(-1)^{j+1}p^j\sum_{|G|=j}\sum_{|F|=k+1\atop G\cap F=\varnothing}\E\left[\nabla\delta_\xi^F\phi_T\cdot C_{G\|F}(\nabla\phi_T^{E^{(p)}\cup F\cup G}+\xi)\right],
\end{align*}
and hence, renumbering the sums,
\begin{align}\label{eq:calerrp}
p^{k+1}G_T^{(p),k}&=\sum_{j=k+2}^{k+\Gamma+1}(-1)^{j-k}p^{j}\sum_{|F|=j}\sum_{G\subset F\atop |G|=k+1}\E\left[\nabla\delta_\xi^{G}\phi_T\cdot C_{F\setminus G\|G}(\nabla\phi_T^{E^{(p)}\cup F}+\xi)\right].
\end{align}
By the induction assumption~\eqref{eq:toprdecnondisj} at order $k$ and the decomposition~\eqref{eq:simpltermb}, we thus have
\begin{align*}
\Delta_T^{(p)}&=\sum_{j=1}^{k+1}p^j\Delta_T^j+p^{k+1} G^{(p),k}_T+\sum_{j=k+2}^{k+\Gamma}p^j\sum_{|F|=j}\sum_{G\subset F\atop |G|\le k}(-1)^{|F\setminus G|+1}\E\left[\nabla\delta_\xi^G\phi_T\cdot C_{F\setminus G\| G}(\nabla\phi_T^{E^{(p)}\cup F}+\xi)\right].
\end{align*}
Combined with~\eqref{eq:calerrp}, this yields
\begin{align*}
\Delta_T^{(p)}&=\sum_{j=1}^{k+1}p^j\Delta_T^j+\sum_{j=k+2}^{k+\Gamma+1}p^j\sum_{|F|=j}\sum_{G\subset F\atop |G|\le k+1}(-1)^{|G|+1}\E\left[\nabla\delta_\xi^G\phi_T\cdot C_{F\setminus G\| G}(\nabla\phi_T^{E^{(p)}\cup F}+\xi)\right],
\end{align*}
that is $\Delta_T^{(p)}=\sum_{j=1}^{k+1}p^j\Delta_T^j+\sum_{j=k+2}^{k+\Gamma+1}p^jE_T^{(p),j,k+1}$, and therefore~\eqref{eq:toprdecnondisj} at step $k+1$.
\end{proof}

We now prove that the approximate derivatives are bounded uniformly in $T$ and in the moments of $\rho$, as a consequence of the improved energy estimates of Proposition~\ref{prop:apimproved}.

\begin{prop}\label{prop:apriori}
Assume that $\E[\rho(Q)^s]<\infty$ for all $s\ge1$. Then, there is a constant $C\sim1$ (independent of $T$ and of the moments of $\rho$) such that, for any $k\ge1$, the approximate $k$-th derivative $\Delta_T^k$ defined in \eqref{eq:rewriteSkdisj1} satisfies
\begin{equation}\label{eq:apriori-unif-k}
|\Delta_T^k|\,\le\, C^k.
\end{equation}
Likewise, for any $j> k\ge1$ and any $p\in[0,1]$, the error $E_T^{(p),j,k}$ defined in~\eqref{eq:rewriteSkdisj1err} satisfies
\begin{equation}\label{eq:apriori-unif-k-err}
|E_T^{(p),j,k}|\,\le\, C^j.
\end{equation}
\qed
\end{prop}

\begin{proof}
The estimates of the errors $E_T^{(p),j,k}$'s are obtained using the same arguments as for the estimates of the approximate derivatives $\Delta_T^j$'s, and we only display the proof of the latter.
Since $\nabla\phi_T^F+\xi=\sum_{S\subset F}\nabla\delta_\xi^{S}\phi_T$  (cf. \eqref{eq:defdeltaxi1} with $G=H=\varnothing$) and $C_{F\setminus G\|G}=C_{F\setminus G}+\sum_{U\subset G,U\ne\varnothing}(-1)^{|U|}C_{U\cup (F\setminus G)}$ for any $G\subsetneq F$ (cf. \eqref{eq:excl3.3}), and
$C_{\varnothing\|G}\equiv 0$, we may rewrite formula~\eqref{eq:rewriteSkdisj1} as follows:
\begin{align*}
\Delta_T^k&=\underbrace{\sum_{|F|=k}\sum_{G\subsetneq F}\sum_{S\subset F}(-1)^{|F\setminus G|+1}\E[\nabla\delta_\xi^G\phi_T\cdot C_{F\setminus G}\nabla\delta_\xi^S\phi_T]}_{=:\Delta^k_{T,1}}\\
&\hspace{3cm}+\underbrace{\sum_{|F|=k}\sum_{G\subsetneq F}\sum_{S\subset F}\sum_{U\subset G\atop U\ne\varnothing}(-1)^{|F\setminus G|+|U|+1}\E[\nabla\delta_\xi^G\phi_T\cdot C_{U\cup (F\setminus G)}\nabla\delta_\xi^S\phi_T]}_{=:\Delta^k_{T,2}}.
\end{align*}
We treat each term separately.
By the change of variables $G\leadsto G\cup U$, $S\leadsto S\cup U$, and $F\leadsto F\cup G\cup S\cup U$ (with $F,G,S,U$ disjoint), we rewrite $\Delta_{T,1}^k$ as
\begin{align*}
\Delta^k_{T,1}=\sum_{l=0}^{k-1}\sum_{i=0}^{k-l}\sum_{j=0}^l\sum_{|F|=k-l-i}\sum_{|G|=l-j\atop G\cap F=\varnothing}\sum_{|S|=i\atop S\cap (F\cup G)=\varnothing}\sum_{|U|=j\atop U\cap (F\cup G\cup S)=\varnothing}(-1)^{|F|+|S|+1}\E[\nabla\delta_\xi^{G\cup U}\phi_T\cdot C_{F\cup S}\nabla\delta_\xi^{S\cup U}\phi_T],
\end{align*}
so that, by the triangle inequality,
\begin{align*}
|\Delta^k_{T,1}|\lesssim\sum_{l=0}^{k-1}\sum_{i=0}^{k-l}\sum_{j=0}^l\sum_{|F|=k-l-i}\sum_{|S|=i\atop S\cap F=\varnothing}\sum_{|U|=j\atop U\cap (F\cup S)=\varnothing}\E\bigg[ \mathds1_{J_{F\cup S}}|\nabla\delta_\xi^{S\cup U}\phi_T|\bigg|\sum_{|G|=l-j\atop G\cap(F\cup S\cup U)=\varnothing}\nabla\delta_\xi^{G\cup U}\phi_T\bigg|\bigg].
\end{align*}
Recall from \eqref{eq:claimcombi+} in the proof of Proposition~\ref{prop:apimproved} that
\begin{align*}
\bigg|\sum_{|G|=l-j\atop G\cap(F\cup S\cup U)=\varnothing}\nabla\delta_\xi^{G\cup U}\phi_T\bigg|&\le\sum_{u=0}^{|F|}\sum_{s=0}^{|S|}\sum_{|W|=u\atop W\subset F}\sum_{|H|=s\atop H\subset S}\bigg|\sum_{|G|=l-j-u-s\atop G\cap(W\cup H\cup U)=\varnothing}\nabla\delta_\xi^{G\cup W\cup H\cup U}\phi_T\bigg|.
\end{align*}
Hence, by the change of variables $F\leadsto F\setminus W$ and $S\leadsto S\setminus H$,
and the notation $\delta_{F,S,U,W,H}=1$ if $F,S,U,W,H$ are disjoint, and $\delta_{F,S,U,W,H}=0$ otherwise, this yields
\begin{align*}
|\Delta^k_{T,1}|&\lesssim\sum_{l=0}^{k-1}\sum_{i=0}^{k-l}\sum_{j=0}^l\sum_{u=0}^{k-l-i}\sum_{s=0}^{i}\sum_{|F|=k-l-i-u}\sum_{|S|=i-s}\sum_{|U|=j}\sum_{|W|=u}\sum_{|H|=s}\delta_{F,S,U,W,H}\\
&\hspace{2cm}\times\E\bigg[ \mathds1_{J_{F\cup W\cup S\cup H}}|\nabla\delta_\xi^{S\cup H\cup U}\phi_T|\bigg|\sum_{|G|=l-j-u-s\atop G\cap(W\cup H\cup U)=\varnothing}\nabla\delta_\xi^{G\cup W\cup H\cup U}\phi_T\bigg|\bigg].
\end{align*}
We rearrange the sums suitably, and use the notation $\mathds1_{J_\varnothing}=1$ (so that we have $\mathds1_{J_{L\cup K}}=\mathds1_{J_L}\mathds1_{J_K}$) to obtain
\begin{align*}
|\Delta^k_{T,1}|&\lesssim\sum_{l=0}^{k-1}\sum_{i=0}^{k-l}\sum_{j=0}^l\sum_{u=0}^{k-l-i}\sum_{s=0}^{i}\sum_{|U|=j}\sum_{|H|=s\atop H\cap U=\varnothing}\E\bigg[ \mathds1_{J_H}\bigg(\sum_{|F|=k-l-i-u}\mathds1_{J_{F}}\bigg)\bigg(\sum_{|S|=i-s\atop S\cap(H\cup U)=\varnothing}\mathds1_{J_S}|\nabla\delta_\xi^{S\cup H\cup U}\phi_T|\bigg)\\
&\hspace{6cm}\times\bigg(\sum_{|W|=u\atop W\cap(H\cup U)=\varnothing}\mathds1_{J_{W}}\bigg|\sum_{|G|=l-j-u-s\atop G\cap(W\cup H\cup U)=\varnothing}\nabla\delta_\xi^{G\cup W\cup H\cup U}\phi_T\bigg|\bigg)\bigg].
\end{align*}
Recalling that $\sum_{L\subset\N}\mathds1_{J_L}(0)\lesssim1$ by~\eqref{eq:firstboundsumchi+} (as a consequence of assumption~\eqref{eq:boundrho}), we deduce from a multiple use of the Cauchy-Schwarz inequality
\begin{align*}
|\Delta^k_{T,1}|&\lesssim\sum_{l=0}^{k-1}\sum_{i=0}^{k-l}\sum_{j=0}^l\sum_{u=0}^{k-l-i}\sum_{s=0}^{i}\sum_{|U|=j}\sum_{|H|=s\atop H\cap U=\varnothing}\E\bigg[\mathds1_{J_H}\bigg(\sum_{|S|=i-s\atop S\cap(H\cup U)=\varnothing}\mathds1_{J_S}|\nabla\delta_\xi^{S\cup H\cup U}\phi_T|^2\bigg)^{\frac12}\\
&\hspace{6cm}\times\bigg(\sum_{|W|=u\atop W\cap(H\cup U)=\varnothing}\mathds1_{J_W}\bigg|\sum_{|G|=l-j-u-s\atop G\cap(W\cup H\cup U)=\varnothing}\nabla\delta_\xi^{G\cup W\cup H\cup U}\phi_T\bigg|^2\bigg)^{\frac12}\bigg],
\end{align*}
and hence, by the Jensen inequality,
\begin{align*}
|\Delta^k_{T,1}|&\lesssim\sum_{l=0}^{k-1}\sum_{i=0}^{k-l}\sum_{j=0}^l\sum_{u=0}^{k-l-i}\sum_{s=0}^{i}\E\bigg[\sum_{|U|=j}\sum_{|H|=s\atop H\cap U=\varnothing}\mathds1_{J_H}\sum_{|S|=i-s\atop S\cap(H\cup U)=\varnothing}\mathds1_{J_S}|\nabla\delta_\xi^{S\cup H\cup U}\phi_T|^2\bigg]\\
&\quad+\sum_{l=0}^{k-1}\sum_{i=0}^{k-l}\sum_{j=0}^l\sum_{u=0}^{k-l-i}\sum_{s=0}^{i}\E\bigg[\sum_{|U|=j}\sum_{|H|=s\atop H\cap U=\varnothing}\mathds1_{J_H}\sum_{|W|=u\atop W\cap(H\cup U)=\varnothing}\mathds1_{J_W}\bigg|\sum_{|G|=l-j-u-s\atop G\cap(W\cup H\cup U)=\varnothing}\nabla\delta_\xi^{G\cup W\cup H\cup U}\phi_T\bigg|^2\bigg].
\end{align*}
By the changes of variables $S\cup H\cup U\leadsto U$ in the first term and $W\cup H\cup U\leadsto U$ in the second term, and using that 
$\sum_{|H|\le k}\sum_{|W|\le k}\mathds1_{J_H}\mathds1_{J_W}\lesssim 1$, this finally yields
\begin{align}\label{eq:bounderrorimproveest}
|\Delta^k_{T,1}|&\lesssim \sum_{j=0}^{k}\E\bigg[\sum_{|U|=j}|\nabla\delta_\xi^{U}\phi_T|^2\bigg]+\sum_{j=0}^{k-1}\sum_{i=0}^{j}\E\bigg[\sum_{|U|=j-i}\bigg|\sum_{|G|=i\atop G\cap U=\varnothing}\nabla\delta_\xi^{G\cup U}\phi_T\bigg|^2\bigg].
\end{align}
The improved energy estimates of Proposition~\ref{prop:apimproved} then allow us to conclude that $|\Delta_{T,1}^k|\lesssim C^{k}$ for some $C\sim1$. As we can easily argue in a similar way for $\Delta_{T,2}^k$, the conclusion follows.
\end{proof}

The combination of Lemma~\ref{lem:decompdiffdisj} and Proposition~\ref{prop:apriori} immediately yields the following result:

\begin{cor}\label{cor:reslimT}
Assume that $\E[\rho(Q)^s]<\infty$ for all $s\ge1$. Then, there exists a constant $C\sim1$ (independent of $T$ and of the moments of $\rho$) such that, for any $k\ge0$ and any $p\in[0,1]$, we have
\begin{align}\label{eq:toprdecnondisjda}
\bigg|\Delta_T^{(p)}-\sum_{j=1}^kp^j\Delta_T^j\bigg|\le (Cp)^{k+1}.
\end{align}
\qed
\end{cor}

The following lemma provides useful alternative formulas for the approximate derivatives $\Delta_T^j$'s (which coincide with the argument of the limit in~\eqref{eq:formderdisj0} for $p_0=0$), showing that they coincide with the arguments of the limits in~\eqref{eq:formderdisj1} and~\eqref{eq:formderdisj2} for $p_0=0$.

\begin{lem}\label{lem:alter}
Assume that $\E[\rho(Q)^s]<\infty$ for all $s\ge1$. For all $T>0$, the approximate derivatives $\Delta_T^j$'s, $j\ge1$, given by~\eqref{eq:rewriteSkdisj1}, satisfy the following two equivalent formulas:
\begin{align}
\Delta_T^j&=\sum_{|F|=j}\sum_{G\subset F}(-1)^{|G|+1}\E\left[(\nabla\phi_T+\xi)\cdot C_G\nabla\delta_\xi^{F\setminus G}\phi_T^{G}\right]\label{eq:rewriteSkdisj2}\\
&=\sum_{|F|=j}\sum_{G\subset F}(-1)^{|F\setminus G|}\E[\xi\cdot A^{F\setminus G}(\nabla\phi_T^{G}+\xi)]\label{eq:rewriteSkdisj3},
\end{align}
where both sums $\sum_{|F|=j}$ are absolutely convergent.\qed
\end{lem}
Before we turn to the proof of this lemma, let us comment on the equivalent formulas~\eqref{eq:rewriteSkdisj1}, \eqref{eq:rewriteSkdisj2} and~\eqref{eq:rewriteSkdisj3}. Formula~\eqref{eq:rewriteSkdisj1} is the natural formula that we obtain by expanding the difference quotient (see proof of~\eqref{eq:deverr-} and of Lemma~\ref{lem:decompdiffdisj}), formula~\eqref{eq:rewriteSkdisj2} is the easiest to use in practice (see e.g. Corollaries~\ref{cor:cm} and~\ref{cor:cm2}), while formula~\eqref{eq:rewriteSkdisj3} is the cluster-expansion formula used by physicists.

\begin{proof}
We split the proof into two steps. We first prove~\eqref{eq:rewriteSkdisj2}, from which~\eqref{eq:rewriteSkdisj3} is an easy consequence.

\medskip

\step{1} Proof of \eqref{eq:rewriteSkdisj2}.\\
All absolute convergence issues that we need here (for fixed $T$) simply follow as before from Lemma~\ref{lem:convfinT} or similar statements (based on Lemma~\ref{lem:finiteT}).
For the clarity of the exposition, we discard this issue in the proof. 
Let $j\ge1$ be fixed. Separating the cases $G=\varnothing$ and $G\ne\varnothing$, and noting that $C_{F\setminus G\|G}$ vanishes whenever $G=F$, the very definition~\eqref{eq:rewriteSkdisj1} of $\Delta_T^j$ reads
\begin{align*}
\Delta_T^j&=\sum_{|F|=j}\sum_{ G\subsetneq F\atop G\ne\varnothing}(-1)^{|F\setminus G|+1}\E\left[\nabla\delta_\xi^G\phi_T\cdot C_{F\setminus G\| G}(\nabla\phi_T^{F}+\xi)\right]\\
&\qquad+(-1)^{j+1}\sum_{|F|=j}\E\left[(\nabla\phi_T+\xi)\cdot C_{F}(\nabla\phi_T^{F}+\xi)\right].
\end{align*}
For any $|F|=j$, $G\subsetneq F$, $G\ne\varnothing$, 
by \eqref{eq:satif2} in Lemma~\ref{lem:eqsatif}, $\delta_\xi^{F\setminus G} \phi_T^G$ satisfies 
\[\frac1T\delta_\xi^{F\setminus G}\phi_T^G-\nabla\cdot A^G\nabla\delta_\xi^{F\setminus G}\phi_T^G=\sum_{S\subset F\setminus G}(-1)^{|S|+1}\nabla\cdot C_{S\|G}\nabla\delta_\xi^{F\setminus (G\cup S)}\phi_T^{G\cup S}.\]
Testing this equation with $\delta_\xi^G\phi_T$  (as in the proof of~\eqref{eq:REF-ERG} in Lemma~\ref{lem:aprioriproba}) yields
\begin{align*}
\Delta_T^j&=-\frac1T\sum_{|F|=j}\sum_{G\subsetneq F\atop G\ne\varnothing}\E\left[\delta_\xi^G\phi_T\delta_\xi^{F\setminus G}\phi_T^G\right]-\sum_{|F|=j}\sum_{G\subsetneq F\atop G\ne\varnothing}\E\left[\nabla\delta_\xi^G\phi_T\cdot A^G\nabla\delta_\xi^{F\setminus G}\phi_T^G\right]\\
&\qquad+\sum_{|F|=j}\sum_{G\subsetneq F\atop G\ne\varnothing}\sum_{S\subsetneq F\setminus G}(-1)^{|S|}\E\left[\nabla\delta_\xi^G\phi_T\cdot C_{S\|G}\nabla\delta_\xi^{F\setminus (G\cup S)}\phi_T^{G\cup S}\right]\\
&\qquad+(-1)^{j+1}\sum_{|F|=j}\E\left[(\nabla\phi_T+\xi)\cdot C_{F}(\nabla\phi_T^{F}+\xi)\right].
\end{align*}
Now, for $G\ne\varnothing$, by \eqref{eq:satif1} in Lemma~\ref{lem:eqsatif} (with $H=\varnothing$), 
$\delta_\xi^G\phi_T$ solves
\[\frac1T\delta_\xi^G\phi_T-\nabla\cdot A^G\nabla\delta_\xi^G\phi_T=\sum_{S\subset G}(-1)^{|S|+1}\nabla\cdot C_{S\|G\setminus S}\nabla\delta_\xi^{G\setminus S}\phi_T.\]
Testing this equation with $\delta_\xi^{F\setminus G}\phi_T^G$ yields
\begin{align*}
\Delta_T^j&=-\sum_{|F|=j}\sum_{G\subsetneq F\atop G\ne\varnothing}\sum_{S\subset G}(-1)^{|S|}\E\left[\nabla\delta_\xi^{G\setminus S}\phi_T\cdot C_{S\|G\setminus S}\nabla\delta_\xi^{F\setminus G}\phi_T^G\right]\\
&\qquad+\sum_{|F|=j}\sum_{G\subsetneq F\atop G\ne\varnothing}\sum_{S\subsetneq F\setminus G}(-1)^{|S|}\E\left[\nabla\delta_\xi^G\phi_T\cdot C_{S\|G}\nabla\delta_\xi^{F\setminus (G\cup S)}\phi_T^{G\cup S}\right]\\
&\qquad+(-1)^{j+1}\sum_{|F|=j}\E\left[(\nabla\phi_T+\xi)\cdot C_{F}(\nabla\phi_T^{F}+\xi)\right],
\end{align*}
and therefore
\begin{align*}
\Delta_T^j&=-\sum_{|F|=j}\sum_{G\subsetneq F}\sum_{S\subsetneq G}(-1)^{|S|}\E\left[\nabla\delta_\xi^{G\setminus S}\phi_T\cdot C_{S\|G\setminus S}\nabla\delta_\xi^{F\setminus G}\phi_T^G\right]\\
&\qquad+\sum_{|F|=j}\sum_{G\subsetneq F\atop{G\ne\varnothing}}\sum_{S\subsetneq F\setminus G}(-1)^{|S|}\E\left[\nabla\delta_\xi^G\phi_T\cdot C_{S\|G}\nabla\delta_\xi^{F\setminus (G\cup S)}\phi_T^{G\cup S}\right]\\
&\qquad+\sum_{|F|=j}\sum_{G\subset F}(-1)^{|G|+1}\E\left[(\nabla\phi_T+\xi)\cdot C_{G}\nabla\delta_\xi^{F\setminus G}\phi_T^G\right].
\end{align*}
With the change of variables $G\leadsto G\setminus S$
in the first term, 
we observe that the first two groups of sums cancel, so that we are left with
\begin{align*}
\Delta_T^j&=\sum_{|F|=j}\sum_{G\subset F}(-1)^{|G|+1}\E\left[(\nabla\phi_T+\xi)\cdot C_{G}\nabla\delta_\xi^{F\setminus G}\phi_T^G\right],
\end{align*}
that is,~\eqref{eq:rewriteSkdisj2}.

\medskip

\step{2} Proof of \eqref{eq:rewriteSkdisj3}.\\
Absolute convergence issues for this part of the proof
(which do not straightforwardly follow from Lemmas~\ref{lem:convfinT} and ~\ref{lem:finiteT}) will be addressed at the end of this step.
Let $j\ge1$ be fixed. Formula~\eqref{eq:rewriteSkdisj2} gives
\begin{align}\label{eq:defverSTk12}
\Delta_T^j&=\underbrace{\sum_{|F|=j}\sum_{G\subset F}(-1)^{|G|+1}\E\left[\xi\cdot C_G\nabla\delta_\xi^{F\setminus G}\phi_T^G\right]}_{=:S_T^{j,1}}+\underbrace{\sum_{|F|=j}\sum_{G\subset F}(-1)^{|G|+1}\E\left[\nabla\phi_T\cdot C_G\nabla\delta_\xi^{F\setminus G}\phi_T^G\right]}_{=:S_T^{j,2}}.
\end{align}
By  \eqref{eq:satif2} in Lemma~\ref{lem:eqsatif}  (with $H=\varnothing$), $\delta_\xi^F\phi_T$ solves
\[\frac1T\delta_\xi^F\phi_T-\nabla\cdot A\nabla\delta_\xi^F\phi_T=\sum_{G\subset F}(-1)^{|G|+1}\nabla\cdot C_G\nabla\delta_\xi^{F\setminus G}\phi_T^G,\]
whereas $\phi_T$ solves
\[\frac1T\phi_T-\nabla\cdot A(\nabla\phi_T+\xi)=0.\]
On the one hand, testing these equations with $\phi_T$ and $\delta_\xi^F\phi_T$ respectively (as in the proof of~\eqref{eq:REF-ERG} in Lemma~\ref{lem:aprioriproba}), we obtain
\begin{align}
S_T^{j,2}&=-\frac1T\sum_{|F|=j}\E\left[\phi_T\delta_\xi^F\phi_T\right]-\sum_{|F|=j}\E\left[\nabla\phi_T\cdot A\nabla\delta_\xi^F\phi_T\right]\nonumber\\
&=\sum_{|F|=j}\E\left[(\nabla\phi_T+\xi)\cdot A\nabla\delta_\xi^F\phi_T\right]-\sum_{|F|=j}\E\left[\nabla\phi_T\cdot A\nabla\delta_\xi^F\phi_T\right]\nonumber\\
&=\sum_{|F|=j}\E\left[\xi\cdot A\nabla\delta_\xi^F\phi_T\right],\label{eq:sk2rewrpl+}
\end{align}
and therefore
\begin{align}\label{eq:sk2rewrpl}
S_T^{j,2}=\sum_{|F|=j}\sum_{G\subset F}(-1)^{|F\setminus G|}\E\left[\xi\cdot A(\nabla\phi_T^G+\xi)\right].
\end{align}
On the other hand, $S_T^{j,1}$ can be rewritten as follows:
\begin{align*}
S_T^{j,1}&=\sum_{|F|=j}\sum_{G\subset F}(-1)^{|G|+1}\sum_{S\subset F\setminus G}(-1)^{F\setminus (S\cup G)}\E\left[\xi\cdot C_G(\nabla\phi_T^{S\cup G}+\xi)\right],
\end{align*}
which yields by the change of variables $S\cup G\leadsto U$
\begin{align}
S_T^{j,1}&=\sum_{|F|=j}\sum_{U\subset F}(-1)^{|F\setminus U|}\E\left[\xi\cdot \left(\sum_{G\subset U}(-1)^{|G|+1}C_G\right)(\nabla\phi_T^{U}+\xi)\right]\nonumber\\
&=\sum_{|F|=j}\sum_{U\subset F}(-1)^{|F\setminus U|}\E\left[\xi\cdot C^U(\nabla\phi_T^{U}+\xi)\right].\label{eq:stk1equiv}
\end{align}
The desired result  follows from the combination of \eqref{eq:defverSTk12}, \eqref{eq:sk2rewrpl}, and \eqref{eq:stk1equiv}.
Note that the sum defining $S_T^{j,1}$ in~\eqref{eq:defverSTk12} is absolutely convergent by virtue of~\eqref{eq:convfinTS+} in Lemma~\ref{lem:convfinT}, and hence the sum $\sum_{|F|=j}$ in~\eqref{eq:stk1equiv} is also absolutely convergent, since its terms have just been rewritten but are still the same. Likewise, the sum in the right-hand side of~\eqref{eq:sk2rewrpl+} is absolutely convergent by Lemma~\ref{lem:finiteT} (thus justifying the testing argument), so that the sum in~\eqref{eq:sk2rewrpl} must also converge absolutely. This finally proves that the sum $\sum_{|F|=j}$ in~\eqref{eq:rewriteSkdisj3} is absolutely convergent too (which would not be clear a priori without performing this decomposition).
\end{proof}

\subsection{Proof of Theorem~\ref{th:analytic} and Corollary~\ref{cor:analytic}}\label{chap:proofth1}

Let $\xi\in\R^d$, $|\xi|=1$ be fixed. It suffices to prove Theorem~\ref{th:analytic} and Corollary~\ref{cor:analytic} for that fixed choice of $\xi$. What needs to be done is to pass to the limit $T\uparrow\infty$ in Corollary~\ref{cor:reslimT}, and get rid of the additional assumption that $\E[\rho(Q)^s]<\infty$ for all $s\ge1$. For that second purpose, given a point process $\rho$, we introduce approximations for which all moments exist: more precisely, we shall construct hardcore approximations $\rho_\theta$ of $\rho$, apply Corollary~\ref{cor:reslimT} for these approximations, and then pass to the limit in both the parameters $T$ and $\theta$. We split the proof into five steps.

\medskip

\step1 Hardcore approximations of $\rho$.\\
Let $\theta>0$ be fixed. In this first step, we construct hardcore approximations $\rho_\theta$ of the stationary point process $\rho$ in the following sense: for any $\theta>0$, $\rho_\theta$ is an ergodic stationary point process on $\R^d$ such that $\rho_\theta\subset\rho$ and $\rho_\theta(Q)\le\theta$ a.s., and moreover $\rho_\theta\uparrow\rho$ locally almost surely as $\theta\uparrow\infty$. For any $\theta>0$, we choose a measurable enumeration $\rho_\theta=(q_n^\theta)_n$. We then define $A_\theta^F$ as the coefficients obtained when replacing $\rho$ by $\rho_\theta$ in $A^F$. Similarly, we define $\phi^{F}_{T,\theta}$ the approximate corrector and $A^{(p)}_{T,\theta}$ the approximate homogenized coefficients associated with $A_\theta^F,A_\theta^{(p)}$ instead of $A^F,A^{(p)}$. We then also prove the following convergence properties, which will be crucial in the next step: for fixed $p\in[0,1]$ and $T>0$, we have
\begin{align}\label{eq:convphitheta}
\E[|\nabla(\phi_{T,\theta}^{(p)}-\phi_{T}^{(p)})|^2]\xrightarrow{\theta\uparrow\infty}0,
\end{align}
and therefore
\begin{align}\label{eq:convatheta}
|A^{(p)}_{T,\theta}-A^{(p)}_T|\xrightarrow{\theta\uparrow\infty}0.
\end{align}

We first give a possible construction of such an approximating sequence $(\rho_\theta)_\theta$. Consider the measurable enumeration $\rho=(q_n)_n$, choose independently a sequence $(U_n)_n$ of i.i.d. random variables that are uniformly distributed on $(0,1)$, and consider the decorated process $(q_n,U_n)_n$. We then build an oriented graph on the points $(q_n,U_n)_n$ in $\R^d\times[0,1]$ as follows: we put an oriented edge from $(q,u)$ to $(q',u')$ whenever $(q+\frac1\theta Q)\cap (q'+\frac1\theta Q)\ne\varnothing$ and $u<u'$ (or $u=u'$ and $q$ precedes $q'$ in the lexicographic order, say). We say that $(q',u')$ is an offspring (resp. a descendant) of $(q,u)$ if $(q,u)$ is a direct ancestor (resp. an ancestor) of $(q',u')$, i.e. if there is an edge (resp. a directed path) from $(q,u)$ to $(q',u')$ in the oriented graph constructed above. We now construct $\rho_\theta$ as follows. Let $F_1$ be the set of all roots in the oriented graph (i.e. the points of $\Pc_0$ without ancestor), let $G_1$ be the set of points of $\Pc_0$ that are offsprings of points of $F_1$, and let $H_1=F_1\cup G_1$. Now consider the oriented graph induced on $(q_n,U_n)_n\setminus H_1$, and define $F_2,G_2,H_2$ in the same way, and so on. By construction, the sets $F_i$ and $G_i$ are all disjoint and constitute a partition of the collection $(q_n,U_n)_n$. Finally define $\rho_\theta:=\pi_1(\bigcup_iF_i)$, where $\pi_1$ is the projection on the first factor, $\pi_1(q,u)=q$. We easily check that $\rho_\theta$ defines a stationary point process on $\R^d$ and satisfies the required properties. Ergodicity of $\rho_\theta$ easily follows from that of $\rho$ exactly in the same way as for the random parking measure in~\cite[Step~4 of the proof of Proposition~2.1]{Gloria-Penrose-13}.

It only remains to prove the convergence property~\eqref{eq:convphitheta}. For that purpose, we write the equation satisfied by the difference $\phi_{T,\theta}^{(p)}-\phi_{T}^{(p)}$:
\[\frac1T(\phi_{T,\theta}^{(p)}-\phi_{T}^{(p)})-\nabla\cdot A_\theta^{(p)}\nabla(\phi_{T,\theta}^{(p)}-\phi_{T}^{(p)})=\nabla\cdot (A^{(p)}_\theta-A^{(p)})(\nabla\phi_T^{(p)}+\xi).\]
Testing this equation in probability with $\phi_{T,\theta}^{(p)}-\phi_{T}^{(p)}$ itself yields
\begin{align*}
\E[|\nabla(\phi_{T,\theta}^{(p)}-\phi_{T}^{(p)})|^2]&\lesssim \E[|A^{(p)}_\theta-A^{(p)}|^2\,(|\nabla\phi_T^{(p)}|^2+1)].
\end{align*}
By assumption, $A(0)$ and $A'(0)$ only depend on $\rho$ via the restriction $\rho|_{B_r}$ for some given $r>0$, so that the same property holds by definition for $A^{(p)}(0)$.
Hence, for some $L>0$,
\begin{align*}
\E[|\nabla(\phi_{T,\theta}^{(p)}-\phi_{T}^{(p)})|^2]&\lesssim \E[\mathds1_{\rho_\theta|_{B_L}\ne\rho|_{B_L}}\,(|\nabla\phi_T^{(p)}|^2+1)].
\end{align*}
Now the desired result simply follows from dominated convergence and the basic energy estimate $\E[|\nabla\phi_T^{(p)}|^2+1]\lesssim1$, recalling that, by definition, we have, almost surely as $\theta\uparrow\infty$,
\[\mathds1_{\rho_\theta|_{B_L}\ne\rho|_{B_L}}\to0.\]

\medskip

\step{2} Reduction by regularization.\\
In this step, we prove Theorem~\ref{th:analytic} and Corollary~\ref{cor:analytic} provided we have that, for fixed  $T$ and under the additional assumption $\E[\rho(Q)^s]<\infty$ for all $s\ge1$,  the map $p\mapsto \xi\cdot A^{(p)}_{T}\xi$ satisfies, for any $p_0\in[0,1]$ and any $k\ge1$, for all $-p_0\le p\le1-p_0$, $|p|\le 1/C_{p_0}$,
\begin{align}\label{eq:devtaylorTfix}
\bigg|\xi\cdot A^{(p_0+p)}_{T}\xi-\xi\cdot A^{(p_0)}_{T}\xi-\sum_{j=1}^kp^j\Delta_{T}^{(p_0),j}\bigg|\le (pC_{p_0})^{k+1},
\end{align}
for some constant $C_{p_0}\sim_{p_0}1$, where the $\Delta_{T}^{(p_0),j}$'s are equivalently given by the arguments of any of the limits~\eqref{eq:formderdisj0}, \eqref{eq:formderdisj1} and~\eqref{eq:formderdisj2}, and further satisfy the bounds $|\Delta_{T}^{(p_0),j}|\le C^j$ for all $j\ge1$ (uniformly in $T,p_0$ and the moments of $\rho$).

Let $p_0\in[0,1]$ be fixed. Consider the approximations $\rho_\theta$ introduced in Step~1, and apply~\eqref{eq:devtaylorTfix} with $\rho$ replaced by $\rho_\theta$ (where obviously all moments of $\rho_\theta$ are finite). For any $k\ge1$, it follows from~\eqref{eq:devtaylorTfix} that the map $p\mapsto \xi\cdot A_{T,\theta}^{(p)}\xi$ is smooth (on the whole of $[0,1]$), and a Taylor expansion of the map around $p_0$ up to order $k$ gives, by Lagrange's remainder theorem, for all $-p_0\le p\le 1-p_0$,
\begin{align}\label{eq:approxok}
\bigg|\xi\cdot A^{(p_0+p)}_{T,\theta}\xi-\xi\cdot A^{(p_0)}_{T,\theta}\xi-\sum_{j=1}^kp^j\Delta_{T,\theta}^{(p_0),j}\bigg|\le p^{k+1}\sup_{u\in[0,1]}|\Delta_{T,\theta}^{(p_0 +up),k+1}|\le (Cp)^{k+1}.
\end{align}
From~\eqref{eq:approxhomcoeff} and~\eqref{eq:convatheta}, we learn that
\begin{align}\label{eq:convniv0}
\lim_{T\uparrow\infty}\lim_{\theta\uparrow\infty}(\xi\cdot A^{(p_0+p)}_{T,\theta}\xi-\xi\cdot A^{(p_0)}_{T,\theta}\xi)=\lim_{T\uparrow\infty}(\xi\cdot A^{(p_0+p)}_{T}\xi-\xi\cdot A^{(p_0)}_{T}\xi)=\xi\cdot A^{(p_0+p)}_{\hom}\xi-\xi\cdot A^{(p_0)}_{\hom}\xi.
\end{align}
Hence, in order to pass to the limit $T,\theta\uparrow \infty$ in \eqref{eq:approxok}, it is enough to prove that the limits 
\begin{equation}\label{eq:Ant-3.1}
\Delta^{(p_0),j}:=\lim_{T\uparrow\infty}\lim_{\theta\uparrow\infty}\Delta_{T,\theta}^{(p_0),j}
\end{equation}
all exist in $\R$, for all $j\ge1$. 
The combination of \eqref{eq:convniv0} and \eqref{eq:Ant-3.1} indeed yields that for any $k\ge1$, for all $-p_0\le p\le1-p_0$, we have
\[\bigg|\xi\cdot A^{(p_0+p)}_{\hom}\xi-\xi\cdot A^{(p_0)}_{\hom}\xi-\sum_{j=1}^kp^j\Delta^{(p_0),j}\bigg|\le (Cp)^{k+1},\]
which is equivalent to the analyticity statement of Theorem~\ref{th:analytic} (with convergence of the Taylor series at $p_0$ for all perturbations $p$ of magnitude $|p|< 1/C$, $-p_0\le p\le1-p_0$), and the derivatives $j!\Delta^{(p_0),j}$'s are then given by the desired well-defined limits stated in Corollary~\ref{cor:analytic}. In the particular case when the process $\rho$ has all its moments finite,  the regularization in $\theta$ can be omitted (so that only the limit in $T$ remains). The proof of formula~\eqref{eq:der1form} in Corollary~\ref{cor:analytic} is postponed to Step~5.

We prove \eqref{eq:Ant-3.1} by induction. 
The proof of the statement for $j=1$ is similar to the proof of the induction step, and we only display the latter.
Assume that the limits $\Delta^{(p_0),j}=\lim_T\lim_\theta\Delta_{T,\theta}^{(p_0),j}$ exist in $\R$ for all $1\le j\le k$, for some $k\ge1$.
We shall then prove that the limit $\Delta^{(p_0),k+1}=\lim_T\lim_\theta\Delta_{T,\theta}^{(p_0),k+1}$ also exists in $\R$. As $\Delta^{(p_0),k+1}_{T,\theta}$ is bounded uniformly in $T,\theta$, it  converges to some limit $L_T^{(p_0)}\in\R$ as $\theta\uparrow\infty$ up to extraction. Passing to the limit $\theta\uparrow \infty$ along a subsequence in inequality~\eqref{eq:approxok} with $k$ replaced by $k+1$, and using the induction assumptions and~\eqref{eq:convniv0}, 
we obtain for any $-p_0\le p\le 1-p_0$,
\[\bigg|\xi\cdot A^{(p_0+p)}_{T}\xi-\xi\cdot A^{(p_0)}_{T}\xi-\sum_{j=1}^{k}p^j\lim_\theta\Delta_{T,\theta}^{(p_0),j}-p^{k+1}L_T^{(p_0)}\bigg|\le (Cp)^{k+2}.\]
This proves that $L^{(p_0)}$ satisfies
\[L_T^{(p_0)}=\lim_{p\to0\atop -p_0\le p\le 1-p_0}\bigg(p^{-k-1}(\xi\cdot A^{(p_0+p)}_{T}\xi-\xi\cdot A^{(p_0)}_{T}\xi)-\sum_{j=1}^kp^{j-k-1}\lim_\theta\Delta_{T,\theta}^{(p_0),j}\bigg),\]
where in particular the limit must exist. Since the right-hand side does not depend on the extraction, $L_T^{(p_0)}$ is uniquely defined, and $L_T^{(p_0)}=\lim_\theta\Delta^{(p),k+1}$ does exists in $\R$.
A similar argument for the limit in $T$ shows that $\lim_T\lim_\theta\Delta^{(p_0),k+1}_{T,\theta}$  exists in $\R$, so that \eqref{eq:Ant-3.1} is proved.

\medskip

\step{3} Reduction by restriction to $p_0=0$.\\
Let $T>0$ be fixed and assume that $\E[\rho(Q)^s]<\infty$ for all $s\ge1$.
In the present step, we prove that it suffices to check the result~\eqref{eq:devtaylorTfix} at $p_0=0$: more precisely, it suffices to show that the map $p\mapsto \xi\cdot A^{(p)}_T\xi$ satisfies, for all $k\ge1$ and $p\in[0,1]$,
\begin{align}\label{eq:resat0plop}
\bigg|\xi\cdot A^{(p)}_T\xi-\xi\cdot A_T\xi-\sum_{j=1}^kp^j\Delta_T^{j}\bigg|\le (Cp)^{k+1},
\end{align}
for some constant $C\sim1$, where, for any $j\ge1$, $\Delta^j_T$ is equivalently given by formulas~\eqref{eq:rewriteSkdisj1}, \eqref{eq:rewriteSkdisj2} and~\eqref{eq:rewriteSkdisj3}, and satisfies the bound $|\Delta_T^j|\le C^j$ for all $j\ge1$, uniformly in $T$ and the moments of $\rho$.

First consider $p_0\in[0,1)$ and positive perturbations $p_0+p$ with $p\ge0$. For $0\le p\le 1-p_0$, choose a sequence $(d_n^{(p_0,p)})_n$ of i.i.d. Bernoulli random variables (independent of all the others) with parameter $\p[d_n^{(p_0,p)}=1]=p/(1-p_0)$, and consider the twice perturbed coefficients
\begin{align*}
A^{(p_0,p)}&=A^{(p_0)}\mathds1_{\R^d\setminus J_n}+\sum_{n}\left(d_n^{(p_0,p)}A'+(1-d_n^{(p_0,p)})A^{(p_0)}\right)\mathds1_{J_n}\\
&=A\mathds1_{\R^d\setminus\bigcup_nJ_n}+\sum_{n}\left((1-d_n^{(p_0,p)})(1-b_n^{(p_0)})A+(d_n^{(p_0,p)}+b_n^{(p_0)}(1-d_n^{(p_0,p)}))A'\right)\mathds1_{J_n}.
\end{align*}
The field $A^{(p_0,p)}$ has by definition the same distribution as $A^{(p_0+p)}$, and it is a perturbation of $A^{(p_0)}$ with perturbation parameter $p/(1-p_0)$ (and with perturbed medium $A'$). Applying to $A^{(p_0,p)}$ the result~\eqref{eq:resat0plop} around $0$ (which is assumed to hold), we deduce that the map $p\mapsto \xi\cdot A^{(p_0,p)}_{T}\xi=\xi\cdot A^{(p_0+p)}_{T}\xi$ satisfies, for any $k\ge1$, for all $0\le p\le1-p_0$,
\begin{align}\label{eq:taylortransldr}
\bigg|\xi\cdot A^{(p_0+p)}_{T}\xi-\xi\cdot A_{T}^{(p_0)}\xi-\sum_{j=1}^k\frac{p^j}{(1-p_0)^j}\tilde\Delta^{(p_0),j}_T\bigg|\le\bigg(\frac{C p}{1-p_0}\bigg)^{k+1},
\end{align}
where, for any $j\ge1$, $\tilde\Delta_T^{(p_0),j}$ is the $j$-th right-derivative at $0$ of the map $p\mapsto \xi\cdot\tilde A^{(p)}_T\xi$, corresponding to the ``reference'' coefficients $\tilde A:=A^{(p_0)}$ and the ``perturbed'' coefficients $A'$. The cluster formula~\eqref{eq:rewriteSkdisj3} reads in that case
\begin{align}\label{eq:defdeltatp0m}
\tilde \Delta_T^{(p_0),j}:=\sum_{|F|=j}\E\bigg[\sum_{G\subset F}(-1)^{|F\setminus G|}\xi\cdot A^{E^{(p_0)}\cup G}(\nabla\phi_T^{E^{(p_0)}\cup G}+\xi)\bigg],
\end{align}
where the sum is absolutely convergent. Now note that the argument of the expectation vanishes whenever $F\cap E^{(p_0)}\ne\varnothing$, while, otherwise, if $F\cap E^{(p_0)}=\varnothing$, the argument of the expectation equals
\[\sum_{G\subset F}(-1)^{|F\setminus G|}\xi\cdot A^{G\cup(E^{(p_0)}\setminus F)}(\nabla\phi_T^{G\cup(E^{(p_0)}\setminus F)}+\xi).\]
As this expression is obviously independent of the event $[F\cap E^{(p_0)}=\varnothing]$, we can then rewrite
\begin{align}
\tilde \Delta_T^{(p_0),j}&=\sum_{|F|=j}\E\bigg[\mathds1_{F\cap E^{(p_0)}=\varnothing}\sum_{G\subset F}(-1)^{|F\setminus G|}\xi\cdot A^{E^{(p_0)}\cup G}(\nabla\phi_T^{E^{(p_0)}\cup G}+\xi)\bigg]\nonumber\\
&=\sum_{|F|=j}\p[F\cap E^{(p_0)}=\varnothing]\E\bigg[\sum_{G\subset F}(-1)^{|F\setminus G|}\xi\cdot A^{E^{(p_0)}\cup G}(\nabla\phi_T^{E^{(p_0)}\cup G}+\xi)\bigg]\nonumber\\
&=(1-p_0)^j\Delta_T^{(p_0),k},\label{eq:equivdeltatildep}
\end{align}
where $\Delta_T^{(p_0),j}$ is defined as the argument of the limit~\eqref{eq:formderdisj2}. The expansion~\eqref{eq:taylortransldr} then becomes, for any $k\ge1$, for all $0\le p\le1-p_0$,
\begin{align}\label{eq:taylortransldr1}
\bigg|\xi\cdot A^{(p_0+p)}_T\xi-\xi\cdot A_T^{(p_0)}\xi-\sum_{j=1}^k{p^j}\Delta_T^{(p_0),j}\bigg|\le\bigg(\frac {Cp}{1-p_0}\bigg)^{k+1}.
\end{align}
Moreover, recalling the bound $|\Delta_T^j|\le C^j$ for the right-derivatives at $0$, which is assumed to hold for any choice of the coefficient (the constant $C$ only depends on $R,\Gamma,d,\lambda$), we conclude, for all $p_0\in[0,1)$,
\begin{align}\label{eq:boundder1}
|\Delta_T^{(p_0),j}|\le C^j(1-p_0)^{-j}.
\end{align}
Note that this estimate for the derivatives deteriorates when $p_0$ gets closer to $1$. This difficulty is overcome by considering negative perturbations, that is, looking at left-derivatives, as we do now.

Let us now consider $p_0\in(0,1]$ and negative perturbations at that point. For $0\le p\le p_0$, choose a sequence $( d_n^{(p_0,-p)})_n$ of i.i.d. Bernoulli random variables (independent of all the others) with $\p[ d_n^{(p_0,-p)}=1]=p/p_0$, and consider the twice perturbed coefficients
\begin{align*}
A^{(p_0,-p)}&=A^{(p_0)}\mathds1_{\R^d\setminus J_n}+\sum_{n}\left(d_n^{(p_0,-p)}A+(1-d_n^{(p_0,-p)})A^{(p_0)}\right)\mathds1_{J_n}\\
&=A\mathds1_{\R^d\setminus\bigcup_nJ_n}+\sum_{n}\left((d_n^{(p_0,-p)}+(1-d_n^{(p_0,-p)})(1-b_n^{(p_0)}))A+b_n^{(p_0)}(1-d_n^{(p_0,-p)})A'\right)\mathds1_{J_n}.
\end{align*}
The field $A^{(p_0,p)}$ has by definition the same distribution as $A^{(p_0-p)}$, and it is a perturbation of $A^{(p_0)}$ with perturbation parameter $p/p_0$ (and with ``perturbed'' medium $A$, instead of $A'$). Applying to $A^{(p_0,p)}$ the result~\eqref{eq:resat0plop} around $0$ (which is assumed to hold), we deduce that the map $p\mapsto \xi\cdot A^{(p_0,p)}_T\xi=\xi\cdot A^{(p_0-p)}_T\xi$ satisfies, for any $k\ge1$, for all $0\le p\le p_0$,
\begin{align}\label{eq:taylortranslg}
\bigg|\xi\cdot A_T^{(p_0-p)}\xi-\xi\cdot A_T^{(p_0)}\xi-\sum_{j=1}^k\frac{p^j}{p_0^j}\hat\Delta_T^{(p_0),j}\bigg|\le\bigg(\frac {Cp}{p_0}\bigg)^{k+1},
\end{align}
where, for any $j\ge1$, $\hat\Delta_T^{(p_0),j}$ is the $j$-th right-derivative of the map $p\mapsto \hat A^{(p)}_T$, corresponding to the ``reference'' coefficients $\hat A:=A^{(p_0)}$ and the ``perturbed'' coefficients $A$. The cluster formula~\eqref{eq:rewriteSkdisj3} gives in this case
\begin{align*}
\hat \Delta_T^{(p_0),j}:=\sum_{|F|=j}\E\bigg[\sum_{G\subset F}(-1)^{|F\setminus G|}\xi\cdot A^{E^{(p_0)}\setminus G}(\nabla\phi_T^{E^{(p_0)}\setminus G}+\xi)\bigg],
\end{align*}
where the sum is absolutely convergent. Arguing as above, the argument of the expectation vanishes unless $F\subset E^{(p_0)}$; hence, by the independence assumption, we obtain
\begin{align*}
\hat \Delta_T^{(p_0),j}=p_0^j\sum_{|F|=j}\E\bigg[\sum_{G\subset F}(-1)^{|F\setminus G|}\xi\cdot A^{(E^{(p_0)}\setminus F)\cup(F\setminus G)}(\nabla\phi_T^{(E^{(p_0)}\setminus F)\cup(F\setminus G)}+\xi)\bigg],
\end{align*}
or equivalently, by the change of variables $F\setminus G\leadsto H$,
\begin{align*}
\hat \Delta_T^{(p_0),j}=(-1)^jp_0^j\sum_{|F|=j}\E\bigg[\sum_{H\subset F}(-1)^{|F\setminus H|}\xi\cdot A^{H\cup (E^{(p_0)}\setminus F)}(\nabla\phi_T^{H\cup (E^{(p_0)}\setminus F)}+\xi)\bigg]=(-1)^jp_0^j\Delta_T^{(p_0),j},
\end{align*}
where, as before, $\Delta_T^{(p_0),k}$ is defined as the argument of the same limit~\eqref{eq:formderdisj2}. The expansion~\eqref{eq:taylortransldr} then becomes, for any $k\ge1$, for all $0\le p\le p_0$,
\begin{align}\label{eq:taylortransldr2}
\bigg|\xi\cdot A^{(p_0-p)}_T\xi-\xi\cdot A_T^{(p_0)}\xi-\sum_{j=1}^k{(-p)^j}\Delta_T^{(p_0),j}\bigg|\le\bigg(\frac {Cp}{p_0}\bigg)^{k+1}.
\end{align}
By the bounds $|\Delta_T^j|\le C^j$ for the right-derivatives at $0$, we conclude, for all $p_0\in[0,1)$,
\begin{align}\label{eq:boundder2}
|\Delta_T^{(p_0),j}|\le C^jp_0^{-j}.
\end{align}

Combining~\eqref{eq:taylortransldr1} and~\eqref{eq:taylortransldr2} then directly yields the desired result~\eqref{eq:devtaylorTfix}. Moreover, combining~\eqref{eq:boundder1} and~\eqref{eq:boundder2} gives, for any $j\ge1$, the uniform bound
\[|\Delta_T^{(p_0),j}|\le  \min\{C^jp_0^{-j},C^j(1-p_0)^{-j}\}\le (2C)^j.\]
Finally, arguing as in Lemma~\ref{lem:alter} (where the argument is performed at $p_0=0$ and proves the equivalence between formulas~\eqref{eq:rewriteSkdisj1}, \eqref{eq:rewriteSkdisj2} and~\eqref{eq:rewriteSkdisj3}), we see that, for fixed $T,\theta$, the cluster formula for $\Delta_T^{(p_0),j}$, that is the argument of the limit~\eqref{eq:formderdisj2}, is equivalent to the formulas given by the argument of the limits~\eqref{eq:formderdisj0} and~\eqref{eq:formderdisj1}.

\medskip

\step{4} Conclusion.\\
Let $T>0$ be fixed, and assume that $\E[\rho(Q)^s]<\infty$ for all $s\ge1$. For any $k\ge1$, Corollary~\ref{cor:reslimT} exactly asserts~\eqref{eq:resat0plop}, Lemma~\ref{lem:alter} ensures that the $\Delta_T^j$'s are equivalently given by formulas~\eqref{eq:rewriteSkdisj1}, \eqref{eq:rewriteSkdisj2} and~\eqref{eq:rewriteSkdisj3}, and Proposition~\ref{prop:apriori} gives the uniform bounds $|\Delta_T^j|\le C^j$, for all $j\ge1$. By the previous steps, this proves Theorem~\ref{th:analytic} and Corollary~\ref{cor:analytic}.

\medskip

\step5 Exact formula for the first derivative.\\
In this last step, we further assume that $\rho(Q)\le\theta_0$ a.s. (so that in particular all the moments are bounded, and we can thus everywhere omit the regularization in $\theta$), and we prove under that assumption the validity of formula~\eqref{eq:der1form} in Corollary~\ref{cor:analytic}. More precisely, we need to prove that we can pass to the limit in $T$ inside the formula for the first approximate derivative
\[\Delta_T^{(p_0),1}=\sum_n\E\bigg[(\nabla\phi_T^{E^{(p_0)}\setminus\{n\}}+\xi)\cdot C^{\{n\}}(\nabla\phi_T^{\{n\}\cup E^{(p_0)}}+\xi)\bigg],\]
i.e. we prove that the well-defined limit $\Delta^{(p_0),1}=\lim_T\Delta^{(p_0),1}_{T}$ is given by the following formula:
\begin{align}\label{eq:der1formex}
\Delta^{(p_0),1}=\sum_n\E\bigg[(\nabla\phi^{E^{(p_0)}\setminus\{n\}}+\xi)\cdot C^{\{n\}}(\nabla\phi^{\{n\}\cup E^{(p_0)}}+\xi)\bigg],
\end{align}
with an abolutely converging sum. As before, we can restrict to $p_0=0$, and shall prove that the limit $\Delta^1:=\lim_T\Delta_{T}^1$ exists and is given by
\[\Delta^1=\sum_n\E\bigg[(\nabla\phi+\xi)\cdot C^{\{n\}}(\nabla\phi^{\{n\}}+\xi)\bigg].\]

We start by showing that the sum is absolutely convergent. Decomposing $\nabla\phi^{\{n\}}=\nabla\delta^{\{n\}}\phi+\nabla\phi$, using assumption~\eqref{eq:boundrho} in the form $\sum_{n}\mathds1_{J_n}\lesssim1$ (see also~\eqref{eq:firstboundsumchi+}), and recalling the elementary energy estimate $\E[1+|\nabla\phi|^2]\lesssim1$ (see~\eqref{eq:modif-corr-estim}), we have
\begin{align*}
|\Delta^{1}|&\lesssim \sum_n\E[\mathds1_{J_n}(1+|\nabla\phi|^2+|\nabla\phi^{\{n\}}|^2)]\lesssim1+\sum_n\E[|\nabla\delta^{\{n\}}\phi|^2],
\end{align*}
where the last sum is finite by Lemma~\ref{lem:aprioriproba} (for $k=1$) and the fact that $\nabla\delta^{\{n\}}\phi_T\cvf{}\nabla\delta^{\{n\}}\phi$ weakly in $\Ld^2_\loc(\R^d;\Ld^2(\Omega))$.

We now prove that $\lim_T\Delta_{T}^1=\Delta^1$. Given $L\sim1$, the additional assumption $\rho(Q)\le\theta_0$ a.s. implies by stationarity $\rho(B_{R+L})\le C\theta_0=:Z$, with $C\sim1$. Hence, we can choose the measurable enumeration $(q_n)_n$ of the point process $\rho$ in such a way that $B_{R+L}\cap(q_n)_n\subset (q_n)_{n=1}^{Z}$. Defining
\[a_L^n:=\E\bigg[\fint_{B_L}(\nabla\phi+\xi)\cdot C^{\{n\}}(\nabla\phi^{\{n\}}+\xi)\bigg],\quad\text{and}\quad a_{T,L}^n:=\E\bigg[\fint_{B_L}(\nabla\phi_{T}+\xi)\cdot C^{\{n\}}(\nabla\phi_{T}^{\{n\}}+\xi)\bigg],\]
we observe $\Delta^1=\sum_{n=1}^Z a_L^n$ and $\Delta_{T}^1=\sum_{n=1}^Z a_{T,L}^n$.
Indeed, by stationarity (together with absolute convergence), $\Delta^1=\sum_{n=1}^\infty a_L^n$, so that $\Delta^1=\sum_{n=1}^Z a_L^n$ by the choice of the measurable enumeration, and likewise for $\Delta_T^1$. Therefore, it is enough to prove $\lim_{T}a^n_{T,L}=a_L^n$ for any $1\le n\le \Gamma$. Since $\nabla\phi_T^{\{n\}}\cvf{}\nabla\phi^{\{n\}}$ weakly and $\nabla\phi_T\to\nabla\phi$ strongly in $\Ld^2_\loc(\R^d,\Ld^2(\Omega))$ (see \cite[Theorem~1]{Gloria-12}), we directly get $a_{T,L}^n\to a_L^n$ as $T\uparrow\infty$, for any $n$, as desired.

\subsection{Proof of Corollaries~\ref{cor:invpr}, \ref{cor:cm} and~\ref{cor:cm2}: Clausius-Mossotti formulas}

In this section we further assume that $\E[\rho(Q)^2]<\infty$. (Note that, in the case of Corollaries~\ref{cor:cm} and~\ref{cor:cm2}, this directly follows from assumption~\eqref{eq:boundrho} together with the fact that we are then dealing with ball inclusions of fixed radius.)

\subsubsection{First-order universality principle}\label{chap:invpr}
Set $J^{(p)}:=\bigcup_{n\in E^{(p)}}J_n$. The volume fraction $v_p$ of the perturbation is defined as follows:
\[v_p:=\lim_{L\uparrow\infty}\frac{\E[|LQ\cap J^{(p)}|]}{L^d},\]
or equivalently, by stationarity of the inclusion process,
\[v_p=\lim_{L\uparrow\infty}L^{-d}\sum_{z\in LQ\cap \Z^d}{\E[|(z+Q)\cap J^{(p)}|]}=\E[|Q\cap J^{(p)}|].\]
An inclusion-exclusion argument gives
\begin{align*}
\sum_{n}\E[\mathds1_{n\in E^{(p)}}|Q\cap J_n|]-\sum_{n\ne m}\E[\mathds1_{n,m\in E^{(p)}}|Q\cap J_n\cap J_m|]\le v_p\le \sum_{n}\E[\mathds1_{n\in E^{(p)}}|Q\cap J_n|].
\end{align*}
By the independence between the Bernoulli process $E^{(p)}$ and all the other random variables,
this turns into
\begin{align*}
p\sum_{n}\E[|Q\cap J_n|]-p^2\sum_{n\ne m}\E[|Q\cap J_n\cap J_m|]\le v_p\le p\sum_{n}\E[|Q\cap J_n|].
\end{align*}
As $J_n\subset B_R(q_n)$ for all $n$, we note that, by the assumption $\E[\rho(Q)^2]<\infty$,
\[\sum_{n\ne m}\E[|Q\cap J_n\cap J_m|]\le \E\bigg[\sum_{n\ne m}\mathds1_{q_n,q_m\in Q+B_R}\bigg]=\E[\rho(Q+B_R)^2]<\infty,\]
so that we have indeed proven
\[v_p=p\sum_{n}\E[|Q\cap J_n|]+O(p^2)=:p\gamma+O(p^2).\]
If $\gamma=0$, then $\cup_n J_n=\varnothing$ so that $A^{(p)}_\ho=A_\ho$ and $v_p=0$, and there is nothing left to prove.
If $\gamma \neq 0$, since
\[\gamma\le\E\bigg[\sum_n\mathds1_{q_n\in Q+B_R}\bigg]=\E[\rho(Q+B_R)]<\infty,\] we have  $v_p\sim_\gamma p+O(p^2)$. In particular, the expansion in $p$ in Theorem~\ref{th:analytic} at first order can as well be rewritten as an expansion in $v_p$: at first order at $p_0=0$, we have, for any $p\ge0$,
\[A^{(p)}_{\hom}=A_{\hom}+Kv_p+O(v_p^2),\]
where $K$ is given by
\[\xi\cdot K\xi=\frac1{\gamma}\xi\cdot A_{\hom}^{(0),1}\xi=\frac1{\gamma}\E\left[\sum_n(\nabla\phi+\xi)\cdot C^{\{n\}}(\nabla\phi^{\{n\}}+\xi)\right].\]

If in addition the random volumes $|J_n^\circ|$'s are i.i.d. and independent of the point process $\rho$ (and of its enumeration), then $\gamma$ can be computed more explicitly. By stationarity of the inclusion process, for any $L>0$,
\[\gamma=\lim_{L\uparrow\infty}L^{-d}\E\bigg[\sum_n|LQ\cap J_n|\bigg],\]
where we can estimate
\begin{align}\label{eq:bornesupgam}
\E\bigg[\sum_n|LQ\cap J_n|\bigg]\le\sum_n\E[\mathds1_{q_n\in LQ}|J_n|]=\E[|J_0^\circ|]\E[\rho(LQ)],
\end{align}
and also
\begin{align}\label{eq:borneinfgam}
\E\bigg[\sum_n|LQ\cap J_n|\bigg]\ge\sum_n\E[\mathds1_{q_n\in (L-R)Q}|J_n|]=\E[|J_0^\circ|]\E[\rho((L-R)Q)].
\end{align}
Now, for all continuous and integrable functions $f:\R^d\to \R$, we have $\E[\sum_nf(q_n)]=\E[\int fd\rho]=\int fd\E[\rho]$. Since $\rho$ is stationary, the Borel measure $\E[\rho]$ is translation-invariant, and hence, since it is locally finite by definition, it is a multiple of the Lebesgue measure: $\E[\rho]=\sigma dx$ for some constant $\sigma\in\R^+$, which is characterized e.g. by $\sigma=\E[\rho(Q)]$.
In these terms, \eqref{eq:bornesupgam} and~\eqref{eq:borneinfgam} give
\[\sigma\E[|J_0^\circ|]=\E[|J_0^\circ|]\lim_L L^{-d}\E[\rho((L-R)Q)]\le\gamma\le\E[|J_0^\circ|]\lim_L L^{-d}\E[\rho(LQ)]=\sigma\E[|J_0^\circ|],\]
which means $\gamma=\sigma\E[|J_0^\circ|]$, and thus
\begin{align}\label{eq:calculvp}
v_p=p\sigma\E[|J_0^\circ|].
\end{align}
The matrix $K$ then takes the form
\[\xi\cdot K\xi=\frac1{\sigma\E[|J^\circ_0|]}\E\bigg[\sum_n(\nabla\phi+\xi)\cdot C^{\{n\}}(\nabla\phi^{\{n\}}+\xi)\bigg].\]
Further assuming that $\rho$ is independent of $A$, of $(A_n')_n$ (as well as of the random volumes $|J_n^\circ|$'s), we note that the random variable $\E[(\nabla\phi(0)+\xi)\cdot C^{\{n\}}(0)(\nabla\phi^{\{n\}}(0)+\xi)\|\rho]$ only depends on the point process $\rho$ through the point $q_n$, so that it can be written as $f(q_n)$ for some measurable function $f$. In these terms, we get
\[\xi\cdot K\xi=\frac1{\sigma\E[|J^\circ_0|]}\E\bigg[\sum_nf(q_n)\bigg]=\frac1{\E[J^\circ_0|]}\int_{\R^d} f(x)dx,\]
which does clearly no longer depend on the choice of the point process $\rho$. This proves Corollary~\ref{cor:invpr}.
\qed

\subsubsection{Electric Clausius-Mossotti formula}

We consider the case when the inclusions are spherical $J_n=B_R(q_n)$, and the unperturbed and perturbed coefficients have the form $A=\alpha\Id$ and $A'=\beta\Id$ respectively. We shall compute explicitly the first derivative $\xi\cdot A^{(0),1}_{\hom}\xi$ of the perturbed homogenized coefficient at $0$, as given by formula~\eqref{eq:der1form}. As inclusions are balls of fixed radius $R$, assumption~\eqref{eq:boundrho} implies $\rho(Q)\le\theta_0$ a.s. for some constant $\theta_0>0$, so that we can indeed apply formula~\eqref{eq:der1form}.

Since $A$ is constant, the unique gradient solution of $-\nabla\cdot A(\nabla\phi_{\xi}+\xi)=0$ is clearly $\nabla\phi_{\xi}=0$. Let now $n$ be fixed. The solution $\phi^{\{n\}}_{\xi}\in H^1(\R^d)$ of $-\nabla\cdot A^{\{n\}}(\nabla\phi_{\xi}^{\{n\}}+\xi)=0$ is easily checked to be unique if it exists, and its existence follows from a direct computation. Since $A^{\{n\}}$ is constant both inside and outside $B_R(q_n)$, the solution $\phi^{\{n\}}_\xi$ is radial and of the form
\[\phi^{\{n\}}_\xi(x)=\psi_{\xi}(x-q_n)=\begin{cases}
C(x-q_n)\cdot \xi,&\text{for $|x-q_n|<R$};\\
C'\frac{(x-q_n)\cdot\xi}{|x-q_n|^{d}},&\text{for $|x-q_n|>R$;}
\end{cases}\]
so that its gradient satisfies
\[\nabla\phi_\xi^{\{n\}}(x)=\nabla\psi_{\xi}(x-q_n)=\begin{cases}
C\xi,&\text{for $|x-q_n|<R$};\\
\frac{C'}{|x-q_n|^d}\left(\xi-d\frac{(x-q_n)\cdot\xi}{|x-q_n|}\frac{x-q_n}{|x-q_n|}\right)
,&\text{for $|x-q_n|>R$.}
\end{cases}\]
Since $\phi^{\{n\}}_\xi$ is radial and in $H^1(\R^d)$, it is continuous, which implies that $C'=CR^{d}$.
The normal component of $A^{\{n\}}(\nabla\phi^{\{n\}}_\xi+\xi)$ must also be continuous through the sphere, so that we conclude
\begin{align}\label{eq:defCcst}
C=\frac{\alpha-\beta}{\beta+\alpha(d-1)}.
\end{align}
This allows us to turn~\eqref{eq:der1form} into an explicit formula for the first derivative $A^{(0),1}_{\hom}$:
\begin{align*}
\xi\cdot A^{(0),1}_{\hom}\xi&=\sum_n\E[(\nabla\phi+\xi)\cdot C^{\{n\}}(\nabla\phi^{\{n\}}+\xi)]\\
&=\sum_n\E[\xi\cdot (\beta-\alpha)\mathds1_{B_R(q_n)}(C\xi+\xi)]=(1+C)(\beta-\alpha)\,\E\bigg[\sum_n\mathds1_{q_n\in B_R}\bigg].
\end{align*}
From Paragraph~\ref{chap:invpr} above, we learn that $\E[\sum_nf(q_n)]=\sigma\int f(x)dx$ for any continous and integrable function $f$. Hence,
\[\xi\cdot A^{(0),1}_{\hom}\xi=(1+C)(\beta-\alpha)\,\E\bigg[\sum_n\mathds1_{B_R}(q_n)\bigg]=\sigma|B_R|(1+C)(\beta-\alpha),\]
so that expression~\eqref{eq:defCcst} for $C$ yields
\[\xi\cdot A^{(0),1}_{\hom}\xi=\sigma|B_R|\frac{\alpha d(\beta-\alpha)}{\beta+\alpha(d-1)}.\]
In the present case, formula~\eqref{eq:calculvp} holds true and gives $v_p=p\sigma|B_R|$.
The conclusion of Corollary~\ref{cor:cm} now follows from Theorem~\ref{th:analytic}.\qed

\subsubsection{Elastic Clausius-Mossotti formula}
We consider the case of spherical inclusions $J_n=B_R(q_n)$ and assume that both the unperturbed stiffness tensor $A$ and the perturbed stiffness tensor $A'$ are constant and isotropic --- we denote by $K,G$ and $K',G'$ their respective bulk and shear moduli. We shall compute explicitly in that case the first derivative $\xi\cdot A^{(0),1}_{\hom}\xi$ of the perturbed homogenized stiffness tensor $A^{(p)}_{\hom}$ at $0$, as given by formula~\eqref{eq:der1form}. Indeed, as inclusions are balls of fixed radius $R$, assumption~\eqref{eq:boundrho} implies $\rho(Q)\le\theta_0$ a.s. for some constant $\theta_0>0$, so that we can apply formula~\eqref{eq:der1form}.

Let $\xi\in\R^{d\times d}$ be symmetric. Since $A$ is constant, the unique gradient solution of $-\nabla\cdot A:(\nabla\phi_\xi+\xi)=0$ is clearly $\nabla\phi_\xi=0$. Let now $n$ be fixed. As shown e.g. in Section~17.2.1 of~\cite{Torquato-02}, equation $-\nabla\cdot A^{\{n\}}:(\nabla\phi_\xi^{\{n\}}+\xi)$ admits a (necessarily unique) solution in $H^1(\R^d)$. Inside the inclusion $B_R(q_n)$, that is, for all $|x-q_n|<R$ (see equation~(17.84) of~\cite{Torquato-02}),
\begin{align}\label{eq:soltorq}
\nabla\phi_\xi^{\{n\}}(x)+\xi=\Id\frac{\Tr\xi}d\frac{K+\beta}{K'+\beta}+\left(\xi-\Id\frac{\Tr\xi}d\right)\frac{G+\alpha}{G'+\alpha},
\end{align}
where $\alpha,\beta$ are defined by~\eqref{eq:defab}. Recalling that $\xi:A:\chi=2G\xi:\chi+\lambda\Tr\xi\Tr\chi$ for any symmetric $\chi\in\R^{d\times d}$, we can now explicitly compute formula~\eqref{eq:der1form} for the first derivative $A^{(0),1}_{\hom}$,
\begin{align*}
\frac12\xi: A^{(0),1}_{\hom}:\xi&=\frac12\sum_n\E\left[\mathds1_{B_R(q_n)}\xi: (A'-A):(\nabla\phi^{\{n\}}_\xi+\xi)\right]\\
&=\sum_n\E\left[\mathds1_{B_R(q_n)}\left((G'-G)\xi:(\nabla\phi^{\{n\}}_\xi+\xi)+\frac12(\lambda'-\lambda)\Tr\xi\Tr(\nabla\phi^{\{n\}}_\xi+\xi)\right)\right],
\end{align*}
and hence, using~\eqref{eq:soltorq}, and recalling from Paragraph~\ref{chap:invpr} that $\E[\sum_n\mathds1_{B_R(q_n)}]=\sigma|B_R|$,
\begin{align*}
\frac1{2\sigma|B_R|}\xi: A^{(0),1}_{\hom}:\xi&=\frac1d(\Tr\xi)^2\frac{K+\beta}{K'+\beta}\left((G'-G)+\frac d2(\lambda'-\lambda)\right)\\
&\hspace{2cm}+\left(|\xi|^2-\frac1d(\Tr\xi)^2\right)(G'-G)\frac{G+\alpha}{G'+\alpha}.
\end{align*}
In terms of bulk moduli $K=\lambda+2G/d$ and $K'=\lambda'+2G'/d$, this takes the form
\begin{align}\label{eq:devderelast}
\frac1{2\sigma|B_R|}\xi: A^{(0),1}_{\hom}:\xi&=\frac12(\Tr\xi)^2(K'-K)\frac{K+\beta}{K'+\beta}+\left(|\xi|^2-\frac1d(\Tr\xi)^2\right)(G'-G)\frac{G+\alpha}{G'+\alpha}.
\end{align}
As formula~\eqref{eq:calculvp} again holds true in the present case and gives $v_p=p\sigma|B_R|$, Corollary~\ref{cor:cm2} now follows by Theorem~\ref{th:analytic}.
\qed

\subsection{Proof of Corollary~\ref{cor:rates}: convergence rates}
Let $p_0\in[0,1]$ be fixed, and assume that $\E[\rho(Q)^s]<\infty$ for all $s\ge1$. Estimate~\eqref{eq:approxok} in the proof of Theorem~\ref{th:analytic} (see Section~\ref{chap:proofth1}) then yields, for all $k\ge1$,
\begin{align*}
\bigg|\sum_{j=1}^kp^j(\Delta_T^{(p_0),j}-\Delta_{2T}^{(p_0),j})\bigg|&\le (Cp)^{k+1}+|A_T^{(p_0+p)}-A_{2T}^{(p_0+p)}|+|A_T^{(p_0)}-A_{2T}^{(p_0)}|\\
&\lesssim (Cp)^{k+1}+\E[|\nabla(\phi_T^{(p_0+p)}-\phi_{2T}^{(p_0+p)})|]+\E[|\nabla(\phi_T^{(p_0)}-\phi_{2T}^{(p_0)})|],
\end{align*}
and hence, combining this with assumption~\eqref{eq:quantas},
\begin{align}\label{eq:respresqueokplopda}
\bigg|\sum_{j=1}^kp^j(\Delta_T^{(p_0),j}-\Delta_{2T}^{(p_0),j})\bigg|&\le(Cp)^{k+1}+C\gamma(T).
\end{align}
By induction, we easily see that this implies, for all $j\ge1$,
\begin{align}\label{eq:toshowrecda}
|\Delta_T^{(p_0),j}-\Delta_{2T}^{(p_0),j}|&\le (2C)^{j+1}\gamma(T)^{2^{-j}}.
\end{align}
Estimate~\eqref{eq:respresqueokplopda} with $k=1$ gives $|\Delta_T^{(p_0),j}-\Delta_{2T}^{(p_0),j}|\le C^2p+C\gamma(T)/p$, which turns into \eqref{eq:toshowrecda} for $j=1$ with the choice $p=\gamma(T)^{\frac12}$. Assume now that the result~\eqref{eq:toshowrecda} is proven for all $0\le j\le J$. Then, equation~\eqref{eq:respresqueokplopda} for $k=J+1$ gives
\begin{align*}
|\Delta_T^{(p_0),J+1}-\Delta_{2T}^{(p_0),J+1}|&\le C^{J+2}p+Cp^{-J-1}\gamma(T)+\sum_{j=1}^J(2C)^jp^{j-J-1}\gamma(T)^{2^{-j}}.
\end{align*}
With the choice $p=\gamma(T)^{2^{-J-1}}$, and noting that $(l+1)2^{-l}\le1$ for any $l\in\N$,
this turns into
\begin{align*}
|\Delta_T^{(p_0),J+1}-\Delta_{2T}^{(p_0),J+1}|&\le C^{J+2}\gamma(T)^{2^{-J-1}}+C\gamma(T)^{1-(J+1)2^{-J-1}}\\
&\hspace{1cm}+\sum_{j=1}^J(2C)^{j+1}\gamma(T)^{2^{-j}(1-(J+1-j)2^{-(J+1-j)})}\\
&\le C^{J+2}\gamma(T)^{2^{-J-1}}+C\gamma(T)^{2^{-1}}+\sum_{j=1}^J(2C)^{j+1}\gamma(T)^{2^{-j-1}}\\
&\le C^{J+2}\gamma(T)^{2^{-J-1}}\left(2+\sum_{j=1}^J2^{j+1}\right)\le(2C)^{J+2}\gamma(T)^{2^{-J-1}},
\end{align*}
which proves~\eqref{eq:toshowrecda} by induction, and concludes the proof of Corollary~\ref{cor:rates}.\qed

\section*{Acknowledgements}
The authors acknowledge financial support from the European Research Council under
the European Community's Seventh Framework Programme (FP7/2014-2019 Grant Agreement
QUANTHOM 335410).

\def\cprime{$'$} \def\cprime{$'$}

\end{document}